\documentclass[10pt,a4paper,french]{smfart}

\usepackage[latin1]{inputenc}
\usepackage[T1]{fontenc}
\usepackage[french, english]{babel}
\usepackage{amsopn}                    
\usepackage{amsmath}
\usepackage{amssymb}
\usepackage[all,2cell,emtex]{xy}
\UseAllTwocells
\xyoption{2cell}
\input xypic

\usepackage{xpatch}
\usepackage[hidelinks,pdfencoding=unicode]{hyperref}
\usepackage{lmodern}

\makeatletter

\patchcmd\@begintheorem
  {\ignorespaces}
  {\hypertarget{\csname @currentHref\endcsname}{}\ignorespaces}
  {}{}

\patchcmd\@maketitlehook
  {\ifx\@empty\@keywords\else\@footnotetext{\@setkeywords}\fi}
  {\ifx\@empty\@keywords\else\@footnotetext{\@setkeywords}\fi
   \ifx\@empty\@altkeywords\else\@footnotetext{\@setaltkeywords}\fi}
  {}{}

\makeatother

\SilentMatrices

\SelectTips{cm}{10}

\makeatletter

\def\@maketitlehook{%
  \ifx\@empty\@subjclass\else\@footnotetext{\@setsubjclass}\fi
  \ifx\@empty\@keywords\else\@footnotetext{\@setkeywords}\fi
  \ifx\@empty\@altkeywords\else\@footnotetext{\@setaltkeywords}\fi
  \ifx\@empty\thankses\else\def\par{\let\par\@par}\@footnotetext{\@setthanks}\fi
 }

\makeatother

\renewcommand\labelitemii\labelitemi

\newcommand{\gmrelax}[1]{\relax}

\newcommand{\gmvieux}[1]{\relax}
\newcommand{\nbd}{\nobreakdash}

\renewcommand{\leq}{\leqslant}
\renewcommand{\geq}{\geqslant}

\theoremstyle{plain}
\newtheorem{thm}{Th\'eor\`eme}[section]
\newtheorem{prop}[thm]{Proposition}
\newtheorem{lemme}[thm]{Lemme}

\theoremstyle{remark}                                             
\newtheorem{rem}[thm]{Remarque}
\newtheorem{notation}[thm]{Notations}

\theoremstyle{definition}                                         

\newtheorem{paragr}[thm]{}
\theoremstyle{plain} 
\swapnumbers

\newtheorem{lembsmpl}[thm]{Lemme bisimplicial}
\swapnumbers

\numberwithin{equation}{thm}

\newcommand{\sHom}{\operatorname{\kern.5truept\underline{\kern-.5truept\mathsf{Hom}\kern-1truept}\kern1.5truept}}
\newcommand{\Hom}{\operatorname{\mathsf{Hom}}}

\newcommand{\pref}[1]{{\widehat{ #1 }}}

\newcommand{\card}{\mathsf{card}}

\newcommand{\sauf}{\mathchoice{\raise 1.8pt\hbox{${\scriptstyle\kern
2.5pt\smallsetminus\kern 2.5pt}$}}{\raise 1.8pt\hbox{${\scriptstyle\kern
2.5pt\smallsetminus\kern 2.5pt}$}}{\raise
1.8pt\hbox{${\scriptscriptstyle\kern 1.5pt\smallsetminus\kern
1.5pt}$}}{\raise 1.8pt\hbox{${\scriptscriptstyle\kern
1.5pt\smallsetminus\kern 1.5pt}$}}}

\newcommand{\toto}{{\hskip -2.5pt\xymatrixcolsep{1.3pc}\xymatrix{\ar[r]&}\hskip -2.5pt}}

\newcommand{\todouble}{\xymatrixcolsep{1pc}\xymatrix{\ar@<.5ex>[r]\ar@<-.5ex>[r]&}}
\newcommand{\todoubleop}{\xymatrixcolsep{1pc}\xymatrix{\ar@<.5ex>[r]&\ar@<.5ex>[l]}}
\renewcommand{\hookrightarrow}{{\hskip -1.5pt\raise 1.5pt\vbox{\xymatrixcolsep{.9pc}\xymatrix{\ar@{^{(}->}[r]&}}\hskip -3.5pt}}
\renewcommand{\longmapsto}{{\hskip -2.5pt\xymatrixcolsep{1.3pc}\xymatrix{\ar@{|->}[r]&}\hskip -2.5pt}}

\newcommand{\cm}[2]{\mathchoice {#1\raise -1.8pt\vbox{\hbox{$\kern -.8pt/#2$}}} {#1\raise -1.8pt\vbox{\hbox{$\kern -.8pt/#2$}}\kern .8pt} {#1\raise -1.8pt\vbox{\hbox{$\scriptstyle\kern -.8pt /#2$}}} {#1\raise -1.8pt\vbox{\hbox{$\scriptscriptstyle\kern -.8pt /#2$}}} }

\newcommand{\mc}[2]{\mathchoice {\raise -1.8pt\vbox{\hbox{$#1\backslash$}}#2} {\raise -1.8pt\vbox{\hbox{$#1\backslash$}}#2} {\raise -1.8pt\vbox{\hbox{$\scriptstyle#1\backslash$}}#2} {\raise -1.8pt\vbox{\hbox{$\scriptscriptstyle#1\backslash$}}#2} }

\newcommand\oo{$\infty$\nbd}
\newcommand\pdfoo{\texorpdfstring{$\infty$}{\unichar{"221E}}}
\newcommand\ndef\emph

\newcommand{\Homi}{\operatorname{\kern.5truept\underline{\kern-.5truept\mathsf{Hom}\kern-.5truept}\kern1truept}}

\newcommand\op\circ
\newcommand\var\bullet

\newcommand{\ooCat}{\nCat{\infty}}
\newcommand{\nCat}[1]{{#1}\hbox{\protect\nbd-}\kern1pt\Cat}
\newcommand{\Cat}{{\mathcal{C}\mspace{-2.mu}\it{at}}}

\renewcommand\le\leqslant
\renewcommand\ge\geqslant
\newcommand\comp\ast

\newcommand\joint\star
\newcommand{\Cda}{\mathcal{C}_{\mathrm{da}}}

\newcommand\On[1]{\mathcal{O}_{#1}}
\newcommand\Deltan[1]{\varDelta_{#1}}
\newcommand\cDelta{\mathbf{\Delta}}
\newcommand\EnsSimp{\pref{\cDelta}}
\newcommand\Z{\mathbb{Z}}
\newcommand{\tb}[1]{\tau_{\le #1}^{\mathrm b}}

\newcommand\zbox[1]{\makebox[0pt][l]{#1}}
\newcommand\pbox[1]{\zbox{\quad#1}}

\newcommand{\tr}[2]{\mathchoice
  {#1\raise -1.8pt\vbox{\hbox{$\kern -.8pt/#2$}}}
  {#1\raise -1.8pt\vbox{\hbox{$\kern -.8pt/#2$}}\kern .8pt}
  {#1\raise -1.8pt\vbox{\hbox{$\scriptstyle\kern -.8pt /#2$}}}
  {#1\raise -1.8pt\vbox{\hbox{$\scriptscriptstyle\kern -.8pt /#2$}}}}

\newcommand{\trm}[2]{\mathchoice
  {#1\raise -1.8pt\vbox{\hbox{$\kern -.8pt\!\stackrel{\,\rm co}{/}\!\!#2$}}}
  {#1\raise -1.8pt\vbox{\hbox{$\kern -.8pt\!\stackrel{\,\rm co}{/}\!\!#2$}}\kern .8pt}
  {TODO}
  {TODO}}

\newcommand{\cotr}[2]{\mathchoice
  {\raise -1.8pt\vbox{\hbox{$#2\backslash$}}#1}
  {\raise -1.8pt\vbox{\hbox{$#2\backslash$}}#1}
  {\raise -1.8pt\vbox{\hbox{$\scriptstyle#2\backslash$}}#1}
  {\raise -1.8pt\vbox{\hbox{$\scriptscriptstyle#2\backslash$}}#1}}

\newcommand{\cotrm}[2]{\mathchoice
  {\raise -1.8pt\vbox{\hbox{$#2\!\stackrel{\!\rm co}{\backslash}$}}#1}
  {\raise -1.8pt\vbox{\hbox{$#2\!\stackrel{\!\rm co}{\backslash}$}}#1}
  {TODO}
  {TODO}}

\newcommand{\atom}[1]{\langle{#1}\rangle}
\newcommand{\tabld}[2]{\begin{pmatrix}#1^0_0 &\dots &#1^0_{#2-1}
  &#1^0_{#2}\cr\noalign{\vskip 3pt} #1^1_0 &\dots &#1^1_{#2-1}
  &#1^1_{#2}\end{pmatrix}}
\newcommand{\tablnu}[3]{\begin{pmatrix}#1^0_0 &\dots &#1^0_{#2}
  &#3\cr\noalign{\vskip 3pt} #1^1_0 &\dots &#1^1_{#2} &#3\end{pmatrix}}
\newcommand{\tabll}[2]{\begin{pmatrix}#1^0_0 &#1^0_1 &\dots &#1^0_{#2-1}
  &#1^0_{#2}\cr\noalign{\vskip 3pt} #1^1_0 &#1^1_1 &\dots &#1^1_{#2-1}
  &#1^1_{#2}\end{pmatrix}}

\newcommand{\cn}{\mathsf{c}}
\newcommand{\NGr}{N}              
\newcommand{\Nrf}{N}              
\newcommand{\Diag}{\mathsf{Diag}} 
\newcommand{\Dsk}{\Deltan{1}}     

\author{Dimitri Ara}
\address{Aix Marseille Univ, CNRS, Centrale Marseille, I2M, Marseille, France}
\email{dimitri.ara@univ-amu.fr}
\urladdr{http://www.i2m.univ-amu.fr/perso/dimitri.ara/}

\author{Georges Maltsiniotis}
\address{CNRS, Institut de Math\'ematiques de Jussieu\\
Universit\'e Paris 7 Diderot\\
Case Postale 7012\\
B\^atiment Sophie Germain\\
75205 Paris Cedex 13\\
France}
\email{georges.maltsiniotis\at imj-prg.fr}
\urladdr{http://webusers.imj-prg.fr/\raise -3.3pt\vbox{\hbox{$\widetilde{ \ }\,$}}georges.maltsiniotis/}

\title[Un théorème A de Quillen pour les $\infty$-catégories strictes I]{Un
théorème A de Quillen pour\\ les \raise .9pt\hbox{$\infty$-}catégories strictes I :\\ {la preuve simpliciale}}
\alttitle{A Quillen's Theorem A for strict $\infty$-categories I}

\begin{document}

\frontmatter

\begin{abstract}
Le but de cet article est de démontrer une généralisation pour les
\oo-catégories strictes du célèbre théorème~A de Quillen. Ce résultat est
central à la théorie de l'homotopie des \oo-catégories strictes développée
par les auteurs. La preuve exposée dans ce texte est de nature simpliciale
et s'appuie sur la théorie des complexes dirigés augmentés de Steiner. Dans
un deuxième article, on démontrera ce même résultat par des méthodes
purement \oo-catégoriques.
\end{abstract}

\begin{altabstract}
The aim of this paper is to prove a generalization of the famous Theorem~A
of Quillen for strict \oo-categories. This result is central to the homotopy
theory of strict \oo-categories developed by the authors. The proof
presented here is of a simplicial nature and uses Steiner's theory of
augmented directed complexes. In a subsequent paper, we will prove the same
result by purely \oo-categorical methods.
\end{altabstract}



\altkeywords{strict $\infty$-categories, augmented directed complexes,
simplicial sets, join, Street's nerve, orientals, Gray tensor
product, Theorem A, slices, oplax transformations}

\maketitle

\tableofcontents

\mainmatter

\section*{Introduction}

\subsection*{Catégories supérieures}

Depuis quelques années, les catégories supérieures et, plus
particulièrement, les $(\infty,1)$\nbd-catégories faibles (aujourd'hui
improprement appelées \oo-catégories) et les
$(\infty,\infty)$\nbd-catégories strictes, plus connues sous le nom de
$\omega$\nbd-catégories ou de \oo-catégories strictes, sont activement
étudiées. Rappelons brièvement ce que sont ces objets.

\smallbreak

Une \emph{\oo-catégorie stricte} $C$ est  la donnée d'un \emph{\oo-graphe}
\[
\xymatrixcolsep{1.5pc}
\xymatrix{
C_{0}
&C_{1} \ar@<.6ex>[l]^(.45){t}\ar@<-.9ex>[l]_(.45){s}  
&\quad\cdots\quad\ar@<.6ex>[l]^(.55){t}\ar@<-.9ex>[l]_(.55){s}
&C_{i-1}\ar@<.6ex>[l]^(.43){t}\ar@<-.9ex>[l]_(.43){s}
&C_{i}\ar@<.6ex>[l]^(.41){t}\ar@<-.9ex>[l]_(.41){s}
&\quad\cdots\ar@<.6ex>[l]^{t}\ar@<-.9ex>[l]_{s}
}\ ,\qquad ss=st\ ,\quad ts=tt\ ,
\]
où pour $i\geq0$, $C_i$ désigne l'ensemble des \emph{$i$\nbd-cellules} de
$C$, et $s$ et $t$ les applications \emph{source} et
\emph{but}, muni d'applications \emph{unité}
\[
C_i\to C_{i+1}\ ,\quad x\mapsto 1_x\ ,\quad i\geq0\ ,\quad x\in C_i\ ,
\]
et de \emph{compositions}
\[
C_i\times_{C_j}C_i\to C_i\ ,\quad (x,y) \mapsto x\comp_jy\ ,\quad0\leq j<i\ ,\quad x,y\in C_i\ ,
\]
associant à un couple $(x,y)$ de $i$\nbd-cellules \emph{$j$\nbd-composables}
(c'est-à-dire telles que la $j$\nbd-cellule but itéré de $y$ soit égale à la
$j$\nbd-cellule source itérée de $x$) une $i$\nbd-cellule $x\comp_jy$. On
demande que ces applications soient compatibles aux opérations source et but
en un sens adéquat et qu'elles vérifient des axiomes de nature équationnelle
: \emph{associativité}, \emph{unité}, \emph{loi d'échange} et
\emph{fonctorialité des unités}. Par
exemple, l'associativité impose que, pour $0\leq j<i$ et $x,y,z$ trois
$i$\nbd-cellules $j$\nbd-composables, on ait
\[
(x\comp_jy)\comp_jz=x\comp_j(y\comp_jz)\ .
\]
Pour $0\leq m\leq n\leq\infty$, une \emph{$(n,m)$\nbd-catégorie stricte} est une
\oo-catégorie stricte dont les $i$\nbd-cellules sont \emph{triviales}
(c'est-à-dire sont des unités) pour $i>n$ et \ndef{strictement inversibles} pour
$i>m$ pour les compositions $\comp_j$, $m\leq j<i$. Pour $n \ge 0$, une
\ndef{$n$\nbd-catégorie stricte} est une $(n,n)$\nbd-catégorie stricte. Par
exemple, une $0$\nbd-catégorie stricte est un ensemble, une
$1$\nbd-catégorie stricte une catégorie, une $(1,0)$\nbd-catégorie stricte
un groupoïde, une $(2,1)$\nbd-catégorie stricte une $2$\nbd-catégorie
stricte dont les $2$\nbd-cellules sont inversibles pour la composition
\emph{verticale} $\comp_1$, une $(\infty,0)$\nbd-catégorie stricte un
\oo-groupoïde strict, une $(\infty,\infty)$\nbd-catégorie stricte une
\oo-catégorie stricte, etc.

\smallbreak

Moralement, une \ndef{\oo-catégorie faible} consiste en un \oo-graphe, des
applications unité et des compositions qui vérifient les mêmes
compatibilités aux opérations source et but que dans le cas strict mais qui,
par contre, ne vérifient pas les axiomes d'associativité, d'unité, d'échange
et de fonctorialité des unités à égalité près (\emph{on the nose} dirait-on
en anglais), mais seulement à des \emph{cohérences} près. De plus, les relations
naturelles que ces cohérences devraient satisfaire ne sont satisfaites qu'à
des cohérences supérieures près, et ainsi de suite. On définit les
\ndef{$(n,m)$-catégories faibles}, $0\leq m\leq n\leq\infty$, comme dans le cas
strict mais en remplaçant les inverses stricts par des \emph{inverses
faibles}.

\smallbreak

Il est difficile de formaliser cette notion de \oo-catégorie faible, et il
n'y a pas une manière unique de le faire. Le premier à avoir trouvé un fil
d'Ariane pour définir une structure de ce type est Grothendieck.  Sa
motivation et son inspiration étaient d'origine topologique et homotopique.
Conscient que les \oo-groupoïdes stricts ne modélisent que des types
d'homotopie d'un genre très particulier (notamment, si on se restreint aux
\oo-groupoïdes simplement connexes, on n'obtient que des produits d'espaces
d'Eilenberg-Mac Lane, voir par exemple~\cite{AraTypHomStr}), il cherchait à
dégager une notion de \oo-groupoïde faible modélisant tous les types
d'homotopie. Dans~\cite{PS1}, il a défini une telle notion en termes
d'esquisses projectives, il a associé à tout espace topologique un
\oo-groupoïde faible fondamental, et il a conjecturé que cela induisait une
équivalence entre la catégorie homotopique des espaces et la catégorie
homotopique des \oo-groupoïdes faibles. Cette conjecture, connue aujourd'hui
sous le nom d'\emph{hypothèse d'homotopie}, n'est toujours pas démontrée.
Quelques progrès ont été réalisés récemment, notamment dans~\cite{Simon} où
cette conjecture a été ramenée à une conjecture technique plausible.

\smallbreak

Une quinzaine d'années après la définition de Grothendieck, Batanin a dégagé
un concept de \oo-catégorie faible en utilisant sa théorie des
opérades globulaires~\cite{Batanin}. Plus tard, le second des auteurs a
réalisé qu'une légère variante de la définition de Grothendieck permettait
d'obtenir une définition de \oo-catégorie faible~\cite{infty}, définition
que le premier des auteurs a prouvé être essentiellement équivalente à celle
de Batanin~\cite{AraThese}. Un des avantages de cette définition « à la
Grothendieck » est qu'elle peut être implémentée pratiquement telle quelle
en théorie de types~\cite{ErSam}.

\smallbreak

Malheureusement, les structures ainsi définies sont malaisées à étudier, très
peu de résultats sont connus, et on ne dispose de pratiquement aucun exemple
autre que le \oo-groupoïde fondamental d'un espace, ou plus généralement,
d'un objet d'une catégorie de modèles dont tous les objets sont
fibrants~\cite{PS1, AraGrFond}, ainsi que bien sûr les \oo-catégories
strictes. L'étude des \oo-catégories strictes, plus simple,
constitue ainsi une source précieuse d'inspiration pour les \oo-catégories
faibles.

\smallbreak

L'approche la plus courante pour contourner la difficulté de l'étude des
modèles \emph{algébriques et globulaires} des \oo-catégories faibles
est de prendre l'hypothèse d'homotopie comme \emph{définition} des
\oo-groupoïdes faibles, autrement dit, de décréter que les \oo-groupoïdes
faibles \emph{sont} les espaces et souvent, plus précisément, les complexes
de Kan. Cette approche a conduit à définir de nombreux modèles
\emph{homotopiques} pour les $(\infty,1)$\nbd-catégories faibles, comme les
quasi-catégories~\cite{Joyal, Joyal1, Joyal2, Lurie, CisQuasi}, les
catégories simpliciales~\cite{Bergner}, les catégories de
Segal~\cite{Carlos}, les espaces de Segal complets~\cite{Rezk1}, et plus
tard pour les $(\infty,m)$\nbd-catégories faibles, $m$ arbitraire mais
\emph{fini}, comme les $m$\nbd-catégories de Segal~\cite{Carlos}, les
$\Theta_m$\nbd-espaces~\cite{Rezk2, Rezk3} et les
$m$\nbd-quasi-catégories~\cite{AraQCat}. Le seul modèle homotopique
développé jusqu'à présent pour les $(\infty,\infty)$\nbd-catégories faibles
est celui des ensembles compliciaux faibles~\cite{Verity1, Verity2},
généralisation naturelle des ensembles compliciaux~\cite{Verity3}, dont la
catégorie est équivalente à celle des \oo-catégories strictes. Les
$(\infty,m)$\nbd-catégories faibles ont trouvé, pour $m=1$, de nombreuses
applications en géométrie algébrique dérivée~\cite{LurieDer, Toen} et en
théorie de topos supérieurs~\cite{Lurie}, pour $m=2$, en théorie géométrique
de Langlands~\cite{Gaitsgory}, et pour $m$ arbitraire, dans une approche en
vue de la démonstration de \emph{l'hypothèse du
cobordisme}~\cite{LurieCobord}.

\subsection*{Les catégories comme modèles des types d'homotopie}

Comme on l'a rappelé, les \oo-groupoïdes stricts ne suffisent pas à
modéliser les types d'homotopie et c'est une des motivations premières à
l'introduction des \oo-groupoïdes faibles. La situation est différente pour
les $(n, m)$-catégories dès que~$m > 0$ : nul besoin de $(n, m)$-catégories
faibles quand on s'intéresse uniquement aux types d'homotopie ; les
$(m, n)$-catégories strictes suffisent. Le premier résultat dans cette
direction se trouve dans la thèse d'Illusie~\cite{IL}, attribué par celui-ci
à Quillen, et affirme que les petites catégories modélisent les types
d'homotopie. Plus précisément, si on localise la catégorie $\Cat$ des
petites catégories par les \emph{équivalences faibles}, foncteurs dont l'image par
le foncteur nerf, introduit par Grothendieck dans~\cite{Nerf}, est une
équivalence faible simpliciale, on obtient une catégorie équivalente à la
catégorie homotopique des CW\nbd-complexes (qui est équivalente à la
localisation de celle de \emph{tous} les espaces topologiques par les
équivalences faibles topologiques, définies en termes de groupes
d'homotopie).

\smallbreak

Plus tard, Thomason a montré~\cite{Th} que ces équivalences faibles de
$\Cat$, qu'on appelle depuis les \emph{équivalences de Thomason}, font
partie d'une structure de catégorie de modèles de Quillen sur $\Cat$, et
qu'il existe une équivalence de Quillen entre celle-ci et la structure de
Kan-Quillen sur les ensembles simpliciaux~\cite{Qu1}. Ainsi, on a une
équivalence, non seulement en tant que catégories, mais aussi en tant que
$(\infty,1)$\nbd-catégories, entre la localisation de Dwyer-Kan~\cite{DK} de
$\Cat$ et la $(\infty,1)$\nbd-catégorie des types d'homotopie. De plus,
grâce à ce théorème, on dispose de tous les outils provenant de la théorie
des catégories de modèles de Quillen~\cite{Qu1} pour étudier la théorie de
l'homotopie des petites catégories.

\smallbreak

Le premier à avoir vu l'importance des petites catégories comme modèles des
types d'homotopie est Quillen et cette idée est au c\oe{}ur de son travail
sur la K-théorie algébrique supérieure~\cite{QuK}. En effet, celui-ci a
défini les groupes de K-théorie comme étant les groupes d'homotopie de
l'espace des lacets d'un espace associé à une catégorie. Ses célèbres
théorèmes~A et B, qui donnent respectivement une condition suffisante pour
qu'un foncteur soit une équivalence de Thomason et pour qu'un certain carré
de foncteurs soit homotopiquement cartésien, jouent un rôle primordial
dans l'établissement des propriétés de la K-théorie algébrique. Le théorème
A est également constamment utilisé dans les huit articles de Neeman
consacrés à ses travaux sur la K-théorie des catégories triangulées, dont on
peut trouver un compte-rendu dans~\cite{Neeman}.

\smallbreak

C'est Grothendieck qui place la catégorie des petites catégories, et le
théorème~A, au centre de la théorie de l'homotopie~\cite{PS1}. Son
inspiration est toposique: la théorie de l'homotopie des espaces est pour
lui un cas particulier de celle des topos. Or, on dispose d'une notion
d'équivalence faible dans la catégorie des topos et des morphismes
géométriques entre ceux-ci, celle d'équivalence
d'Artin-Mazur~\cite{ArtMaz}. Un morphisme de topos $f=(f_*,f^*):X\to Y$ est
une \emph{équivalence d'Artin-Mazur} si pour tout $n\geq0$, le morphisme
$H^n(Y,\mathcal{L})\to H^n(X,f^*(\mathcal{L}))$, induit par $f$, est un
isomorphisme pour tout faisceau localement constant $\mathcal{L}$ sur $Y$,
d'ensembles si $n=0$, de groupes si $n=1$, et de groupes abéliens si
$n\geq2$. D'autre part, pour Grothendieck, une petite catégorie~$A$ est
indissociable de son topos des préfaisceaux $\pref{A}$. Ainsi, il
\emph{définit} les équivalences faibles dans $\Cat$ comme étant les
foncteurs $u:A\to B$ tels que le morphisme de topos
$(u_*,u^*):\pref{A}\to\pref{B}$ (où $u^*$ désigne le foncteur de
précomposition par $u$ et $u_*$ son adjoint à droite) est une
équivalence d'Artin-Mazur.  Avec cette définition, le théorème A est la
traduction dans $\Cat$ de la propriété des équivalences d'Artin-Mazur
affirmant que pour qu'un morphisme de topos soit une équivalence
d'Artin-Mazur, il suffit qu'il le soit localement sur son but. Les
équivalences faibles de $\Cat$ ainsi définies coïncident avec les
équivalences de Thomason définies par le nerf (voir par
exemple~\cite{Moer}).

\smallbreak

Grothendieck a axiomatisé l'étude des équivalences faibles dans $\Cat$ en
introduisant dans~\cite{PS1} (voir aussi~\cite{MalPRST} et~\cite{Cisinski}) la
notion de localisateur fondamental, classe de flèches de $\Cat$ satisfaisant
à une liste de propriétés formelles vérifiées par les équivalences de
Thomason, dont la plus importante est le théorème~A. Cette notion est à la
base de sa théorie des catégories test, petites catégories dont la catégorie
des préfaisceaux modélise canoniquement les types d'homotopie, à l'instar de
la catégorie des simplexes. Il a associé à tout localisateur fondamental des
notions de foncteur asphérique, propre, lisse, les foncteurs à la fois
propres et lisses ayant des propriétés analogues à celles des
quasi-fibrations en topologie.  Cela lui a permis de construire l'espace des
lacets d'une petite catégorie à partir d'une catégorie de chemins en forme
de zigzag. Il a conjecturé que les équivalences de Thomason forment le plus
petit localisateur fondamental (conjecture de minimalité), et qu'il est le
seul non grossier à satisfaire au théorème~B. En pratique, la conjecture de
minimalité signifie que le \emph{seul} moyen non trivial pour prouver qu'un
foncteur est une équivalence de Thomason est le théorème~A. Ces deux
conjectures ont été prouvées par Cisinski~\cite{CisinskiLFM, Cisinski}, qui a
également démontré que, modulo des questions ensemblistes, les localisateurs
fondamentaux sont en bijection avec les localisations de Bousfield à gauche
des types d'homotopie.

\subsection*{Les $n$-catégories comme modèles des types d'homotopie}

Le présent article fait partie d'un vaste projet consacré à la théorie de
l'homotopie des \oo-catégories strictes et à la généralisation aux
catégories supérieures strictes des résultats exposés dans les paragraphes
précédents. Cette généralisation est en partie motivée par le fait que, si
les petites catégories modélisent bien les types d'homotopie, les modèles
catégoriques produits sont en général peu naturels. On verra plus loin dans
cette introduction que les \oo-catégories strictes fournissent des modèles
plus simples et parfois très géométriques.

\smallbreak

Pour pouvoir généraliser ces résultats aux catégories supérieures, on doit
disposer d'une classe convenable d'équivalences faibles. Cela a été rendu
possible grâce à l'introduction par Street~\cite{StreetOrient} de l'objet
cosimplicial des \emph{orientaux} dans $\ooCat$, la catégorie
des petites \oo-catégories strictes et \oo-foncteurs stricts.  Cet objet
cosimplicial définit par le procédé de Kan un foncteur nerf, appelé
\emph{nerf de Street}, de $\ooCat$ vers la catégorie des ensembles
simpliciaux, adjoint à droite d'un foncteur de \emph{réalisation
\oo-catégorique}. La restriction du nerf de Street à $\Cat$ coïncide avec le
nerf usuel. On définit les équivalences faibles dans $\ooCat$, appelées par
extension \emph{équivalences de Thomason}, comme étant les \oo-foncteurs
stricts dont le nerf de Street est une équivalence faible simpliciale. Les
équivalences faibles dans la catégorie $\nCat{n}$ des petites
$n$\nbd-catégories strictes, pour $n \ge 0$, et plus généralement
dans $\nCat{(n,m)}$, la catégorie des petites $(n,m)$\nbd-catégories
strictes, pour $0\leq m\leq n\leq\infty$, appelées aussi \emph{équivalences
de Thomason}, sont les flèches dont l'image par l'inclusion dans $\ooCat$
est une équivalence de Thomason. Il existe d'autres foncteurs nerfs comme le
\emph{nerf $n$\nbd-simplicial}, le \ndef{nerf cubique}, ou le \emph{nerf
cellulaire} à valeurs dans les ensembles cellulaires, préfaisceaux sur la
catégorie $\Theta$ de Joyal~\cite{Joy, CB1, CB2, CiGM}.
Dans~\cite{DG2}, nous montrerons que ces nerfs définissent les mêmes
équivalences faibles que le nerf de Street, ce qui permet entre autres
d'étudier les propriétés de dualité des équivalences de Thomason, et s'avère
crucial dans~\cite{conde}.

\smallbreak

Le théorème d'Illusie-Quillen affirmant l'équivalence de la localisation de
$\Cat$ par les équivalences de Thomason et de la catégorie homotopique des
espaces se généralise aux $(n,m)$\nbd-catégories pourvu que $m > 0$,
c'est-à-dire pourvu qu'on n'exige pas que les $1$-cellules soient
inversibles. Ce résultat a été établi pour les $2$-catégories strictes par
J.~Chiche~\cite{ChicheThese, ChicheThmA} en utilisant des techniques
semblables à celles de~\cite{Hoyo}. Il a été généralisé par
A.~Gagna~\cite{Andrea} pour les $n$\nbd-catégories strictes,
\hbox{$0<n\leq\infty$}.
En utilisant nos résultats de~\cite{conde}, il obtient le cas général des
$(n,m)$\nbd-catégories strictes, $0<m\leq n\leq\infty$. Comme les
\oo-groupoïdes stricts ne modélisent pas les types d'homotopie, ce résultat
est optimal (il est faux si~$m=0$).

\smallbreak

On peut expliquer la philosophie de ces résultats en rapport avec
l'hypothèse d'homotopie comme suit. La catégorie des \oo-groupoïdes
faibles est une sous-catégorie de la catégorie des \oo-catégories faibles
($\infty$ signifiant comme dans le reste de cette introduction
$(\infty,\infty)$, et \emph{non pas} $(\infty,1)$, et les morphismes
respectant \emph{strictement} les structures). Cette inclusion admet
un adjoint à gauche qui inverse formellement (faiblement)
toutes les $i$\nbd-cellules, pour $i>0$. Les catégories, les \oo-catégories
strictes, et plus généralement, les $(n,m)$\nbd-catégories strictes sont en
particulier des \oo-catégories faibles. Conjecturalement leur type
d'homotopie, autrement dit le type d'homotopie de leur nerf de Street,
coïnciderait avec le type d'homotopie représenté, en vertu de l'hypothèse
d'homotopie, par le \oo-groupoïde faible obtenu en leur appliquant cet
adjoint à gauche. (Cette conjecture généraliserait le fait que le
$1$\nbd-tronqué du type d'homotopie d'une catégorie est représenté par son
groupoïde fondamental, obtenu en inversant formellement (strictement)
toutes ses flèches.) Ainsi, le besoin de cohérences non triviales à la place
des égalités dans les \oo-groupoïdes pour modéliser les types d'homotopie
apparaîtrait dans ce processus d'inversion formelle des cellules. On se
gardera de croire que les seules cohérences non triviales seraient celles
exprimant que ces inverses formels sont des inverses faibles. En effet, on sait
que les \oo-groupoïdes faibles dont la \oo-catégorie sous-jacente est
stricte ne suffisent pas pour modéliser les types
d'homotopie~\cite{Simps3types, AraTypHomStr}. \emph{A priori}, toutes les
cohérences faisant intervenir au moins un des inverses formels seront non
triviales.

\smallbreak

Dans~\cite{DG}, nous avons démontré un « théorème de Thomason abstrait »
permettant d'établir une liste de propriétés suffisantes pour construire une
structure de catégorie de modèles « à la Thomason » sur $\nCat{n}$, pour
$0<n\leq\infty$, et une adjonction de Quillen (qui sera une équivalence
grâce aux résultats de J.~Chiche pour $n=2$ et de A.~Gagna pour $n$ général)
avec la structure de Kan-Quillen sur les ensembles simpliciaux. Nous
avons prouvé toutes ces propriétés pour $n=1,2$, établissant ainsi, pour $n=2$,
l'analogue exact du théorème de Thomason, et toutes sauf deux dans le cas
général. L'une de ces deux propriétés est établie dans~\cite{conde}, la
dernière restant pour l'instant conjecturale. Dans un travail en cours,
A.~Gagna explore le cas~$n=3$. Le cas général est un des fils conducteurs de
notre projet. Notons que si cette conjecture est établie, les
généralisations du théorème d'Illusie-Quillen fourniraient non seulement des
équivalences de catégories, mais aussi des équivalences de
$(\infty,1)$\nbd-catégories et on disposerait de tous les outils des
catégories de modèles pour étudier la théorie de l'homotopie des petites
$n$\nbd-catégories.

\smallbreak

Dans~\cite{BCeg}, M.~Bullejos et A.~M.~Cegarra ont établi un théorème~A pour
les $2$\nbd-catégories strictes.  Une variante du théorème~A pour les
$2$\nbd-foncteurs lax normalisés de source une $1$\nbd-catégorie a été
prouvée par M.~del~Hoyo~\cite{delHoyoLoop} et étendue aux $2$\nbd-foncteurs
lax généraux dans~\cite{HoyoThA}. Dans~\cite{ChicheThese, ChicheThmA},
J.~Chiche en a établi une version relative et à transformation oplax près. Dans
le présent article, nous établissons une version relative et à
transformation oplax près dans $\ooCat$. Ce théorème est au centre de notre
projet. Il implique comme cas particuliers des théorèmes analogues pour les
$n$\nbd-catégories strictes et plus généralement pour les
$(n,m)$\nbd-catégories strictes. Notons qu'un théorème~A pour les
$(\infty,1)$\nbd-catégories faibles a été établi par Lurie~\cite{Lurie} dans
le cadre de la théorie des $(\infty, 1)$-foncteurs cofinaux, ainsi que par
G.~Heuts et I.~\hbox{Moerdijk}~\cite{HMThA}. Dans~\cite{Ceg}, A.~M.~Cegarra a
établi un théorème~B pour les $2$\nbd-catégories strictes.
Dans~\cite{AraThB}, le premier des auteurs démontre la généralisation de ce
théorème dans $\ooCat$.

\smallbreak

Pour énoncer et démontrer les théorèmes~A et~B dans $\ooCat$, nous avons été
conduits à revisiter le produit tensoriel de Gray lax \oo-catégorique
(construit pour la première fois par F.~A.~Al-Agl et
R.~Steiner~\cite{AlAglSteiner}, généralisant une construction analogue pour
les \oo-groupoïdes due à R.~Brown et P.~J.~Higgins~\cite{BrownTensor}, et
étudié dans la thèse de S.~Crans~\cite{CransThese}) et introduire une
opération joint et des tranches dans~$\ooCat$~\cite{joint}. Il s'agit, nous
semble-t-il, de constructions importantes en théorie des \oo-catégories,
indépendamment de la théorie de l'homotopie. Dans l'étude des
$(\infty,1)$-catégories faibles, et plus précisément des quasi-catégories,
le joint et les tranches jouent un rôle essentiel. En revanche, ces notions
n'ont pas encore été introduites dans le cadre des $(\infty,m)$-catégories
faibles, et notre construction peut servir de modèle.

\smallbreak

Dans~\cite{ChicheThese} et~\cite{ChicheLoc}, J.~Chiche a généralisé la théorie
des localisateurs fondamentaux de Grothendieck aux $2$-catégories strictes,
et il a montré que les localisateurs fondamentaux de $\nCat{2}$ sont en
bijection avec ceux de $\Cat$. Ce résultat profond, combiné avec le
théorème~A $2$\nbd-catégorique, implique aussitôt la conjecture de
minimalité pour~$\nCat{2}$ : les équivalences de Thomason $2$-catégoriques
forment le plus petit localisateur fondamental de $\nCat{2}$. De même,
combiné avec le théorème~B $2$\nbd-catégorique, il implique que le
localisateur fondamental des équivalences de Thomason est le seul
localisateur fondamental non grossier de $\nCat{2}$ satisfaisant à ce
théorème. De plus, le premier des auteurs a montré que modulo des questions
ensemblistes, tout localisateur fondamental de $\nCat{2}$ est la classe des
équivalences faibles d'une structure de catégorie de modèles à la
Thomason~\cite{Ara2Cat}. La définition des localisateurs fondamentaux
s'étend facilement aux \oo-catégories strictes. La généralisation des
résultats précédents, ainsi que de la théorie des foncteurs propres et lisses
de Grothendieck, fait partie de notre projet.

\smallbreak

Le principal avantage des modèles des types d'homotopie dans $\ooCat$ par
rapport à ceux dans $\Cat$ est qu'ils sont beaucoup plus naturels. Par
exemple, le théorème~B \oo-catégorique~\cite{AraThB} implique que pour tout
groupe commutatif $\pi$, et tout $n>1$, le $n$\nbd-groupoïde strict ayant un
seul objet, les unités itérées de cet objet comme $i$\nbd-cellules, pour
$0<i<n$, et les éléments de $\pi$ comme $n$\nbd-cellules, toutes les
compositions des $n$\nbd-cellules étant définies par la composition du
groupe, est un $K(\pi,n)$, généralisant le résultat bien connu pour $n=1$ et
$\pi$ un groupe arbitraire.
De plus, si le groupe $\pi$ est réticulé, c'est-à-dire est muni d'une
structure de treillis compatible à la structure de groupe, par exemple $\pi
= \Z$, notre théorème~A implique que la sous-$n$\nbd-catégorie du
$n$-groupoïde précédent avec comme $n$\nbd-cellules les éléments
\emph{positifs} du groupe est aussi un $K(\pi,n)$. On obtient ainsi des
modèles remarquablement simples des $K(\pi,n)$.

\smallbreak

Par ailleurs, le foncteur de réalisation \oo-catégorique, adjoint à gauche
du nerf de Street, associe à tout ensemble simplicial $K$ une \oo-catégorie
stricte qui est librement engendrée au sens des polygraphes par les
simplexes non dégénérés de $K$, tout $n$\nbd-simplexe de $K$ définissant une
$n$\nbd-cellule. En général, le type d'homotopie de cette \oo-catégorie
n'est pas le même que celui de $K$. Néanmoins, dans~\cite{conde}, nous
conjecturons que si l'ensemble simplicial $K$ provient d'un \emph{complexe
simplicial} (ordonné), alors ces deux types d'homotopie coïncident, et nous
démontrons cette conjecture quand le complexe simplicial est celui associé à
un ensemble ordonné, et en particulier s'il est la subdivision barycentrique
d'un autre. Ainsi, si $X$ est une variété \emph{munie} d'une triangulation
(autrement dit, d'un homéomorphisme avec la réalisation topologique d'un
complexe simplicial), quitte éventuellement à remplacer cette triangulation
par sa subdivision barycentrique, on lui associe naturellement une
\oo-catégorie ayant le même type d'homotopie et dont les $n$\nbd-cellules
sont des composés de $n$\nbd-simplexes de $X$. Nous soupçonnons que cette
\oo-catégorie encode non seulement le type d'homotopie de $X$, mais aussi
son type d'homotopie \emph{dirigée} (induit par l'homéomorphisme avec le
complexe simplicial).

\smallbreak

Un des outils techniques essentiels pour la réalisation de notre projet est
la très belle théorie des complexes dirigés augmentés de
Steiner~\cite{Steiner}. Un \ndef{complexe dirigé augmenté} est un complexe de
chaînes de groupes abéliens, concentré en degrés positifs, muni d'une
augmentation, et pour tout $n\geq0$, d'un sous-monoïde du groupe des
$n$\nbd-chaînes (ou de façon équivalente d'une relation de préordre
compatible avec la structure de groupe) sans aucune compatibilité avec la
différentielle ou l'augmentation. Steiner définit un foncteur
d'\emph{abélianisation} de $\ooCat$ vers la catégorie $\Cda$ des complexes
dirigés augmentés qui admet un adjoint à droite, le foncteur de
\emph{\oo-catégorification}. Il définit une classe de complexes dirigés
augmentés, que nous appelons des \emph{complexes de \hbox{Steiner}}, telle que la
restriction du foncteur de \oo-catégorification à la sous-catégorie pleine
de $\Cda$ formée de ces complexes est pleinement fidèle et induit une
équivalence de catégories avec une sous-catégorie pleine dense de $\ooCat$
formée de \oo-catégories strictes libres au sens des polygraphes. Ce
résultat permet de réduire certaines constructions ou vérifications
\oo-catégoriques à des constructions ou vérifications bien plus simples sur
les complexes de groupes abéliens.

\smallbreak

Dans~\cite{conde}, nous avons conjecturé que la localisation de la catégorie
$\Cda$ par les morphismes dont l'image par le foncteur de
\oo-catégorification est une équivalence de Thomason est équivalente à la
catégorie homotopique des espaces, cette équivalence étant induite par le
composé du nerf de Street et du foncteur de \oo-catégorification. Cette
conjecture a été démontrée par A.~Gagna, initialement en utilisant nos
résultats de~\cite{conde}, et ensuite par une méthode alternative, mais
toujours passant par les \oo-catégories strictes~\cite{Andrea}. Il s'agit
peut-être du sous-produit le plus spectaculaire de notre projet; c'est à
notre connaissance le premier modèle de \emph{tous} les types d'homotopie
entièrement basé sur les complexes de groupes abéliens.

\smallbreak

Un autre aspect de notre projet consiste à faire le pont entre nos résultats
et ceux relatifs à la structure de catégorie de modèles dite « folk » sur
$\ooCat$ dont les équivalences faibles sont les équivalences
\oo-catégoriques~\cite{folk}. On peut montrer que la classe de ces
équivalences est incluse dans celle des équivalences de Thomason et que ces
deux classes coïncident si on se restreint aux \oo-groupoïdes. En dérivant
le foncteur d'abélianisation de $\ooCat$ vers la catégorie des complexes des
groupes abéliens pour la structure « folk », on obtient la notion
d'\ndef{homologie polygraphique}~\cite{MetHomol}, tandis qu'en le dérivant
relativement aux équivalences de Thomason, on obtient l'homologie de
son nerf de Street. Ces deux homologies ne coïncident pas en général,
l'homologie polygraphique étant un invariant plus fin. Néanmoins, quand on
se restreint aux monoïdes, considérés comme \oo-catégories strictes avec un
seul objet et des $i$\nbd-cellules triviales pour $i\geq2$, ces deux
homologies sont isomorphes et on retrouve l'homologie ordinaire du
monoïde~\cite{LMMon}. Ce résultat est sur le point d'être généralisé à
toutes les ($1$\nbd-)catégories~\cite{Leonard}. L'analogue de l'homologie
polygraphique dans le cadre des $(n,1)$\nbd-catégories strictes, $2\leq
n\leq\infty$, est activement étudié en rapport avec la théorie de la
réécriture et la généralisation des liens établis par Squier entre les
propriétés des systèmes de réécriture pour les monoïdes et leurs propriétés
homologiques et homotopiques (voir par exemple~\cite{GuirMal}).

\subsection*{Contenu de l'article}

\medskip

Le but de cet article est de présenter une preuve aussi directe que possible
d'une version du théorème~A de Quillen pour les \oo-catégories strictes,
avec le minimum d'outils \oo-catégoriques, en utilisant des techniques
simpliciales, et les complexes dirigés augmentés de Steiner~\cite{Steiner}.
Dans~\cite{DG3}, nous donnerons une autre preuve, basée sur la théorie du
joint \oo-catégorique développée dans~\cite{joint}, et on montrera que la
mystérieuse homotopie simpliciale de la section~\ref{section:hmtpsmpl}
provient d'une transformation oplax de fonctorialité des
tranches.

\smallbreak

Dans la section introductive~\ref{section:ThAQ}, on rappelle l'énoncé du
théorème~A de Quillen dans $\Cat$ et on esquisse la preuve d'une variante
relative et à transformation près. Les sections~\ref{section:prelim_simpl}
et~\ref{section:prelim_ADC} sont consacrées respectivement à des
préliminaires sur les ensembles simpliciaux et des rappels sur les complexes
dirigés augmentés de Steiner et le nerf de Street. Dans la
section~\ref{section:prelim_cotr}, on introduit les notions de
transformation oplax et de tranche \oo-catégoriques, nécessaires pour
l'énoncé et la preuve du théorème~A dans~$\ooCat$. Les
sections~\ref{section:ThA_ooCat} et~\ref{section:hmtpsmpl} sont dédiées à la
preuve proprement dite de ce théorème.

\section{Le théorème A de Quillen}\label{section:ThAQ}

\begin{paragr}
On rappelle que le \emph{nerf} d'une petite catégorie $A$ est l'ensemble simplicial $\NGr A$ dont les $n$\nbd-simplexes sont les suites de $n$ morphismes composables de $A$:
\[
(\NGr A)_n=\{a_0\to a_1\to\cdots\to a_n\}\,,
\]
les opérateurs simpliciaux étant définis de la façon évidente. On dit qu'un foncteur entre petites catégories $u:A\to B$ est une \emph{équivalence faible} si son image par le foncteur nerf $\NGr u:\NGr A\to\NGr B$ est une équivalence faible simpliciale. On rappelle que si $b$ est un objet de $B$, on note $\cotr{A}{b}$ la catégorie dont les objets sont les couples $(a,g)$ formés d'un objet $a$ de $A$ et d'une flèche $g:b\to u(a)$ de $B$, un morphisme de $(a,g)$ vers un autre objet $(a',g')$ de $\cotr{A}{b}$ étant une flèche $f:a\to a'$ de $A$ telle que $g'=u(f)g$. On vérifie facilement que pour toute petite catégorie $T$, se donner un foncteur de $T$ vers $\cotr{A}{b}$ revient à se donner un foncteur $t:T\to A$ et une transformation naturelle $\tau:b\to ut$, où l'on note aussi $b:T\to B$ le foncteur constant de valeur $b$. Le théorème~A classique de Quillen~\cite{QuK} affirme que si pour tout objet $b$ de $B$ le foncteur
\[
\cotr{u}{b}:\cotr{A}{b}\to\cotr{B}{b}\,,\qquad(a,g)\mapsto(u(a),g)\,,
\]
induit par $u$, est une équivalence faible, alors le foncteur $u$ lui-même est une équivalence faible. Une version légèrement plus générale de ce théorème, qui se démontre exactement de la même façon, est la suivante:
\end{paragr}

\begin{thm}\label{thm:th_A_1-Cat}
Soit $\mathcal{T}$ un triangle de foncteurs entre petites catégories, commutatif à transformation naturelle donnée (non nécessairement inversible) près:
\[\mathcal{T}=
\raise 25pt
\vbox{
    \shorthandoff{;}
    \xymatrix@C=1.5pc{
      A \ar[rr]^u \ar[dr]_{v}_{}="f" & & B \ar[dl]^(0.42){w} \\
      & C
      \ar@{}"f";[ur]_(.15){}="ff"
      \ar@{}"f";[ur]_(.55){}="oo"
      \ar@<-0.5ex>@2"ff";"oo"^{\alpha}
      &.
    }
}
\]
Si pour tout objet $c$ de $C$, le foncteur $\cotr{\mathcal{T}}{c}:\cotr{A}{c}\to\cotr{B}{c}$
\[
\xymatrixcolsep{1pc}
(a,\xymatrix{c\ar[r]^-g&v(a)})\ \mapsto\ (u(a),\xymatrix{c\ar[r]^-g&v(a)\ar[r]^-{\alpha_a}&w(u(a)))}\,,
\]
induit par $\mathcal{T}$, est une équivalence faible, alors il en est de même de $u$.
\end{thm}

Le théorème A de Quillen classique est le cas particulier où $C=B$, $w=1_B$, $v=u$ et $\alpha=1_u$. L'ingrédient principal de la preuve de ce théorème est le \og lemme bisimplicial\fg{}. Si $K$ est un ensemble bisimplicial, on note $\Diag(K)$ l'ensemble simplicial dont les $n$\nbd-simplexes sont les éléments de $K_{n,n}$, les opérateurs simpliciaux étant définis de la façon évidente à partir de ceux de $K$. On dit qu'un morphisme d'ensembles bisimpliciaux $K\to L$ est une \emph{équivalence faible diagonale} si le morphisme induit $\Diag(K)\to\Diag(L)$ est une équivalence faible simpliciale.

\begin{lembsmpl} \label{lembsmpl}
Soit $K\to L$ un morphisme d'ensembles bisimpliciaux. Si pour tout $m\geq0$ \emph{(resp.} pour tout $n\geq0$\emph{)}, le morphisme d'ensembles simpliciaux
\[
K_{m,\bullet}\to L_{m,\bullet}\qquad\hbox{\emph{(resp.}}\quad K_{\bullet,n}\to L_{\bullet,n}\ )
\]
est une équivalence faible, alors le morphisme d'ensembles bisimpliciaux $K\to L$ est une équivalence faible diagonale.
\end{lembsmpl}

\begin{proof}
Voir par exemple \cite[Chapitre XII, paragraphe 4.3]{BousKan} ou \cite[proposition 2.1.7]{CisinskiLFM} pour une preuve plus moderne.
\end{proof}

\begin{proof}[Esquisse de preuve du théorème~\ref{thm:th_A_1-Cat}]
Pour tout foncteur $f:X\to Y$ entre petites catégories, on note $S(f)$ l'ensemble bisimplicial
\[
(S(f))_{m,n}=\{y_0\to y_1\to\cdots\to y_m\to f(x_0)\,,\ x_0\to x_1\to\cdots\to x_n\}\,,
\]
où les $y_i$ sont dans $Y$ et les $x_j$ dans $X$, les opérateurs simpliciaux étant définis de la façon évidente. On a un morphisme d'oubli $U_f:S(f)\to\NGr X$, défini par
\[
(y_0\to y_1\to\cdots\to y_m\to f(x_0)\,,\ x_0\to x_1\to\cdots\to x_n)\mapsto x_0\to x_1\to\cdots\to x_n\,,
\]
l'ensemble simplicial $\NGr X$ étant considéré comme ensemble bisimplicial constant en la première variable
\[
(\NGr X)_{m,n}=(\NGr X)_n=\{x_0\to x_1\to\cdots\to x_n\}\,.
\]
Pour tout $n\geq0$, le morphisme d'ensembles simpliciaux $(U_f)^{}_{\bullet,n}$ s'identifie à la somme
\[
\textstyle\coprod\limits_{x_0\to\cdots\to x_n}\kern -5pt\NGr(\tr{Y}{f(x_0)})\kern6pt\toto\kern-3pt\coprod\limits_{x_0\to\cdots\to x_n}\kern -5pt\ast\,,
\]
indexée par les $n$\nbd-simplexes $x_0\to\cdots\to x_n$ de $\NGr X$, des morphismes de source $\NGr(\tr{Y}{f(x_0)})$ et de but le point simplicial, qui sont des équivalences faibles puisque les catégories $\tr{Y}{f(x_0)}$ admettent un objet final. La stabilité des équivalences faibles par sommes et le lemme bisimplicial impliquent alors que le morphisme d'ensembles bisimpliciaux $U_f$ est une équivalence faible diagonale.
\smallbreak

Or, sous les hypothèses du théorème, on a un carré commutatif
\[
\xymatrixcolsep{3pc}
\xymatrix{
&S(v)\ar[r]^{S(\mathcal{T})}\ar[d]_{U_v}
&S(w)\ar[d]^{U_w}
\\
&\NGr A\ar[r]_{\NGr(u)}
&\NGr B
&\kern -30pt,\kern40pt
}
\] 
le morphisme d'ensembles bisimpliciaux $S(\mathcal{T})$ étant défini par
\[\begin{aligned}
(S(\mathcal{T}))_{m,n}(c_0\to&\cdots\to c_m\to v(a_0),\,a_0\to\cdots\to a_n)\cr
&=(c_0\to\cdots\to c_m\to wu(a_0),\,u(a_0)\to\cdots\to u(a_n))\,,\quad m,n\geq0\,,
\end{aligned}\]
où $c_m\to wu(a_0)$ est le composé
\[
\xymatrixcolsep{1.4pc}\xymatrix{c_m\ar[r]&v(a_0)\ar[r]^(.45){\alpha_{a_0}}&wu(a_0)}\,.
\]
D'autre part, on observe que pour tout $m\geq0$, le morphisme d'ensembles simpliciaux $(S(\mathcal{T}))_{m,\bullet}$ s'identifie à la somme
\[
\textstyle\coprod\limits_{c_0\to\cdots\to c_m}\kern -5pt\NGr(\cotr{A}{c_m})\kern6pt\toto\kern-3pt\coprod\limits_{c_0\to\cdots\to c_m}\kern -5pt\NGr(\cotr{B}{c_m})\,,
\]
indexée par les $m$\nbd-simplexes $c_0\to\cdots\to c_m$ de $\NGr C$, des morphismes 
\[
\NGr(\cotr{\mathcal{T}}{c_m}):\NGr(\cotr{A}{c_m})\toto\NGr(\cotr{B}{c_m})
\]
qui sont des équivalences faibles en vertu de l'hypothèse du théorème. La stabilité des équivalences faibles par sommes et le lemme bisimplicial impliquent alors que le morphisme d'ensembles bisimpliciaux $S(\mathcal{T})$ est une équivalence faible diagonale. On en déduit par deux sur trois que le morphisme bisimplicial constant en la première variable $\NGr(u)$ est une équivalence faible diagonale, autrement dit, que le morphisme d'ensemble simpliciaux $\NGr(u)$ est une équivalence faible, ce qui prouve le théorème.
\end{proof}

\begin{paragr}
Le but de cet article est de prouver l'analogue exact de ce théorème pour un
triangle de \oo-foncteurs stricts entre petites \oo-catégories strictes,
commutatif à transformation oplax près. Pour ce faire, il faut commencer par
donner un sens à l'énoncé, à savoir définir les équivalences faibles dans la
catégorie $\ooCat$ des petites \oo-catégories strictes et \oo-foncteurs
stricts entre celles-ci, et introduire les \og cotranches\fg{} $\cotr{X}{y}$,
pour $X\to Y$ un \oo-foncteur et $y$ un objet de $Y$, ainsi que la notion de
transformation oplax entre \oo-foncteurs. Dans ce texte, toutes les
\oo-catégories considérées seront strictes et petites et tous les
\oo-foncteurs seront stricts, ainsi on dira plus simplement \oo-catégorie et
\oo-foncteur pour les objets et les morphismes de $\ooCat$. La structure de
la preuve du théorème qu'on va établir est la même que celle du théorème A
de Quillen, sauf que les vérifications deviennent hautement non triviales.
\smallbreak

Si $C$ est une \oo-catégorie et $i\geq0$, on note $C_i$ l'ensemble de ses $i$\nbd-cellules. Pour toute $i$\nbd-cellule $x$ de $C$, on note $1_x$ la $(i+1)$\nbd-cellule unité de $x$, et si $i>0$, on note $s(x)$ la $(i-1)$\nbd-cellule source de $x$ et $t(x)$ la $(i-1)$\nbd-cellule but. Pour $j$ tel que $0\leq j<i$, on note $s_j(x)$ (resp. $t_j(x)$) la $j$\nbd-cellule source itérée (resp. but itéré) de~$x$ et on pose $s_i(x)=x$ (resp. $t_i(x)=x$). Pour $0\leq j<i$, si $x$ et $y$ sont deux $i$\nbd-cellules \emph{$j$\nbd-composables} de $C$, autrement dit si 
$s_j(x)=t_j(y)$, on note \smash{$x\comp_jy$} le composé correspondant. 
\smallbreak

Plus généralement, pour $i,j,k$ tels que $0\leq k<\min\{i,j\}$, si $x$ est une $i$\nbd-cellule et $y$ une $j$\nbd-cellule de $C$ telles que $s_k(x)=t_k(y)$, on notera \smash{$x\comp_ky$} la $(\max\{i,j\})$\nbd-cellule obtenue en composant $x$ avec la $i$\nbd-cellule identité itérée de $y$ ou la $j$\nbd-cellule identité itérée de $x$ avec $y$ selon que $i\geq j$ ou $i\leq j$. 
\end{paragr}

\section{Préliminaires simpliciaux}\label{section:prelim_simpl}

\begin{paragr}
On rappelle que la \ndef{catégorie des simplexes} $\cDelta$ est la sous-catégorie pleine de la catégorie des ensembles ordonnés formée des ensembles
\[
\Deltan{n}=\{0,1,\dots,n\}\,,\qquad n\geq0\,,
\]
ordonnés par l'ordre naturel. La catégorie des \ndef{ensembles simpliciaux} est la catégorie $\EnsSimp$ des préfaisceaux sur $\cDelta$. On identifiera, par le plongement de Yoneda, $\cDelta$ à une sous-catégorie pleine de $\EnsSimp$. Comme de coutume, si $X$ est un ensemble simplicial et $n\geq0$, on note $X_n$ l'ensemble $X(\Deltan{n})$ de ses \ndef{$n$\nbd-simplexes}, et si $\varphi:\Deltan{m}\to\Deltan{n}$ est un morphisme de $\cDelta$, on note $X_\varphi$ l'application $X(\varphi):X_n\to X_m$.
\end{paragr}

\begin{notation}\label{notation:simili_joint}
Soient $X$ un ensemble simplicial, $n\geq0$ un entier positif, et $x$ un $n$\nbd-simplexe de $X$. Pour tout $m\geq0$ et toute suite d'entiers $0\leq i_0\leq\cdots\leq i_m\leq n$, on note $x_{i_0,\dots,i_m}$ le $m$\nbd-simplexe $X_\varphi(x)$ de $X$, où $\varphi:\Deltan{m}\to\Deltan{n}$ désigne l'application définie par $\varphi(k)=i_k$, pour $0\leq k\leq m$.
\smallbreak

Pour tous $m,n\ge0$, on pose $\Deltan{m}\amalg\Deltan{n}=\Deltan{m+1+n}$ et on note
\[
i_{m,n}:\Deltan{m}\to\Deltan{m+1+n}\qquad\hbox{et}\qquad j_{m,n}:\Deltan{n}\to\Deltan{m+1+n}
\]
les inclusions \og initiale\fg{} et \og finale\fg, définies par
\[
i_{m,n}(k)=k\,,\quad0\leq k\leq m\,,\qquad\hbox{et}\qquad j_{m,n}(l)=m+1+l\,,\quad0\leq l\leq n\,.
\]
Ainsi, pour tout ensemble simplicial $X$, et tout $(m+1+n)$\nbd-simplexe $x$ de $X$, on a 
\[
X_{i_{m,n}}(x)=x_{0,\dots,m}\qquad\hbox{et}\qquad X_{j_{m,n}}(x)=x_{m+1,\dots,m+1+n}\,.
\]
Si $\varphi:\Deltan{m'}\to\Deltan{m}$ et $\psi:\Deltan{n'}\to\Deltan{n}$ sont deux morphismes de $\cDelta$, on note $\varphi\amalg\psi$ le morphisme $\varphi\amalg\psi:\Deltan{m'+1+n'}\to\Deltan{m+1+n}$ défini par
\[
(\varphi\amalg\psi)(i)=
\left\{\begin{matrix}
&\varphi(i)\hfill&\qquad0\leq i\leq m'\hfill\,,\cr
\noalign{\vskip 3pt}
&m+1+\psi(i-1-m')&\qquad m'+1\leq i\leq m'+1+n'\,,
\end{matrix}\right.
\]
de sorte que pour tout ensemble simplicial $X$ et tout $(m+1+n)$\nbd-simplexe $x$ de $X$, on a
\[
(X_{\varphi\amalg\psi}(x))_{0,\dots,m'}=X_\varphi(x_{0,\dots,m})\ \hbox{et}\ (X_{\varphi\amalg\psi}(x))_{m'+1,\dots,m'+1+n'}=X_\psi(x_{m+1,\dots,m+1+n})\,.
\]
Si $m'=m$ et $\varphi=1_{\Deltan{m}}$, on note plus simplement $\varphi\amalg\psi=\Deltan{m}\amalg\psi$ et de même si $n'=n$ et $\psi=1_{\Deltan{n}}$, on note $\varphi\amalg\psi=\varphi\amalg\Deltan{n}$. On remarque que $\Deltan{m}\amalg\Deltan{n}$ est la somme disjointe au niveau des ensembles sous-jacents, mais \emph{n'est pas} la somme catégorique de $\Deltan{m}$ et $\Deltan{n}$ dans $\cDelta$ (qui n'existe pas). Les foncteurs $\Deltan{0}\,\amalg\,?:\cDelta\to\cDelta$ et $?\,\amalg\Deltan{0}:\cDelta\to\cDelta$ sont connus sous les noms de \ndef{foncteurs de décalage vers la droite} et \ndef{vers la gauche} respectivement, et les foncteurs $\Deltan{m}\amalg\,?$ et $?\,\amalg\Deltan{n}$ sont des \ndef{foncteurs de décalage itérés}.
\end{notation}

\begin{paragr}
Soient $Y$ un ensemble simplicial, $m\geq0$, et $y$ un $m$\nbd-simplexe de $Y$. On définit des ensembles simpliciaux $\cotr{Y}{y}$ et $\tr{Y}{y}$ en posant, pour $n\geq0$,
\[\begin{aligned}
&(\cotr{Y}{y})_n=\{y'\in Y_{m+1+n}\mid y'_{0,\dots,m}=y\}=\{y'\in Y_{m+1+n}\mid Y_{i_{m,n}}(y')=y\}\,,\cr
\noalign{\vskip3pt}
&(\tr{Y}{y})_n=\{y'\in Y_{n+1+m}\mid y'_{n+1,\dots,n+1+m}=y\}=\{y'\in Y_{n+1+m}\mid Y_{j_{n,m}}(y')=y\}\,,
\end{aligned}\]
les opérateurs simpliciaux étant définis de la façon suivante. Soit $\psi:\Deltan{n'}\to\Deltan{n}$ un morphisme de $\cDelta$. Si $y'$ est un $n$\nbd-simplexe de $\cotr{Y}{y}$, alors avec les notations~\ref{notation:simili_joint}, $(\cotr{Y}{y})^{}_\psi(y')=Y^{}_{\Deltan{m}\amalg\psi}(y')$, et de même si $y'$ est un $n$\nbd-simplexe de $\tr{Y}{y}$, alors $(\tr{Y}{y})^{}_\psi(y')=Y^{}_{\psi\amalg\Deltan{m}}(y')$. On a des morphismes canoniques
\[
\cotr{Y}{y}\to Y\qquad\hbox{et}\qquad\tr{Y}{y}\to Y\,,
\]
définis par les applications
\[
(\cotr{Y}{y})_n\to Y_n\,,\quad y'\mapsto y'_{m+1,\dots,m+1+n}\qquad\hbox{et}\qquad(\tr{Y}{y})_n\to Y_n\,,\quad y'\mapsto y'_{0,\dots,n}\,.
\]
Si $f:X\to Y$ est un morphisme d'ensembles simpliciaux, on pose
\[
\cotr{X}{y}=\cotr{Y}{y}\times_YX\qquad\hbox{et}\qquad\tr{X}{y}=X\times_Y\tr{Y}{y}\,.
\]
Ainsi, pour tout $n\geq0$, on a
\[\begin{aligned}
&(\cotr{X}{y})_n=\{(y'\in Y_{m+1+n},\,x\in X_n)\mid y'_{0,\dots,m}=y,\,y'_{m+1,\dots,m+1+n}=f(x)\}\,,\cr
\noalign{\vskip 3pt}
&(\tr{X}{y})_n=\{(x\in X_n,\,y'\in Y_{n+1+m})\mid y'_{0,\dots,n}=f(x),\,y'_{n+1,\dots,n+1+m}=y\}\,.
\end{aligned}\]
\end{paragr}

Le lemme suivant résulte de~\cite[chapitre 6, section 1]{IL}; on en donne une preuve élémentaire.

\begin{lemme}\label{lemme:decal}
Pour tout ensemble simplicial $X$, tout $n\geq0$, et tout $n$\nbd-simplexe $x$ de $X$, l'ensemble simplicial $\tr{X}{x}$ est contractile.
\end{lemme}

\begin{proof}
Le $(1+n)$\nbd-simplexe $X_{\sigma_0}(x)$ de $X$, où $\sigma_0:\Deltan{1+n}\to\Deltan{n}$ est l'unique surjection croissante prenant deux fois la valeur $0$, est un $0$\nbd-simplexe de $\tr{X}{x}$ et définit donc un morphisme $\Deltan{0}\to\tr{X}{x}$. On note $s$ le composé
\[
s:\tr{X}{x}\to\Deltan{0}\to\tr{X}{x}\,.
\]
On définira une homotopie $h:\Deltan{1}\times\tr{X}{x}\to\tr{X}{x}$ de $1_{\tr{X}{x}}$ vers $s$ comme suit, ce qui prouvera l'assertion. Soient $m\geq0$ et $(\varphi,x')$,
\[
\varphi:\Deltan{m}\to\Deltan{1}\,,\quad x'\in X_{m+1+n}\,,\qquad x'_{m+1,\dots,m+1+n}=x\,,
\]
un $m$\nbd-simplexe de $\Deltan{1}\times\tr{X}{x}$. On pose $h(\varphi,x')=X_{\theta_\varphi}(x')$, où $\theta_\varphi:\Deltan{m+1+n}\to\Deltan{m+1+n}$ est défini par
\[
\theta_\varphi(i)=
\left\{\begin{matrix}
&i\hfill &\kern30pt0\leq i\leq m\quad\hbox{et}\quad\varphi(i)=0\,,\hfill\cr
\noalign{\vskip 3pt}
&m+1&\kern30pt0\leq i\leq m\quad\hbox{et}\quad\varphi(i)=1\,,\hfill\cr
&i\hfill &\kern30ptm+1\leq i\leq m+1+n\hfill\,,
\end{matrix}\right.
\]
et on vérifie aussitôt que $X_{\theta_\varphi}(x')$ est bien un $m$\nbd-simplexe de $\tr{X}{x}$, que si $\varphi=0$, alors $X_{\theta_\varphi}(x')=x'$ et que si $\varphi=1$, alors $X_{\theta_\varphi}(x')=s(x')$. Il reste donc à prouver que ces formules définissent un morphisme d'ensembles simpliciaux $h:\Deltan{1}\times\tr{X}{x}\to\tr{X}{x}$. Soit $\psi:\Deltan{m'}\to\Deltan{m}$; il s'agit de montrer que le carré
\[
\xymatrix{
(\Deltan{1}\times\tr{X}{x})_m\ar[r]^-h\ar[d]_{(\Deltan{1}\times\tr{X}{x})_\psi}
&(\tr{X}{x})_m\ar[d]^{(\tr{X}{x})_\psi}
\\
(\Deltan{1}\times\tr{X}{x})_{m'}\ar[r]_-h
&(\tr{X}{x})_{m'}
}
\]
est commutatif. Or, pour tout $m$\nbd-simplexe $(\varphi,x')$ de $\Deltan{1}\times\tr{X}{x}$, on a 
\[\begin{aligned}
&(\tr{X}{x})_\psi h(\varphi,x')=(\tr{X}{x})_\psi(X_{\theta_\varphi}(x'))=X_{\psi\amalg\Deltan{n}}X_{\theta_\varphi}(x')=X_{\theta_\varphi(\psi\amalg\Deltan{n})}(x')\,,\cr
\noalign{\vskip 3pt}
&h(\Deltan{1}\times\tr{X}{x})_\psi(\varphi,x')=h(\varphi\psi,X_{\psi\amalg\Deltan{n}}(x'))=X_{\theta_{\varphi\psi}}X_{\psi\amalg\Deltan{n}}(x')=X_{(\psi\amalg\Deltan{n})\theta_{\varphi\psi}}(x')\,.
\end{aligned}\]
Il suffit donc de vérifier que
$\theta_\varphi(\psi\amalg\Deltan{n})=(\psi\amalg\Deltan{n})\theta_{\varphi\psi}$.
Or, pour tout entier $i$ tel que $0\leq i\leq m'+1+n$, on a
\[\begin{aligned}
&(\psi\amalg\Deltan{n})\theta_{\varphi\psi}(i)=\left\{\begin{matrix}
&\kern-3pt(\psi\amalg\Deltan{n})(i)=\psi(i)\hfill&\quad0\leq i\leq m'\quad\hbox{et}\quad\varphi\psi(i)=0\,,\cr
\noalign{\vskip 3pt}
&\kern-3pt(\psi\amalg\Deltan{n})(m'+1)=m+1&\quad0\leq i\leq m'\quad\hbox{et}\quad\varphi\psi(i)=1\,,\cr
\noalign{\vskip 3pt}
&\kern-3pt(\psi\amalg\Deltan{n})(i)=m+i-m'&\quad m'+1\leq i\leq m'+1+n\hfill\,,
\end{matrix}\right.
\cr
\noalign{\vskip 5pt}
&\theta_\varphi(\psi\amalg\Deltan{n})(i)=\left\{\begin{matrix}
&\kern-3pt\theta_\varphi\psi(i)=
  \left\{\begin{matrix}
  &\kern-3pt\psi(i)\hfill&\ \varphi\psi(i)=0\cr
  \noalign{\vskip 2pt}
  &\kern-3ptm+1&\ \varphi\psi(i)=1
  \end{matrix}\right.
&\quad0\leq i\leq m'\,,\hfill\cr
\noalign{\vskip 3pt}
&\kern-3pt\theta_\varphi(m+i-m')=m+i-m'\hfill&\quad m'+1\leq i\leq m'+1+n\,,
\end{matrix}\right.
\end{aligned}\]
ce qui achève la démonstration.
\end{proof}

\section{Rappels sur les complexes dirigés augmentés de Steiner}\label{section:prelim_ADC}

\begin{paragr}\label{paragr:def_cda}
On rappelle qu'un \ndef{complexe dirigé augmenté}, notion introduite par Steiner~\cite{Steiner}, est un triplet $(K,K^*,e)$, où
$$K=\kern 10pt
\xymatrix{
\cdots\ar[r]^-{d_{i+1}}
&K_i\ar[r]^-{d_{i}}
&K_{i-1}\ar[r]^-{d_{i-1}}
&\cdots\ar[r]^-{d_{2}}
&K_1\ar[r]^-{d_{1}}
&K_0
}$$
est un complexe de chaînes de groupes abéliens en degrés positifs, $e:K_0\to\mathbb{Z}$ une augmentation (de sorte que $d_id_{i+1}=0$ pour $i\geq1$, et $ed_1=0$), et $K^\ast=(K^\ast_i)_{i \ge 0}$ est la donnée pour tout~$i \geq 0$ d'un sous-monoïde $K^\ast_i$ du groupe abélien $K_i$. (On ne demande aucune compatibilité entre la différentielle ou l'augmentation et ces sous-monoïdes.) Les éléments de $K_i$ seront appelés des \ndef{$i$-chaînes} ou des \ndef{chaînes de degré $i$}. 
\smallbreak

Pour $i\geq0$, le sous-monoïde $K^*_i$ induit une relation de préordre $\leq$ sur $K_i$, compatible avec sa structure de groupe, définie par
\[
x\leq y\ \Longleftrightarrow\ y-x\in K^*_i\,,
\]
et alors on a l'égalité
\[
K^*_i=\{x\in K_i\,|\,x\geq0\}\,.
\]
Ainsi, pour~$i \ge 0$, on dira qu'une $i$\nbd-chaîne de $K$ est \ndef{positive} si elle appartient à $K^*_i$, et on appellera les sous-monoïdes~$K^\ast_i$ les \ndef{sous-monoïdes de positivité} de $K$.
\smallbreak

Un \ndef{morphisme de complexes dirigés augmentés} $f:(K,K^*,e)\to(K',K'^*,e')$ est un morphisme de complexes $f:K\to K'$ compatible à l'augmentation et aux sous-monoïdes de positivité (au sens où $e'f_0=e$ et où pour tout $i\geq0$, on a $f_i(K^*_i)\subset K'^*_i$, autrement dit $f_i$ envoie les $i$\nbd-chaînes positives de $K$ sur des $i$\nbd-chaînes positives de $K'$). On note $\Cda$ la catégorie des complexes dirigés augmentés. On peut montrer que la catégorie $\Cda$ est localement présentable. En particulier, elle admet des petites limites inductives et projectives. Le foncteur complexe sous-jacent commute aux limites inductives (pour la description des limites inductives dans $\Cda$, voir~\cite[paragraphe~3.1]{joint}). 
\smallbreak

On désignera souvent, par abus de notation, un complexe dirigé augmenté par son complexe de chaînes sous-jacent.
\end{paragr}

\begin{paragr}\label{paragr:def_nu}
Dans \cite{Steiner}, Steiner  définit un couple de foncteurs adjoints
\[
\lambda:\ooCat\toto\Cda\ ,\qquad\nu:\Cda\toto\ooCat\,.
\]

Le foncteur $\lambda : \ooCat \to \Cda$ est défini de la façon suivante. Pour $C$ une \oo-catégorie et $i \ge 0$, le groupe abélien $\lambda(C)_i$ est engendré par les générateurs
\[ [x], \qquad\text{pour $x$ une $i$-cellule de $C$}, \]
soumis aux relations
\[
[x \ast_j y] = [x] + [y], \qquad\text{pour $0 \le j < i$ et $x$ et $y$ des $i$-cellules $j$-composables.}
\]
Le sous-monoïde de positivité $\lambda(C)^\ast_i$ est le sous-monoïde engendré par les $[x]$, pour $x$ une $i$-cellule de $C$. Pour $i > 0$, la différentielle $d_i : \lambda(C)_i \to \lambda(C)_{i-1}$ est définie par
\[
d([x]) = [t(x)] - [s(x)],\qquad\text{pour $x$ une $i$-cellule de $C$}.
\]
Enfin, l'augmentation $\lambda(C)_0 \to \Z$ est l'unique morphisme qui envoie, pour toute $0$\nbd-cellule~$x$ de $C$, le générateur $[x]$ sur $1$. Si $u : C \to D$ est un \oo-foncteur, pour tout $i \ge 0$, le morphisme de groupes abéliens $\lambda(u)_i : \lambda(C)_i \to \lambda(D)_i$ envoie un générateur $[x]$, pour $x$ une $i$-cellule de $C$, sur le générateur $[u(x)]$ de $\lambda(D)$. 

Le foncteur $\nu : \Cda \to \ooCat$ est défini de la façon suivante. Soit $K$ un complexe dirigé augmenté. Pour $i \ge 0$, les $i$-cellules de
$\nu(K)$ sont les tableaux
\[
\tabld{x}{i}
\]
tels que
\begin{enumerate}
\item[(\emph{a})] $x^\varepsilon_k$ est une $k$\nbd-chaîne positive de $K$, pour $\varepsilon = 0, 1$ et $0\le k \le i$ ;
\item[(\emph{b})] $d(x^\varepsilon_k) = x^1_{k-1} - x^0_{k-1}$, pour $\varepsilon = 0, 1$ et $0 < k \le i$ ;
\item[(\emph{c})] $e(x^\varepsilon_0) = 1$, pour $\varepsilon = 0, 1$ ;
\item[(\emph{d})] $x_i^0 = x_i^1$.
\end{enumerate}
La structure de \oo-catégorie sur $\nu(K)$ est décrite comme suit. Soient $i \ge 0$ et
\[
x = \tabld{x}{i}
\]
une $i$-cellule de $\nu(K)$. Si $i > 0$, les sources et buts de $x$ sont les tableaux
\[
s(x)=\tablnu{x}{i-2}{x^0_{i-1}}\quad\text{et}\quad t(x)=\tablnu{x}{i-2}{x^1_{i-1}}
\]
respectivement. Pour tout $i \ge 0$, l'identité de $x$ est le tableau
\[
1_x = \tablnu{x}{i}{0}.
\]
Enfin, si
\[
x = \tabld{x}{i}\quad\text{et}\quad y = \tabld{y}{i}
\]
sont deux $i$-cellules $j$-composables pour $i > j \ge 0$, on a
\[
x \comp^{}_j y =
\begin{pmatrix}
y^0_0
&\dots
&y^0_j
&x^0_{j+1}+y^0_{j+1}
&\dots
&x^0_{i}+y^0_{i}\cr
\noalign{\vskip 3pt}
x^1_0
&\dots
&x^1_j
&x^1_{j+1}+y^1_{j+1}
&\dots
&x^1_{i}+y^1_{i}
\end{pmatrix}.
\]
Si $f : K \to K'$ est un morphisme de complexes dirigés augmentés, le \oo-foncteur $\nu(f) : \nu(K) \to \nu(K')$ est défini par
\[
\tabld{x}{i}
\mapsto
\begin{pmatrix}
f(x^0_1)
&\dots
&f(x^0_{i-1})
&f(x^0_{i})
\cr\noalign{\vskip 3pt}
f(x^1_1)
&\dots
&f(x^1_{i-1})
&f(x^1_{i})
\end{pmatrix}.
\]
\end{paragr}

\begin{prop}[Steiner]\label{prop:Steiner_adj}
Les foncteurs
\[
\lambda : \ooCat \to \Cda, \qquad \nu : \Cda \to \ooCat
\]
forment un couple de foncteurs adjoints.
\end{prop}

\begin{proof}
Voir \cite[théorème 2.11]{Steiner}.
\end{proof}

\begin{paragr}\label{paragr:def_St} 
  Une \ndef{base} d'un complexe dirigé augmenté $K$ est un ensemble gradué
  \hbox{$B = (B_i)_{i \ge 0}$} tel que, pour tout $i \ge 0$,
  \begin{enumerate}
    \item[(\emph{a})] $B_i$ est une base du $\Z$-module $K_i$ ;
    \item[(\emph{b})] $B_i$ engendre le sous-monoïde $K^\ast_i$ de $K_i$.
  \end{enumerate}
Si un complexe dirigé augmenté $K$ admet une base $B$, on dira qu'il est \ndef{à base}, et alors pour tout $i\geq0$, la relation de préordre sur $K_i$ définie par le sous-monoïde de positivité $K^*_i$ (voir le paragraphe~\ref{paragr:def_cda}) est une relation d'ordre, et $B_i$ est l'ensemble des éléments minimaux de $K_i^\ast \sauf \{0\}$ pour cette relation d'ordre. Ainsi, si $K$ admet une base, cette base est unique et ne constitue pas une donnée supplémentaire.
\smallbreak

Soit $K$ un complexe dirigé augmenté admettant une base $B$. Pour $i\geq0$, toute $i$\nbd-chaîne $x$ de $K$ s'écrit de façon unique comme combinaison linéaire à coefficients entiers d'éléments de $B_i$. Le \ndef{support} de $x$ est l'ensemble (fini) des éléments de $B_i$ ayant un coefficient non nul dans cette combinaison linéaire. Toute $i$\nbd-chaîne $x$ de $K$ s'écrit de façon unique comme différence de deux $i$\nbd-chaînes positives à supports disjoints, $x=x_+-x_-$. On définit un tableau
\[
\atom{x}=\tabll{\atom{x}}{i},
\]
où les $\atom{x}^\varepsilon_k$ sont définis par récurrence descendante sur $k$ de $i$ à $0$ :
\begin{itemize}
\item $\atom{x}^0_i = x = \atom{x}^1_i$;
\smallskip

\item $\atom{x}^0_{k - 1} = d(\atom{x}^0_k)_-$ et $\atom{x}^1_{k - 1} = d(\atom{x}^1_k)_+$\,, \,pour $0 < k \le i$\,.
\end{itemize}
Ce tableau est une $i$-cellule de $\nu(K)$ si et seulement si, d'une part, $x$ est une $i$\nbd-chaîne positive et, d'autre part, on a les égalités $e(\atom{x}^0_0) = 1 = e(\atom{x}^1_0)$. On dit que la base $B$ de $K$ est \ndef{unitaire} si, pour tout $i \ge 0$ et tout $x$ dans~$B_i$, le tableau $\atom{x}$ est une $i$-cellule de $\nu(K)$, ce qui revient à dire qu'on a l'égalité \hbox{$e(\atom{x}^0_0) = 1 = e(\atom{x}^1_0)$}. Si le complexe dirigé augmenté $K$ est à base unitaire, pour tout élément $x$ de la base de $K$, on appelle \ndef{atome} associé à $x$ la cellule $\atom{x}$ de $\nu(K)$. On remarque que l'ensemble des objets de $\nu(K)$ est alors formé des atomes $\atom{x}$, pour $x$ dans $B_0$.
\smallbreak

Soit $K$ un complexe dirigé augmenté admettant une base $B$. On dit que la base $B$ est \ndef{sans boucles} si pour tout $i\geq0$, il existe une relation d'ordre sur $B_i$ telle que pour tout $j>i$ et tout $b\in B_j$, tout élément de $B_i$ qui est dans le support de $\atom{b}^0_i$ précède ceux qui sont dans le support de $\atom{b}^1_i$ (voir~\cite{SteinerOpetops}; l'équivalence avec la définition donnée dans~\cite{Steiner} est immédiate). On dit qu'elle est \ndef{fortement sans boucles} s'il existe une relation d'ordre sur $B$ telle que pour tout $i>0$ et tout élément $b$ de $B_i$, tout élément du support de $d_i(b)_-$ précède $b$, et $b$ précède tout élément du support de $d_i(b)_+$. On notera $\preccurlyeq_K$ la plus petite relation de préordre sur $B$ satisfaisant à ces conditions. Ainsi, dire que $B$ est fortement sans boucles revient à dire que la relation de préordre $\preccurlyeq_K$ est une relation d'ordre. Si la base $K$ est fortement sans boucles, alors elle est sans boucles (voir~\cite[proposition 3.7]{Steiner}). On appellera \ndef{complexe de Steiner} (resp. \ndef{complexe de Steiner fort}) un complexe dirigé augmenté à base unitaire sans boucles (resp. fortement sans boucles).
\end{paragr}

\begin{thm}[Steiner]\label{thm:Steiner}
Pour tout complexe de Steiner $K$, le morphisme d'adjonction
\[
\lambda(\nu(K)) \to K
\]
est un isomorphisme. En particulier, la restriction du foncteur $\nu : \Cda \to \ooCat$ à la sous-catégorie pleine de $\Cda$ formée des complexes de Steiner est un foncteur pleinement fidèle.
\end{thm}

\begin{proof}
Voir \cite[théorème 5.6]{Steiner}.
\end{proof}

\begin{paragr}\label{paragr:def_pol}
Soit $C$ une \oo-catégorie. Pour tout $i\geq0$, on note $\tb{i}(C)$ le \emph{$i$\nbd-tronqué bête de $C$}, sous-\oo-catégorie de $C$ ayant, pour $0\leq j\leq i$, les mêmes $j$\nbd-cellules que $C$, et pour $j>0$, que des $j$\nbd-cellules unités.
\smallbreak

Soit $E$ un ensemble de cellules de $C$, et posons $E_i=E\cap C_i$. On dit que $C$ est \emph{engendrée librement au sens des polygraphes par $E$} si
\begin{enumerate}
\item[(\emph{a})] $E_0 = C_0$ ;
\item[(\emph{b})] pour tout $i \ge 0$, toute \oo-catégorie $D$, tout \oo-foncteur $u : \tb{i}(C) \to D$ et toute application $f : E_{i+1} \to D_{i+1}$ \emph{compatible à la formation des sources et buts}, autrement dit telle que, pour tout $x$ dans $E_{i+1}$, on ait
      \[ s(f(x)) = u(s(x)) \quad\text{et}\quad t(f(x)) = u(t(x)), \]
      il existe un unique \oo-foncteur $u' : \tb{i+1}(C) \to D$ tel que
      \[ u'|\tb{i}(C) = u \quad\text{et}\quad u'|E_{i+1} = f. \]
  \end{enumerate}
\end{paragr}

\begin{thm}[Steiner]\label{thm:Steiner_pol}
Soit $K$ un complexe de Steiner. Alors la \oo-caté\-gorie~$\nu(K)$ est engendrée librement au sens des polygraphes par ses atomes, autrement dit, par les $\atom{x}$, où $x$ varie dans la base de $K$.
\end{thm}

\begin{proof}
Voir \cite[théorème 6.1]{Steiner}.
\end{proof}

\begin{paragr}\label{paragr:def_rigide}
Soit $f : K \to L$ un morphisme de complexes dirigés augmentés à base. On dira que $f$ est un \ndef{monomorphisme rigide} s'il est un monomorphisme envoyant tout élément de la base de $K$ sur un élément de la base de $L$. 
\end{paragr}

\begin{prop}\label{prop:amalg_Stf}
Soient $K,L,M$ trois complexes de Steiner forts, de bases respectives $B^K, B^L, B^M$, et $M\to K$ et $M\to L$ deux monomorphismes rigides. On suppose que la relation d'ordre $\preccurlyeq_M$ sur $B^M$ est une relation d'ordre total. Alors:
\begin{itemize}
\item[(a)] Le complexe dirigé augmenté $K\amalg_ML$ est un complexe de Steiner fort de base $B^K\amalg_{B^M}B^L$.
\item[(b)] Le \oo-foncteur canonique $\nu(K)\amalg_{\nu(M)}\nu(L)\to\nu(K\amalg_ML)$ est un isomorphisme de $\ooCat$.
\end{itemize}
\end{prop}

\begin{proof}
Voir \cite[proposition~3.6 et corollaire~3.20]{joint}.
\end{proof}

\begin{paragr}\label{paragr:def_cX}
\`A tout ensemble simplicial $X$, on associe un complexe dirigé augmenté $\cn X$ comme suit. Le complexe sous-jacent à $\cn X$ est le complexe de chaînes normalisé de $X$, qui a comme base les simplexes non dégénérés de $X$. Les sous-monoïdes de positivité sont les sous-monoïdes engendrés par les simplexes non dégénérés. L'augmentation associe à un 0-simplexe l'entier 1. On en déduit un foncteur $\cn :\EnsSimp\to\Cda$. On remarque que pour tout ensemble simplicial $X$, le complexe dirigé augmenté $\cn X$ est à base, et que cette base est formée des simplexes non dégénérés de $X$.
\smallbreak

On va s'intéresser plus particulièrement à la restriction de ce foncteur à $\cDelta$. Pour $m\geq0$, le complexe dirigé augmenté $\cn \Deltan{m}$ se décrit comme suit. Pour $p\geq0$, un $p$\nbd-simplexe non dégénéré de l'ensemble simplicial représentable $\Deltan{m}$ est une application strictement croissante $\Deltan{p}\to\Deltan{m}$. Ainsi, pour $p\geq0$, $(\cn {\Deltan{m}})_p$ (resp.~$(\cn {\Deltan{m}})^*_p$) s'identifie au groupe (resp. au monoïde) commutatif libre engendré par la famille des \hbox{$(p+1)$}\nbd-uplets 
\[
(i_0,i_1,\dots,i_p)\,,\quad 0\leq i_0<i_1<\cdots<i_p\leq m\,.
\]
La différentielle est définie par 
\[
d(i_0,i_1,\dots, i_p)=\textstyle\sum\limits_{0\leq k\leq p}(-1)^k(i_0,\dots, \widehat{i}_k,\dots, i_p)\,,\quad p>0\,,
\]
où $(i_0,\dots, \widehat{i}_k,\dots, i_p)=(i_0,\dots, i_{k-1},i_{k+1},\dots, i_p)$, et l'augmentation par $e(i_0)=1$. On remarque que pour $p>m$, on a $\cn (\Deltan{m})_p=0$. Le complexe dirigé augmenté $\cn {\Deltan{m}}$ est un complexe de Steiner fort, et de plus, la relation $\preccurlyeq_{\cn \Deltan{m}}$ est une relation d'ordre \emph{total} (voir~\cite[corollaire~6.1]{StreetParComp} et~\cite[exemple~3.8]{Steiner}). 
\smallbreak

Si $\varphi:\Deltan{m}\to\Deltan{n}$ est un morphisme de $\cDelta$, le morphisme $\cn (\varphi):\cn \Deltan{m}\to \cn \Deltan{n}$ est défini par
\[
(i_0,\dots,i_p)\mapsto(\varphi(i_0),\dots,\varphi(i_p))\,,\qquad\hbox{pour}\quad0\leq i_0<\cdots<i_p\leq m\,,\quad p\geq0\,,
\]
avec la convention que pour $0\leq j_0\leq\cdots\leq j_p\leq n$, s'il existe $k$ tel que $0\leq k<p$ et tel que $j_k=j_{k+1}$, alors $(j_0,\dots,j_p)=0$.
\end{paragr}

\begin{paragr}\label{paragr:def_nrf_Str}
En composant la restriction à $\cDelta$ du foncteur $\cn $ avec le foncteur $\nu$, on obtient un objet cosimplicial
\[
\xymatrix{
\cDelta\ar[r]^-{\cn |\cDelta}
&\Cda\ar[r]^-{\nu}
&\ooCat\,
}
\]
dans $\ooCat$, l'objet cosimplicial $\On{}$ des orientaux de
Street~\cite{StreetOrient, StreetParComp, Steiner, SteinerOrient}. Pour $n\geq0$, la \oo-catégorie $\On{n}$ est une $n$\nbd-catégorie, autrement dit les $i$\nbd-cellules de $\On{n}$ pour $i>n$ sont des unités. En basse dimension, on a
\[\begin{aligned}
\On{0}&=\Deltan{0}=\quad\{0\}\ ,
\\
\On{1}&=\Deltan{1}=\quad0\toto1\ ,
\\
\On{2}&=\quad\left.\raise 21pt\vbox{
\xymatrixrowsep{-.07pc}
\xymatrixcolsep{1.pc}
\xymatrix{
&1\ar[dddddr]
\\
\\
&
\\
&
\\
&\ar@{=>}[uu]^{}
\\
0\ar[uuuuur]^{}\ar[rr]_{}
&\vrule height 6pt width 0pt\ar@{=>}[uuu]^{}
&2
}
}\right.\ \ ,
\\
\On{3}&=\quad\left.\raise 30pt\vbox{
\UseTwocells
\xymatrixcolsep{.2pc}
\xymatrixrowsep{.0pc}
\xymatrix{
&&&1\ar[dddrrr]
&&&&&&&
&&&1\ar[dddrrr]\ar[dddddd]
\\&&&&&&&&&&
\\&&
&&\lltwocell<>
&&&&&&
&&&&&&
\\
0\ar[dddrrr]\ar[uuurrr]\ar[rrrrrr]
&&&
&&&2\ar[dddlll]
&\ar@3{->}[rr]
&&&
0\ar[dddrrr]\ar[uuurrr]
&&&&&&2\ar[dddlll]
\\&&&&&&&&&&
&&&\uulltwocell<>
&\uutwocell<>
\\
&&&\uultwocell<>
&&&
\\
&&&3
&&&&&&&
&&&3
}
}\right.\ \ .
\end{aligned}\]
Pour la description explicite des $n$\nbd-catégories $\On{n}$ pour $4\leq n\leq 6$, voir~\cite{StreetOrient}.
\smallbreak

L'objet cosimplicial des orientaux définit un foncteur nerf
\[
N_\infty:\ooCat\toto\EnsSimp\,,\qquad C\longmapsto (\Deltan{m}\mapsto\Hom_{\ooCat}(\On{m},C)=\Hom_{\ooCat}(\nu \cn \Deltan{m},C))\,,
\]
le nerf de Street~\cite{StreetOrient}. On vérifie facilement que si $C$ est une catégorie, considérée comme \oo-catégorie dont les $i$\nbd-cellules sont des unités pour \hbox{$i>1$}, on a \hbox{$N_\infty(C)=\NGr(C)$}. Ainsi, pour alléger la notation, on notera aussi $\Nrf$ le nerf de Street~$N_\infty$. On définit les \ndef{équivalences faibles} dans $\ooCat$ comme étant les \oo-foncteurs dont le nerf de Street est une équivalence faible simpliciale.
\end{paragr}

\begin{paragr}
On définit le produit tensoriel $K\otimes L$ de deux complexes dirigés augmentés $K$ et $L$ comme suit. Le complexe sous-jacent à $K\otimes L$ est le produit tensoriel des complexes sous-jacents à $K$ et $L$, de sorte qu'en particulier, on a
\[
(K \otimes L)_p\kern 4pt =\kern-4pt
\textstyle\bigoplus\limits_{\substack{i + j = p\\i \ge 0,\, j \ge 0}} \kern-4ptK_i \otimes L_j\,,\qquad\text{pour $p \ge 0$\,.}
\]
Pour $p\geq0$, le sous-monoïde de positivité $(K \otimes L)^\ast_p$ est le sous-monoïde de~$(K \otimes L)_p$ engendré par les éléments de la forme $x \otimes y$, avec $x$ dans~$K^\ast_i$, $y$ dans $L^\ast_j$ et~$i + j = p$. L'augmentation $e:(K\otimes L)_0=K_0\otimes L_0\to \Z$ est définie par $e(x \otimes y) = e(x)e(y)$, pour $x$ dans $K_0$ et $y$ dans $L_0$. 
\end{paragr}

\begin{prop}\label{prop:produit_Stf}
Si $K$ et $L$ sont deux complexes de Steiner forts, il en est de même pour $K\otimes L$.
\end{prop}

\begin{proof}
Voir \cite[exemple~3.10]{Steiner}.
\end{proof}

\begin{paragr}
La catégorie $\Cda$, munie du produit tensoriel $\otimes$, est une catégorie monoïdale, d'objet unité $\cn \Deltan{0}$. En vertu de la proposition précédente, la sous-catégorie pleine de $\Cda$ formée des complexes de Steiner forts en est une sous-catégorie monoïdale.
\end{paragr}

\begin{thm}\label{thm:produit_Gray}
Il existe une structure de catégorie monoïdale sur $\ooCat$, unique à isomorphisme monoïdal près, de produit 
\[ 
\otimes:\ooCat\times\ooCat\toto\ooCat\,,
\]
satisfaisant aux deux conditions suivantes:
\begin{itemize}
\item[(a)] le foncteur $\otimes:\ooCat\times\ooCat\to\ooCat$ commute aux petites limites inductives en chaque variable;
\item[(b)] la restriction du foncteur $\nu:\Cda\to\ooCat$ à la sous-catégorie pleine de $\Cda$ formée des complexes de Steiner forts est un foncteur monoïdal. 
\end{itemize}
En particulier, l'unité de cette structure monoïdale sur $\ooCat$ est la \oo-catégorie ponctuelle $\On{0}=\nu \cn \Deltan{0}$.
\end{thm}

\begin{proof}
Voir \cite[théorème A.15]{joint}.
\end{proof}

Le produit tensoriel du théorème précédent a été construit pour la première fois par Al-Agl et Steiner~\cite{AlAglSteiner}, généralisant une construction analogue pour les \oo-groupoïdes due à Brown et Higgins~\cite{BrownTensor}. Deux autres constructions sont données par Crans dans sa thèse~\cite{CransThese}.

\section{Transformations oplax et tranches \pdfoo-catégoriques}\label{section:prelim_cotr}

On va adopter la convention suivante pour les opérations de composition dans une \oo-catégorie. Si $i<j$, l'opération \smash{$\comp^{}_i$} sera prioritaire sur l'opération \smash{$\comp^{}_j$}. Par exemple:
\[
u\comp^{}_0v\comp^{}_1w\comp^{}_2x\comp^{}_1y\comp^{}_0z=\bigl((u\comp^{}_0v)\comp^{}_1w\bigr)\comp^{}_2\bigl(x\comp^{}_1(y\comp^{}_0z)\bigr)\,,
\]
lorsque le membre de droite a un sens.

\begin{paragr}
Soient $C$ une \oo-catégorie et $c$ un objet de $C$. La \oo-catégorie $\cotr{C}{c}$ se décrit comme suit. De façon informelle, les objets, $1$\nbd-cellules et $2$\nbd-cellules de $\cotr{C}{c}$ sont respectivement les diagrammes de la forme
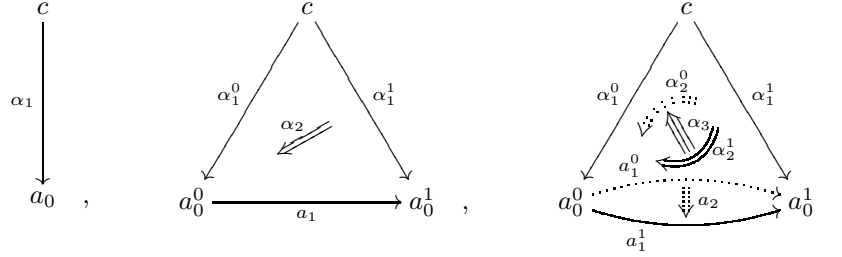
\[
   \xymatrix@R=5pc{
   c \ar[d]_{\alpha_1} \\
   a_0 \pbox{,}
   }
   \qquad
   \qquad
    \shorthandoff{;}
    \xymatrix@C=2.5pc@R=5pc{
      & c
    \ar[dl]_{\alpha^0_1}_{}="f" \ar[dr]^{\alpha^1_1}_{}="s" \\
      a^0_0 \ar[rr]_{a_1} & & a^1_0
      \ar@{}"s";[ll]_(.15){}="ss"
      \ar@{}"s";[ll]_(.55){}="tt"
      \ar@<0.0ex>@2"ss";"tt"_{\alpha_2} \pbox{,}
    }
   \qquad
   \qquad
    \shorthandoff{:;}
    \xymatrix@C=2.5pc@R=5pc{
      & c
    \ar[dl]_{\alpha^0_1}_{}="f" \ar[dr]^{\alpha^1_1}_{}="s" \\
      a^0_0
      \ar@/^2ex/@{.>}[rr]^(.27){a^0_1}^{}="0"
      \ar@/_2ex/[rr]_(.30){a^1_1}^{}="1"
      \ar@{:>}"0";"1"^{\,a_2}
      & & a^1_0
      \ar@{}"s";[ll]_(.15){}="ss"
      \ar@{}"s";[ll]_(.55){}="tt"
      \ar@<-1.5ex>@/^-2ex/@{:>}"ss";"tt"_(.30){\alpha^0_2}_{}="11"
      \ar@<-0ex>@/^2ex/@2"ss";"tt"^(.20){\!\!\alpha^1_2}^{}="00"
      \ar@3"00";"11"_{\alpha_3}
      \pbox{.}
    }
  \]
Plus formellement, pour $i\geq0$, les $i$\nbd-cellules de $\cotr{C}{c}$ sont les tableaux
\[
(a,\alpha)=
\begin{pmatrix}
(a^0_0,\alpha^0_1)&\cdots&(a^0_{i-1},\alpha^0_i)&(a^0_i,\alpha^0_{i+1})\cr
\noalign{\vskip 5pt}
(a^1_0,\alpha^1_1)&\cdots&(a^1_{i-1},\alpha^1_i)&(a^1_i,\alpha^1_{i+1})
\end{pmatrix}\]
avec $a^0_i=a^1_i$, $\alpha^0_{i+1}=\alpha^1_{i+1}$, où pour $\varepsilon=0,1$, $a^\varepsilon_0$ est un objet de $C$ et
\[\begin{matrix}
a^\varepsilon_k:a^0_{k-1}\toto a^1_{k-1}\,,\hfill&0<k\leq i\,,\hfill\cr
\noalign{\vskip 5pt}
\alpha^\varepsilon_k:\alpha^1_{k-1}\toto a^\varepsilon_{k-1}\comp^{}_0\alpha^0_1\comp^{}_1\cdots\comp^{}_{k-2}\alpha^0_{k-1}\,,\kern20pt&0<k\leq i+1\,,
\end{matrix}\] 
sont des $k$\nbd-cellules de $C$, où par convention, on a posé $\alpha^1_0=c$, de sorte que pour $k=1$, on a $\alpha^\varepsilon_1:c\to a^\varepsilon_0$. On posera souvent $a_i=a^0_i=a^1_i$ et $\alpha_{i+1}=\alpha^0_{i+1}=\alpha^1_{i+1}$
Les tableaux correspondant aux diagrammes du début du paragraphe sont respectivement
\[
\begin{pmatrix}
(a_0,\alpha_1)\cr
\noalign{\vskip 5pt}
(a_0,\alpha_1)
\end{pmatrix},
\qquad
\begin{pmatrix}
(a^0_0,\alpha^0_1)&(a_1,\alpha_2)\cr
\noalign{\vskip 5pt}
(a^1_0,\alpha^1_1)&(a_1,\alpha_2)
\end{pmatrix},
\qquad
\begin{pmatrix}
(a^0_0,\alpha^0_1)&(a^0_1,\alpha^0_2)&(a_2,\alpha_3)\cr
\noalign{\vskip 5pt}
(a^1_0,\alpha^1_1)&(a^1_1,\alpha^1_2)&(a_2,\alpha_3)
\end{pmatrix}.
\]
Si $i>0$, la source de la $i$\nbd-cellule $(a,\alpha)$ de $\cotr{C}{c}$ est le tableau
\[
s(a,\alpha)=
\begin{pmatrix}
(a^0_0,\alpha^0_1)&\cdots&(a^0_{i-2},\alpha^0_{i-1})&(a^0_{i-1},\alpha^0_{i})\cr
\noalign{\vskip 5pt}
(a^1_0,\alpha^1_1)&\cdots&(a^1_{i-2},\alpha^1_{i-1})&(a^0_{i-1},\alpha^0_{i})
\end{pmatrix}\phantom{.}\]
et le but le tableau
\[
t(a,\alpha)=
\begin{pmatrix}
(a^0_0,\alpha^0_1)&\cdots&(a^0_{i-2},\alpha^0_{i-1})&(a^1_{i-1},\alpha^1_{i})\cr
\noalign{\vskip 5pt}
(a^1_0,\alpha^1_1)&\cdots&(a^1_{i-2},\alpha^1_{i-1})&(a^1_{i-1},\alpha^1_{i})
\end{pmatrix}.\]
Pour $i\geq0$, l'unité de $(a,\alpha)$ est le tableau 
\[
1_{(a,\alpha)}=
\begin{pmatrix}
(a^0_0,\alpha^0_1)&\cdots&(a^0_{i-1},\alpha^0_i)&(a^0_i,\alpha^0_{i+1})&(1_{a^0_i},1_{\alpha^0_{i+1}})\cr
\noalign{\vskip 5pt}
(a^1_0,\alpha^1_1)&\cdots&(a^1_{i-1},\alpha^1_i)&(a^1_i,\alpha^1_{i+1})&(1_{a^1_i},1_{\alpha^1_{i+1}})
\end{pmatrix}.\]
Pour $j$ un entier, $0\leq j<i$, et 
\[
(b,\beta)=
\begin{pmatrix}
(b^0_0,\beta^0_1)&\cdots&(b^0_{i-1},\beta^0_i)&(b^0_i,\beta^0_{i+1})\cr
\noalign{\vskip 5pt}
(b^1_0,\beta^1_1)&\cdots&(b^1_{i-1},\beta^1_i)&(b^1_i,\beta^1_{i+1})
\end{pmatrix}\]
une deuxième $i$\nbd-cellule, $j$\nbd-composable avec $(a,\alpha)$, autrement dit telle que
\[
a^\varepsilon_k=b^\varepsilon_k\,,\ \alpha^\varepsilon_{k+1}=\beta^\varepsilon_{k+1}\,,\quad\varepsilon=0,1\,,\ 0\leq k<j\,,
\quad\hbox{et}\quad a^0_j=b^1_j\,,\ \alpha^0_{j+1}=\beta^1_{j+1}\,,
\]
la $j$\nbd-composition est définie par le tableau
\[
(a,\alpha)\comp^{}_j(b,\beta)=
\begin{pmatrix}
(b^0_0,\beta^0_1)\kern-3pt&\cdots\kern-3pt&(b^0_{j},\beta^0_{j+1})\kern-3pt&(a^0_{j+1}\comp^{}_jb^0_{j+1},\gamma^0_{j+2})\kern-3pt&\cdots\kern-3pt&(a^0_i\comp^{}_jb^0_i,\gamma_{i+1}^0)\cr
\noalign{\vskip 5pt}
(a^1_0,\alpha^1_1)\kern-3pt&\cdots\kern-3pt&(a^1_{j},\alpha^1_{j+1})\kern-3pt&(a^1_{j+1}\comp^{}_jb^1_{j+1},\gamma^1_{j+2})\kern-3pt&\cdots\kern-3pt&(a^1_i\comp^{}_jb^1_i,\gamma_{i+1}^1)
\end{pmatrix},
\]
où pour $\varepsilon=0,1$,
\[
\begin{aligned}
\gamma^\varepsilon_{j+2}&=a^\varepsilon_{j+1}\comp^{}_0\beta^0_1\comp^{}_1\cdots\comp^{}_{j-1}\beta^0_j\comp^{}_j\beta^\varepsilon_{j+2}\comp^{}_{j+1}\alpha^\varepsilon_{j+2}\,,\cr
\noalign{\vskip 3pt}
\gamma^\varepsilon_{k}&=a^1_{j+1}\comp^{}_0\beta^0_1\comp^{}_1\cdots\comp^{}_{j-1}\beta^0_j\comp^{}_j\beta^\varepsilon_{k}\comp^{}_{j+1}\alpha^\varepsilon_{k}\,,\qquad j+2<k\leq i+1\,.
\end{aligned}
\]
On remarquera que l'information contenue dans les tableaux représentant les
cellules de $\cotr{C}{c}$ est largement redondante. La $i$\nbd-cellule
$(a,\alpha)$ est déjà déterminée par les cellules $\alpha^0_1,\dots,\alpha^0_{i+1}$ et $a^0_i$ de $C$. En effet, pour $0\leq k<i$, on a $a^0_k=s_k(a^0_i)$, $a^1_k=t_k(a^0_i)$ et pour $0<k\leq i$, on a $\alpha^1_k=s(\alpha^0_{k+1})$.   Néanmoins, cette représentation est plus symétrique et rend plus simples les formules exprimant la source, le but, l'unité et la composition des cellules. De plus, elle rappelle  la représentation par tableaux des cellules de $\nu(K)$, pour $K$ un complexe dirigé augmenté, ce qui n'est pas fortuit.
\smallbreak

On a un \oo-foncteur d'oubli $\cotr{C}{c}\to C$, défini par $(a,\alpha)\mapsto a_i=a^0_i=a^1_1$. Pour tout \oo-foncteur $A\to C$, on pose $\cotr{A}{c}=\cotr{C}{c}\times_CA$. Ainsi, les $i$\nbd-cellules de $\cotr{A}{c}$ sont données par des tableaux du même type
\[
(a,\alpha)=
\begin{pmatrix}
(a^0_0,\alpha^0_1)&\cdots&(a^0_{i-1},\alpha^0_i)&(a^0_i,\alpha^0_{i+1})\cr
\noalign{\vskip 5pt}
(a^1_0,\alpha^1_1)&\cdots&(a^1_{i-1},\alpha^1_i)&(a^1_i,\alpha^1_{i+1})
\end{pmatrix},
\]
mais où cette fois-ci, pour $\varepsilon=0,1$, $a^\varepsilon_0$ est un objet de $A$, pour $0<k\leq i$, 
\[
a^\varepsilon_k:a^0_{k-1}\toto a^1_{k-1}
\] 
est une $k$\nbd-cellule de $A$, et pour $0<k\leq i+1$
\[
\alpha^\varepsilon_k:\alpha^1_{k-1}\toto u(a^\varepsilon_{k-1})\comp^{}_0\alpha^0_1\comp^{}_1\cdots\comp^{}_{k-2}\alpha^0_{k-1}
\]
est une $k$-cellule de $C$, toujours en posant par convention $\alpha^1_0=c$, de sorte que pour $k=1$, on a $\alpha^\varepsilon_1:c\to u(a^\varepsilon_0)$. Les formules pour les sources, buts, unités et compositions sont tout à fait analogues à celles pour $\cotr{C}{c}$.
\end{paragr}

\begin{paragr}\label{paragr:def_tr_opl}
Soient $u,v:A\to C$ deux \oo-foncteurs de même source et même but. Une prétransformation oplax de $u$ vers $v$ consiste en la donnée, pour tout $i \ge 0$ et toute $i$-cellule $a$ de $A$, d'une $(i+1)$-cellule
  \[
    \alpha_a :
    \alpha_{t_{i-1}(a)} \comp^{}_{i-1} \cdots \comp^{}_1 \alpha_{t_0(a)} \comp^{}_0 u(a)
    \toto
    v(a) \comp^{}_0 \alpha_{s_0(a)} \comp^{}_1 \cdots \comp^{}_{i-1}
    \alpha_{s_{i-1}(a)}
  \]
  de $C$.
Ainsi, si $a$ est un objet de $A$, on dispose d'une $1$-cellule
  \[
    \xymatrix{
      u(a) \ar[d]_{\alpha_a} \\ v(a)
    }
  \]
  de $C$; si $a$ est une $1$-cellule de $A$, on dispose d'une $2$-cellule
  \[
    \shorthandoff{;:}
    \xymatrix{
      u(s_0(a)) \ar[d]_{\alpha_{s_0(a)}} \ar[r]^{u(a)} &
      u(t_0(a)) \ar[d]^{\alpha_{t_0(a)}} \\
      v(s_0(a)) \ar[r]_{v(a)} &
      v(t_0(a))
      \ar@{}[u];[l]_(.35){}="x"
      \ar@{}[u];[l]_(.65){}="y"
      \ar@2"x";"y"_{\alpha_a}
    }
  \]
  de $C$; si $a$ est une $2$-cellule de $A$, on dispose d'une $3$-cellule
  \[
    \shorthandoff{;:}
    \xymatrix@C=4pc@R=5pc{
      u(s_0(a))
      \ar@/^2ex/[r]^{u(s_1(a))}_{}="0"
      \ar@/_2ex/[r]_(0.60){u(t_1(a))}_{}="1"
      \ar[d]_{}="f"_{\alpha_{s_0(a)}}
      \ar@2"0";"1"_{u(a)}
      &
      u(t_0(a))
      \ar[d]^{\alpha_{t_0(a)}} \\
      v(s_0(a))
      \ar@{.>}@/^2ex/[r]^(0.40){v(s_1(a))}_{}="0"
      \ar@/_2ex/[r]_{v(t_1(a))}_{}="1"
      \ar@2{:>}"0";"1"_{v(a)}
      &
      v(t_0(a))
      \ar@{}[u];[l]_(.40){}="x"
      \ar@{}[u];[l]_(.60){}="y"
      \ar@<-1ex>@/_1.5ex/@{:>}"x";"y"_(0.60){\alpha_{s_1(a)}}_{}="0"
      \ar@<1ex>@/^1.5ex/@2"x";"y"^(0.40){\alpha_{t_1(a)}}_{}="1"
      \ar@{}"1";"0"_(.05){}="z"
      \ar@{}"1";"0"_(.95){}="t"
      \ar@3"z";"t"_{\alpha_a}
    }
  \]
  de $C$ de source $\alpha_{t_1(a)} \comp^{}_1 (\alpha_{t_0(a)} \comp^{}_0 u(a))$ et
  de but $(v(a) \comp^{}_0 \alpha_{s_0(a)}) \comp^{}_1 \alpha_{s_1(a)}$; etc.
\smallbreak

  Une telle prétransformation oplax est une \ndef{transformation oplax}
  si elle satisfait aux axiomes de fonctorialité suivants :
  \begin{enumerate}
    \item[(\emph{a})] pour tout $i \ge 0$ et toute $i$-cellule $a$ de $A$, on a
      \[ \alpha^{}_{1_{a}} = 1_{\alpha^{}_a}; \]
    \item[(\emph{b})] pour tous $i > j \ge 0$ et tout couple $a, b$ de $i$-cellules
      $j$-composables de $A$, on a
      \[
        \begin{split}
          \alpha_{a \comp^{}_j b} & =
          \left(v(t_{j+1}(a)) \comp^{}_0 \alpha_{s_0(b)} \comp^{}_1 \cdots
          \comp^{}_{j-1} \alpha_{s_{j-1}(b)} \comp^{}_j \alpha_b\right) \\
          & \phantom{=1} \qquad
          \comp^{}_{j+1} \left(\alpha_a \comp^{}_j \alpha_{t_{j-1}(a)} \comp^{}_{j-1}
          \cdots \comp^{}_1 \alpha_{t_0(a)} \comp^{}_0 u(s_{j+1}(b))\right).
        \end{split}
      \]
  \end{enumerate}

Si $\alpha$ est une transformation oplax de $u$ vers $v$ et $f:A'\to A$ un \oo-foncteur, on vérifie immédiatement qu'on définit une transformation oplax $\alpha\comp^{} f$ de $uf$ vers $vf$ en posant pour $a$ une cellule de $A'$, $(\alpha\comp^{} f)_a=\alpha_{f(a)}$. De même, pour tout \oo-foncteur $g:C\to C'$, on définit une transformation oplax $g\comp^{}\alpha$ de $gu$ vers $gv$ en posant pour $a$ une cellule de $A$, $(g\comp^{}\alpha)_a=g(\alpha_a)$. 
\end{paragr}

La proposition suivante résulte de \cite[corollaire B.5.3]{joint}. On en donne ici l'esquisse d'une preuve élémentaire.

\begin{prop}\label{prop:cotr_et_tr_opl}
Soient $u:A\to C$ un \oo-foncteur et $c$ un objet de $C$. Pour toute \oo-catégorie $X$, on a une bijection canonique naturelle de l'ensemble des \oo-foncteurs de $X$ vers $\cotr{A}{c}$ sur l'ensemble des couples $(a,\alpha)$ formés d'un \oo-foncteur $a:X\to A$ et d'une transformation oplax $\alpha:c\to ua$, où $c$ désigne aussi le \oo-foncteur constant de $X$ vers $C$ de valeur $c$.
\end{prop}

\begin{proof}
Soit $a:X\to A$ un \oo-foncteur. Une prétransformation de $c$ vers $ua$ consiste en la donnée, pour tout $i\geq0$ et toute $i$\nbd-cellule $x$ de $X$, d'une $(i+1)$\nbd-cellule 
\[
\alpha_x :
\alpha_{t_{i-1}(x)} \comp^{}_{i-1} \cdots \comp^{}_1 \alpha_{t_0(x)} \comp^{}_0 c(x)
\toto
ua(x) \comp^{}_0 \alpha_{s_0(x)} \comp^{}_1 \cdots \comp^{}_{i-1}
\alpha_{s_{i-1}(x)}
\]
de $C$. Si $i=0$, la source de la $1$\nbd-cellule $\alpha_x$ est $c$ et son but $ua(x)$. 
Pour $i>0$, comme $c$ est le foncteur constant de valeur $c$, la source de la $(i+1)$\nbd-cellule $\alpha_x$ est la $i$\nbd-cellule 
$
\alpha_{t_{i-1}(x)} \comp^{}_{i-1}\alpha_{t_{i-2}(x)}\comp^{}_{i-2} \cdots \comp^{}_1 \alpha_{t_0(x)} \comp^{}_0 1^i_c\,
$
de $C$, où $1^i_c$ désigne la $i$\nbd-cellule unité itérée de l'objet $c$. Or, la $i$\nbd-cellule $\alpha_{t_{i-2}(x)}\comp^{}_{i-2} \cdots \comp^{}_1 \alpha_{t_0(x)} \comp^{}_0 1^i_c$ est l'unité d'une $(i-1)$\nbd-cellule, et par suite, la source de $\alpha_x$ est égale à $\alpha_{t_{i-1}(x)}$. Ainsi, $\alpha_x$ est une $(i+1)$\nbd-cellule
\[
\alpha_x :
\alpha_{t_{i-1}(x)} \toto ua(x) \comp^{}_0 \alpha_{s_0(x)} \comp^{}_1 \cdots \comp^{}_{i-1}\alpha_{s_{i-1}(x)}\,.
\]

Par définition, pour que la prétransformation oplax $\alpha$ soit une transformation oplax, il faut et il suffit, d'une part, que pour tout $i\geq0$ et toute $i$\nbd-cellule $x$ de $X$, on ait $\alpha_{1_x}=1_{\alpha_x}$, et d'autre part, que pour tous $i>j\geq0$ et tout couple $x, y$ de $i$-cellules
$j$-composables de $X$, on ait
\[
\begin{split}
\alpha_{x \comp^{}_j y} & = \left(ua(t_{j+1}(x)) \comp^{}_0 \alpha_{s_0(y)} \comp^{}_1 \cdots\comp^{}_{j-1} \alpha_{s_{j-1}(y)} \comp^{}_j \alpha_y\right) \\
& \phantom{=1} \qquad \comp^{}_{j+1} \left(\alpha_x \comp^{}_j \alpha_{t_{j-1}(x)} \comp^{}_{j-1}\cdots \comp^{}_1 \alpha_{t_0(x)} \comp^{}_0 1^{j+1}_c\right).
\end{split}
\]
Or, la $(j+1)$\nbd-cellule $\alpha_{t_{j-1}(x)} \comp^{}_{j-1}\cdots \comp^{}_1 \alpha_{t_0(x)} \comp^{}_0 1^{j+1}_c$ est l'unité d'une $j$\nbd-cellule de $C$, et par suite, cette deuxième condition est équivalente à
\[
\alpha_{x \comp^{}_j y} = ua(t_{j+1}(x)) \comp^{}_0 \alpha_{s_0(y)} \comp^{}_1 \cdots\comp^{}_{j-1} \alpha_{s_{j-1}(y)} \comp^{}_j \alpha_y \comp^{}_{j+1} \alpha_x\,.
\]

Il résulte alors facilement de ces considérations qu'en associant, pour $i\geq0$, à toute $i$\nbd-cellule $x$ de $X$, le tableau
\[
\begin{pmatrix}
\left(a(s_0(x)),\,\alpha_{s_0(x)}\right)&\cdots&\left(a(s_{i-1}(x)),\,\alpha_{s_{i-1}(x)}\right)&\bigl(a(x),\,\alpha_x\bigr)\cr
\noalign{\vskip 5pt}
\left(a(t_0(x)),\,\alpha_{t_0(x)}\right)&\cdots&\left(a(t_{i-1}(x)),\,\alpha_{t_{i-1}(x)}\right)&\bigl(a(x),\,\alpha_x\bigr)
\end{pmatrix},
\]
on définit un \oo-foncteur $X\to\cotr{A}{c}$, et on établit ainsi une bijection de l'ensemble des couples $(a,\alpha)$, formés d'un \oo-foncteur $a:X\to A$ et d'une transformation oplax $\alpha:c\to ua$, sur l'ensemble des \oo-foncteurs de $X$ vers $\cotr{A}{c}$. La naturalité en $X$ de cette bijection est évidente.
\end{proof}

\begin{prop}\label{prop:tr_opl_hmtp}
Soient $u,v:A\to C$ deux \oo-foncteurs de même source et même but. On a une bijection canonique naturelle entre l'ensemble des transformations oplax de $u$ vers $v$ et l'ensemble des \oo-foncteurs $h:\Dsk\otimes A\to C$ rendant commutatif le diagramme suivant
\[
\xymatrixrowsep{1.4pc}
\xymatrixcolsep{.4pc}
\xymatrix{
\{0\}\otimes A\ar@{}[r]|-{\textstyle\simeq}\ar@{^{(}->}@<-1ex>[rrd]
&A\ar@/^2ex/[rrrrrd]^u
\\
&&\Dsk\otimes A\ar[rrrr]^h
&&&&C
\\
\{1\}\otimes A\ar@{}[r]|-{\textstyle\simeq}\ar@{^{(}->}[rru]
&A\ar@/_2ex/[rrrrru]_v
&&&&&&.}
\]
\end{prop}

\begin{proof}
Voir \cite[corollaire B.2.6]{joint}.
\end{proof}

On notera souvent par la même lettre une transformation oplax et le \oo-foncteur correspondant par la proposition précédente. Avec cet abus de notation, si $\alpha$ est une transformation oplax entre deux \oo-foncteurs de source $A$ et de but $C$, et si $a:A'\to A$ et $c:C\to C'$ sont des \oo-foncteurs, on a les égalités $\alpha\comp a=\alpha(1_{\Dsk}\otimes a)$ et $c\comp\alpha=c\alpha$, où dans les membres de gauche $\alpha$ est vue comme une transformation oplax et dans les membres de droite comme un \oo-foncteur.

\begin{paragr}\label{paragr:def_comp_vert}
Soient $u,v,w:A\to C$ trois \oo-foncteurs de même source et même but, $\alpha$~une transformation oplax de $u$ vers $v$ et $\beta$ une transformation oplax de $v$ vers $w$. On définit une transformation oplax $\beta\alpha$ de $u$ vers $w$, \ndef{composé vertical} de $\beta$ avec $\alpha$, comme suit. On forme dans $\ooCat$ la somme amalgamée $\Dsk\amalg_{\Deltan{0}}\Dsk$, où $\Deltan{0}$ s'envoie dans le terme de gauche (resp. de droite) de la somme amalgamée par $0\mapsto0$ (resp. par $0\mapsto1$). On définit $\delta:\Dsk\to\Dsk\amalg_{\Deltan{0}}\Dsk$ en envoyant $0$ (resp. $1$) sur l'objet $0$ du terme de droite (resp.~sur l'objet $1$ du terme de gauche) et on pose $\beta\alpha=(\beta,\alpha)(\delta\otimes1_A)$
\[
\xymatrix{
\Dsk\otimes A\ar[r]^-{\delta\otimes1_A}
&(\Dsk\amalg_{\Deltan{0}}\Dsk)\otimes A\simeq(\Dsk\otimes A)\amalg_A(\Dsk\otimes A)\ar[r]^-{(\beta,\alpha)}
&C\ .
}
\]
On remarquera qu'en vertu du théorème~\ref{thm:Steiner}, de la proposition~\ref{prop:amalg_Stf} et des paragraphes~\ref{paragr:def_cX} et~\ref{paragr:def_nrf_Str}, \hbox{$\delta:\Dsk\to\Dsk\amalg_{\Deltan{0}}\Dsk$} s'identifie à l'image par le foncteur $\nu$ du morphisme de complexes dirigés augmentés $\lambda(\delta):\cn \Deltan{1}\to \cn \Deltan{1}\amalg_{\cn \Deltan{0}}\cn \Deltan{1}$.
\smallbreak

On se gardera de croire que cette composition verticale et les compositions d'une transformation oplax à gauche ou à droite par un \oo-foncteur (voir le paragraphe~\ref{paragr:def_tr_opl}) satisfont à la \og règle de Godement\fg{}. Autrement dit, si
\[
\xymatrixcolsep{.9pc}
\xymatrix{
A\ar@/_1pc/[rrr]_g\ar@/^1pc/[rrr]^f
&\rtwocell<>{\kern3pt\alpha}
&&B\ar@/_1pc/[rrr]_k\ar@/^1pc/[rrr]^h
&\rtwocell<>{\kern3pt\beta}
&&C
}
\]
est un diagramme de \oo-catégories, \oo-foncteurs et transformations oplax, on \emph{n'a pas} en général l'égalité $(k\comp\alpha)(\beta\comp f)=(\beta\comp g)(h\comp\alpha)$. En particulier, le composé horizontal $\beta\comp\alpha$ \emph{n'est pas} défini. Ainsi, les \oo-catégories, \oo-foncteurs et transformations oplax \emph{ne forment pas} une $2$\nbd-catégorie, mais seulement ce qu'on appelle une sesquicatégorie.
\end{paragr}

\begin{prop}\label{prop:iso_comp_cotr}
Soient $u:A\to C$ un \oo-foncteur et $c$ un objet de $C$. On a un isomorphisme canonique d'ensembles simpliciaux $\Nrf(\cotr{A}{c})\simeq\cotr{\Nrf(A)}{c}$, où $c$ désigne aussi le $0$\nbd-simplexe de $\Nrf(C)$ correspondant au \oo-foncteur, de source la \oo-catégorie ponctuelle $\On{0}$ et de but $C$, défini par l'objet $c$ de $C$.
\end{prop}

\begin{proof}
Pour commencer, on va définir, pour tout $n\geq0$, une application \hbox{$\theta_n:(\cotr{\Nrf(A)}{c})_n\to(\Nrf(\cotr{A}{c}))_n$} comme suit. Soit $(c',a)$
\[
c':\On{1+n}\to C\,,\quad a:\On{n}\to A\,,\qquad c'_0=c\,,\quad c'_{1,\dots,n}=ua\,,
\]
un $n$\nbd-simplexe de $\cotr{\Nrf(A)}{c}$. On va définir un morphisme de complexes dirigés augmentés
\[
\pi_n:\cn \Deltan{1}\otimes \cn \Deltan{n}\to \cn \Deltan{1+n}\,,
\]
d'où en vertu du théorème~\ref{thm:produit_Gray}, de la proposition~\ref{prop:produit_Stf}, et des paragraphes~\ref{paragr:def_cX} et~\ref{paragr:def_nrf_Str}, un \oo-foncteur
\[
\nu(\pi_n):\Dsk\otimes\On{n}\to\On{1+n}\,.
\]
On montrera que le composé $c'\nu(\pi_n):\Dsk\otimes\On{n}\to C$ définit une transformation oplax du \oo-foncteur constant $c:\On{n}\to C$ vers le \oo-foncteur $ua$, de sorte que le couple $(a,c'\nu(\pi_n))$ correspondra, en vertu de la proposition~\ref{prop:cotr_et_tr_opl}, à un \oo-foncteur $\On{n}\to\cotr{A}{c}$, autrement dit, à un $n$\nbd-simplexe de $\Nrf(\cotr{A}{c})$. Par définition, ce $n$-simplexe sera l'image de $(c',a)$ par $\theta_n$. Pour conclure, il restera à prouver que l'application $\theta_n$ est bijective, et que la famille des $\theta_n$, pour $n\geq0$, est un morphisme d'ensembles simpliciaux.
\smallbreak

Définissons le morphisme de complexes dirigés augmentés $\pi_n$. Le complexe dirigé augmenté $\cn \Deltan{1+n}$ admet comme base l'ensemble gradué $E$,
\[
E_p = \{(i_0, \dots, i_p) \mid 0 \le i_0 < \cdots < i_p \le 1+n\}\,,\quad p\geq0\,,
\]
et $\cn \Deltan{n}$, identifié à un sous-complexe dirigé augmenté de $\cn \Deltan{1+n}$ par le morphisme $\cn (j_{0,n})$ (voir les notations~\ref{notation:simili_joint}), admet comme base le sous-ensemble gradué $E'$,
\[
E'_p = \{(i_0, \dots, i_p) \mid  1\leq i_0<\cdots<i_p\leq 1+n \}\,,\quad p\geq0\,,
\]
de $E$. La base de $\cn \Deltan{1}\otimes \cn \Deltan{n}$ est formée des éléments
\[
(0)\otimes(i_0, \dots, i_p)\,,\qquad (1)\otimes(i_0, \dots, i_p)\,,\qquad (0,1)\otimes(i_0, \dots, i_p)\,,
\]
pour $1\leq i_0<\cdots<i_p\leq 1+n$, $p\geq0$, les deux premiers étant des $p$\nbd-chaînes et le troisième une $(p+1)$\nbd-chaîne.
On définit $\pi_n$ par les formules
\[
\begin{matrix}
(1)\qquad\quad&\pi_n((0)\otimes(i_0, \dots, i_p))=\left\{\begin{matrix}
(0)&\quad\hbox{si}\ \ p=0\,,\cr
\noalign{\vskip3pt}
0&\quad\hbox{si}\ \ p>0\,,
\end{matrix}\right.\cr
(2)\qquad\quad&\pi_n((1)\otimes(i_0, \dots, i_p))=(i_0, \dots, i_p)\,,\hfill\cr
\noalign{\vskip6pt}
(3)\qquad\quad&\pi_n((0,1)\otimes(i_0, \dots, i_p))=(0,i_0, \dots, i_p)\,,\hfill&\qquad\quad
\end{matrix}
\]
pour $1\leq i_0<\cdots<i_p\leq 1+n$. \'Etant donné des entiers $0\leq i_0\leq\cdots\leq i_p\leq n$, on observe que s'il existe $l$ tel que $0\leq l<p$ et tel que  $i_l=i_{l+1}$, les formules définissant $\pi_n$ sont compatibles avec la convention $(i_0,\dots,i_p)=0$ (voir le paragraphe~\ref{paragr:def_cX}). La compatibilité de $\pi_n$ aux augmentations et aux sous-monoïdes de positivité est évidente, ainsi que celle aux différentielles dans les cas $(1)$ et $(2)$. Pour montrer que $\pi_n$ est un morphisme de complexes dirigés augmentés, il reste donc à vérifier que
\[
\pi_nd((0,1)\otimes(i_0, \dots, i_p))=d\pi_n((0,1)\otimes(i_0, \dots, i_p))\,.
\]
Pour $p=0$, on a 
\[
\pi_nd((0,1)\otimes(i_0))=\pi_n((1)\otimes(i_0)-(0)\otimes(i_0))=(i_0)-(0)=d(0,i_0)=d\pi_n((0,1)\otimes(i_0))\,,
\]
et pour $p\geq1$, on a
\[\begin{aligned}
\pi_nd((0,1)\otimes(i_0, \dots, i_p))&=\pi_n\Bigl((1)\otimes(i_0, \dots, i_p)-(0)\otimes(i_0, \dots, i_p)\cr
\noalign{\vskip-5pt}
&\phantom{=\pi_n\Bigl(}\kern3pt-\textstyle\sum\limits_{k=0}^p(-1)^k(0,1)\otimes(i_0,\dots,\hat i_k,\dots,i_p)\Bigr)\cr
&=(i_0, \dots, i_p)-\textstyle\sum\limits_{k=0}^p(-1)^k(0,i_0,\dots,\hat i_k,\dots,i_p)\cr
\noalign{\vskip3pt}
&=d(0,i_0, \dots, i_p)=d\pi_n((0,1)\otimes(i_0, \dots, i_p))\,.
\end{aligned}\]
Le fait que le composé $c'\nu(\pi_n):\Dsk\otimes\On{n}\to C$ définit une transformation oplax du \oo-foncteur constant $c:\On{n}\to C$ vers le \oo-foncteur $ua$ résulte aussitôt des formules $(1)$ et $(2)$ ci-dessus et des égalités $c'_0=c$ et $c'_{1,\dots,1+n}=ua$.
\smallbreak

D'autre part, pour tout morphisme $\psi:\Deltan{n'}\to\Deltan{n}$, il est immédiat, dans les notations de~\ref{notation:simili_joint} (et en tenant compte de l'observation qui suit la définition de $\pi_n$), que le carré
\[
\xymatrixcolsep{3pc}
\xymatrix{
\cn \Deltan{1}\otimes \cn \Deltan{n'}\ar[d]_{1_{\cn \Deltan{1}}\otimes\psi}\ar[r]^-{\pi_{n'}}
&\cn \Deltan{1+n'}\ar[d]^{\cn (1_{\Deltan{0}}\amalg\psi)}
\\
\cn \Deltan{1}\otimes \cn \Deltan{n}\ar[r]_-{\pi_{n}}
&\cn \Deltan{1+n}
}
\]
est commutatif, ce qui implique facilement que la famille formée des $\theta_n$, $n\geq0$, est un morphisme d'ensembles simpliciaux.
\smallbreak

Il reste à montrer que les applications $\theta_n$ sont bijectives. Cela revient à montrer que pour tout $n\geq0$ et tout couple $(a,\alpha)$, formé d'un \oo-foncteur $a:\On{n}\to A$ et d'une transformation oplax $\alpha:\Dsk\otimes\On{n}\to C$ du \oo-foncteur constant $c:\On{n}\to C$ vers $ua$, il existe un unique \oo-foncteur $c':\On{1+n}\to C$ rendant commutatif le triangle
\[
\raise 20pt
\vbox{
\xymatrix{
\Dsk\otimes\On{n}\ar[r]^-{\alpha}\ar[d]_{\nu(\pi_n)}
&C
\\
\On{1+n}\ar@{-->}[ru]_{c'}
}}\leqno(*)
\]
(les égalités $c'_0=c$ et $c'_{1,\dots,1+n}=ua$ étant alors automatiques). On va le démontrer en utilisant le théorème~\ref{thm:Steiner_pol}, qui implique que la \oo-catégorie $\On{1+n}$ est engendrée librement au sens des polygraphes par ses atomes 
\[
\atom{(j_0,\dots,j_q)}\,,\qquad0\leq j_0<\cdots<j_q\leq1+n\,,\quad q\geq0.
\]

Pour commencer, on observe qu'on a les égalités
\[
\begin{matrix}
(1')\qquad\quad&\nu(\pi_n)\atom{(0)\otimes(i_0, \dots, i_p)}=\left\{\begin{matrix}
\atom{(0)}&\quad\hbox{si}\ \ p=0\,,\cr
\noalign{\vskip3pt}
1^p_{\atom{(0)}}&\quad\hbox{si}\ \ p>0\,,
\end{matrix}\right.\cr
\noalign{\vskip3pt}
(2')\qquad\quad&\nu(\pi_n)\atom{(1)\otimes(i_0, \dots, i_p)}=\atom{(i_0, \dots, i_p)}\,,\hfill\cr
\noalign{\vskip6pt}
(3')\qquad\quad&\nu(\pi_n)\atom{(0,1)\otimes(i_0, \dots, i_p)}=\atom{(0,i_0, \dots, i_p)}\,,\hfill&\qquad\quad
\end{matrix}
\]
pour $p\geq0$ et $1\leq i_0<\cdots<i_p\leq 1+n$. Les deux premières sont immédiates. Pour prouver la troisième, on remarque qu'une récurrence descendante montre immédiatement que, pour $0\leq r\leq p$, on a dans $\cn \Deltan{1}\otimes \cn \Deltan{n}$
\[
\begin{aligned}
&\atom{(0,1)\otimes(i_0, \dots, i_p)}^0_r=(0)\otimes\atom{(i_0, \dots, i_p)}^0_r+(0,1)\otimes\atom{(i_0, \dots, i_p)}^1_{r-1}\,,\cr
\noalign{\vskip3pt}
&\atom{(0,1)\otimes(i_0, \dots, i_p)}^1_r=(1)\otimes\atom{(i_0, \dots, i_p)}^1_r+(0,1)\otimes\atom{(i_0, \dots, i_p)}^0_{r-1}\,
\end{aligned}
\]
(avec, pour $r=0$, la convention $\atom{(i_0, \dots, i_p)}^0_{-1}=\atom{(i_0, \dots, i_p)}^1_{-1}=0$). De même, une autre récurrence descendante montre que dans $\cn \Deltan{1+n}$, on a, pour $1\leq r\leq p$, 
\[
\begin{aligned}
&\atom{(0,i_0, \dots, i_p)}^0_r=(0,\atom{(i_0, \dots, i_p)}^1_{r-1})\,,\cr
\noalign{\vskip3pt}
&\atom{(0,i_0, \dots, i_p)}^1_r=\atom{(i_0, \dots, i_p)}^1_r+(0,\atom{(i_0, \dots, i_p)}^0_{r-1})\,,
\end{aligned}
\]
où pour $x$ une $r$-chaîne de $\cn \Deltan{1+n}$ de la forme 
\[
x=\textstyle\sum\limits_{1\leq k_0<\cdots<k_r\leq1+n}x_{k_0,\dots,k_r}(k_0,\dots,k_r)\,,
\]
on a posé
\[
(0,x)=\textstyle\sum\limits_{1\leq k_0<\cdots<k_r\leq1+n}x_{k_0,\dots,k_r}(0,k_0,\dots,k_r)\,,
\]
et pour $r=0$, $\atom{(0,i_0, \dots, i_p)}^0_0=(0)$ et $\atom{(0,i_0, \dots, i_p)}^1_0=\atom{(i_0, \dots, i_p)}^1_0$.
L'égalité~$(3')$ résulte alors des formules définissant le morphisme $\pi_n$, en tenant compte, pour $r=0$, du fait que la base de $\cn \Deltan{1+n}$ est unitaire.
\smallbreak

On remarque que les égalités $(1')$, $(2')$, $(3')$ impliquent que tout atome
\[
\atom{(j_0,\dots,j_q)}\,,\qquad0\leq j_0<\cdots<j_q\leq1+n\,,\quad q\geq0\,,
\]
de $\On{1+n}$ est l'image par $\nu(\pi_n)$ d'un atome de $\Dsk\otimes\On{n}$. Plus précisément, on a
\[
\begin{matrix}
\atom{(0)}=\nu(\pi_n)\atom{(0)\otimes(1)}\,,\hfill&\hbox{si $j_0=0$ et $q=0$,}\cr
\noalign{\vskip3pt}
\atom{(j_0,\dots,j_q)}=\nu(\pi_n)\atom{(0,1)\otimes(j_1, \dots, j_q)}\,,&\hbox{si $j_0=0$ et  $q>0$,}\cr
\noalign{\vskip3pt}
\atom{(j_0,\dots,j_q)}=\nu(\pi_n)\atom{(1)\otimes(j_0, \dots, j_q)}\,,\hfill&\hbox{si $j_0>0$ et $q\geq0$.}\hfill
\end{matrix}
\]
Cela prouve déjà qu'il y a au plus un seul \oo-foncteur $c'$ rendant commutatif le triangle~$(*)$,
et permet aussi d'associer à tout atome $y$ de $\On{1+n}$ une cellule $f(y)$ de $C$, image par le \oo-foncteur $\alpha$ d'un atome $x$ de $\Dsk\otimes\On{n}$ tel que $\nu(\pi_n)(x)=y$. Par exemple on peut poser
\[
\begin{matrix}
f\atom{(0)}=\alpha\atom{(0)\otimes(1)}\,,\hfill&\hbox{si $j_0=0$ et $q=0$,}\cr
\noalign{\vskip3pt}
f\atom{(j_0,\dots,j_q)}=\alpha\atom{(0,1)\otimes(j_1, \dots, j_q)}\,,&\hbox{si $j_0=0$ et  $q>0$,}\cr
\noalign{\vskip3pt}
f\atom{(j_0,\dots,j_q)}=\alpha\atom{(1)\otimes(j_0, \dots, j_q)}\,,\hfill&\hbox{si $j_0>0$ et $q\geq0$.}\hfill
\end{matrix}
\]
Pour montrer l'existence d'un \oo-foncteur $c'$ rendant commutatif le triangle $(*)$, on va construire, par récurrence sur $q\geq0$, une suite de \oo-foncteurs $c'_q:\tb{q}(\On{1+n})\to C$ tels que $c'_q\tb{q}(\nu(\pi_n))=\alpha|\tb{q}(\Dsk\otimes\On{n})$ et $c'_{q+1}|\tb{q}(\On{1+n})=c'_q$, ce qui prouvera l'assertion.
\smallbreak

Pour $q=0$, on pose $c'\atom{(j_0)}=f\atom{(j_0)}$, pour $0\leq j_0\leq1+n$. Comme $\alpha$ est une transformation oplax de source le \oo-foncteur constant de valeur $c$, pour $1\leq i_0\leq1+n$, on a $\alpha\atom{(0)\otimes(i_0)}=c$, et par suite
\[
c'_0\nu(\pi_n)\atom{(0)\otimes(i_0)}=c'_0\atom{(0)}=f\atom{(0)}=\alpha\atom{(0)\otimes(1)}=c=\alpha\atom{(0)\otimes(i_0)}\,.
\]
D'autre part, on a
\[
c'_0\nu(\pi_n)\atom{(1)\otimes(i_0)}=c'_0\atom{(i_0)}=f\atom{(i_0)}=\alpha\atom{(1)\otimes(i_0)}\,,
\]
ce qui prouve que $c'_0\tb{0}(\nu(\pi_n))=\alpha|\tb{0}(\Dsk\otimes\On{n})$. Supposons maintenant que $c'_q$ soit construit. On remarque que l'application associant à un atome $y$ de dimension $q+1$ de $\On{1+n}$ la $(q+1)$\nbd-cellule $f(y)$ de $C$ est compatible à la formation des sources et de buts au sens du paragraphe~\ref{paragr:def_pol}. En effet, il existe un atome $x$ de $\Dsk\otimes\On{n}$ tel que $y=\nu(\pi_n)(x)$ et $f(y)=\alpha(x)$, et par suite, on a
\[\begin{aligned}
&s(f(y))=s(\alpha(x))=\alpha(s(x))=c'_q\nu(\pi_n)(s(x))=c'_q(s(\nu(\pi_n)(x)))=c'_q(s(y))\,,\cr
&t(f(y))=t(\alpha(x))=\alpha(t(x))=c'_q\nu(\pi_n)(t(x))=c'_q(t(\nu(\pi_n)(x)))=c'_q(t(y))\,.
\end{aligned}\]
On en déduit l'existence d'un unique \oo-foncteur $c'_{q+1}:\On{1+n}\to C$ tel que $c'_{q+1}|\tb{q}(\On{1+n})=c'_q$ et tel que pour tout atome $y$ de dimension $q+1$ de $\On{1+n}$, on ait $c'_{q+1}(y)=f(y)$. Il reste à prouver l'égalité $c'_{q+1}\tb{q+1}(\nu(\pi_n))=\alpha|\tb{q+1}(\Dsk\otimes\On{n})$. Or, en vertu du théorème~\ref{thm:Steiner_pol}, de la proposition~\ref{prop:produit_Stf}, et du paragraphe~\ref{paragr:def_cX}, la \oo-catégorie $\Dsk\otimes\On{n}$ est engendrée librement au sens des polygraphes par ses atomes. Il suffit donc de montrer que pour tout atome $x$ de dimension $q+1$ de $\Dsk\otimes\On{n}$, on a $c'_{q+1}\nu(\pi_n)(x)=\alpha(x)$. Le seul cas non trivial est celui des atomes de la forme $x=\atom{(0)\otimes(i_0,\dots,i_{q+1})}$, pour $1\leq i_0<\cdots<i_{q+1}\leq n$. Comme  $\alpha$ est une transformation oplax de source le \oo-foncteur constant de valeur $c$, on a $\alpha(x)=1^{q+1}_c$. D'autre part, en vertu de l'égalité $(1')$, on a $\nu(\pi_n)(x)=1^{q+1}_{\atom{(0)}}$, et par suite
\[
c'_{q+1}\nu(\pi_n)(x)=c'_{q+1}(1^{q+1}_{\atom{(0)}})=1^{q+1}_{c'_0\atom{(0)}}=1^{q+1}_c=\alpha(x)\,,
\]
ce qui achève la démonstration.
\end{proof}

\section{Le théorème A \pdfoo-catégorique}\label{section:ThA_ooCat}

\begin{paragr}\label{paragr_def_cotrT}
Soient $A,B,C$ trois \oo-catégories, $u:A\to B$, $v:A\to C$, $w:B\to C$ des \oo-foncteurs, et $\alpha:v\to wu$ une transformation oplax, formant un triangle
\[\mathcal{T}=
\raise 25pt
\vbox{
    \shorthandoff{;}
    \xymatrix@C=1.5pc{
      A \ar[rr]^u \ar[dr]_(.45){v}_{}="f" & & B \ar[dl]^(.45){w} \\
      & C
      \ar@{}"f";[ur]_(.15){}="ff"
      \ar@{}"f";[ur]_(.55){}="oo"
      \ar@<-0.5ex>@2"ff";"oo"^{\alpha}
      &
    }
}
\]
dans $\ooCat$, commutatif à transformation oplax donnée près. Pour tout objet $c$ de $C$, on définit un \oo-foncteur
\[
\cotr{\mathcal{T}}{c}:\cotr{A}{c}\to\cotr{B}{c}
\]
comme suit. Pour $T$ une \oo-catégorie, se donner un \oo-foncteur de $T$ vers $\cotr{A}{c}$ revient à se donner un couple $(t,\tau)$, où $t:T\to A$ est un \oo-foncteur, et $\tau:c\to vt$ une transformation oplax de source le \oo-foncteur constant de $T$ vers $C$ de valeur $c$ et de but le \oo-foncteur composé $vt$ (voir proposition~\ref{prop:cotr_et_tr_opl}). On en déduit un couple $(ut,(\alpha\comp t)\tau)$
\[\xymatrixcolsep{3.2pc}\xymatrixrowsep{2.8pc}
    \shorthandoff{;}
    \xymatrix{
      T \ar[r]^t \ar[dr]_{}="g"_(.50){c}
      & A \ar[r]^{u}_(.75){}="fp" \ar[d]_(.70){}="gp"_(.50){v} & B
      \ar[dl]_{}="gpp"^(.50){w} \\
      & C
      \ar@{}"g";[u]_(0.10){}="x"
      \ar@{}"g";[u]_(.85){}="y"
      \ar@<-0.1ex>@2"x";"y"^(.50){\tau}
      \ar@{}"gp";"fp"_(.25){}="x2"
      \ar@{}"gp";"fp"_(.75){}="y2"
      \ar@<0.4ex>@2"x2";"y2"^{\alpha}
      &
    }
  \]
correspondant à un \oo-foncteur de $T$ vers $\cotr{B}{c}$. On définit ainsi une application
\[
\Hom_{\ooCat}(T,\cotr{A}{c})\to\Hom_{\ooCat}(T,\cotr{B}{c})\,,
\]
naturelle en $T$, d'où en vertu du lemme de Yoneda, un \oo-foncteur $\cotr{\mathcal{T}}{c}:\cotr{A}{c}\to\cotr{B}{c}$.
\end{paragr}

\begin{thm}\label{thm:Th_A_oocat}
Soit $\mathcal{T}$ un triangle dans $\ooCat$, commutatif à transformation oplax donnée près:
\[\mathcal{T}=
\raise 25pt
\vbox{
    \shorthandoff{;}
    \xymatrix@C=1.5pc{
      A \ar[rr]^u \ar[dr]_(0.45){v}_{}="f" & & B \ar[dl]^(0.45){w} \\
      & C
      \ar@{}"f";[ur]_(.15){}="ff"
      \ar@{}"f";[ur]_(.55){}="oo"
      \ar@<-0.5ex>@2"ff";"oo"^{\alpha}
      &.
    }
}
\]
Si pour tout objet $c$ de $C$, le \oo-foncteur $\cotr{\mathcal{T}}{c}:\cotr{A}{c}\to\cotr{B}{c}$ est une équivalence faible, alors il en est de même de $u$.
\end{thm}

\begin{rem}
On rappelle que, pour $n\geq1$, la catégorie $\nCat{n}$ des petites
$n$\nbd-catégories strictes, et $n$\nbd-foncteurs stricts entre celles-ci,
s'identifie à la sous-catégorie pleine de $\ooCat$ formée des \oo-catégories
dont les $i$\nbd-cellules, pour $i>n$, sont des identités. La notion
d'équivalence faible dans $\ooCat$ induit une notion d'équivalence faible
dans $\nCat{n}$, un $n$\nbd-foncteur étant une équivalence faible s'il l'est
en tant que \oo-foncteur. D'autre part, si $X\to Y$ est un morphisme de
$\nCat{n}$ et $y$ un objet de $Y$, on vérifie immédiatement que la
\oo-catégorie $\cotr{X}{y}$ est en fait une $n$\nbd-catégorie. Ainsi, le
théorème A \oo-catégorique énoncé ci-dessus implique aussitôt un théorème A
$n$\nbd-catégorique dont l'énoncé est le même, sauf qu'on remplace
\og$\ooCat$\fg{} par \og$\nCat{n}$\fg{} et \og\oo-foncteur\fg{} par
\og$n$\nbd-foncteur\fg.
\end{rem}

La suite de cette section est consacrée à la preuve du théorème~\ref{thm:Th_A_oocat}.

\begin{paragr}\label{paragr:def_S_ooCat}
Pour tout \oo-foncteur $f:X\to Y$, on note $S(f)$ l'ensemble bisimplicial défini, avec les notations~\ref{notation:simili_joint}, par
\[\begin{aligned}
(S(f))_{m,n}&=\{(y\in(\Nrf Y)_{m+1+n},\,x\in(\Nrf X)_n)\mid (\Nrf f)(x)=y_{m+1,\dots,m+1+n}\}\cr
\noalign{\vskip 3pt}
&=\{(y:\On{m+1+n}\to Y,\,x:\On{n}\to X)\mid fx=y\,\On{j_{m,n}}\}\,,
\end{aligned}\]
pour $m,n\geq0$, les opérateurs simpliciaux étant définis de la façon suivante. Soient \hbox{$\varphi:\Deltan{m'}\to\Deltan{m}$} et $\psi:\Deltan{n'}\to\Deltan{n}$ des morphismes de $\cDelta$. Pour $(y,x)$ dans $(S(f))_{m,n}$, on pose $(S(f))_{\varphi,\psi}(y,x)=(y\,\On{{\varphi\amalg\psi}},\,x\,\On{\psi})$. 
\smallbreak

On définit un morphisme d'oubli \hbox{$U_f:S(f)\to\Nrf X$} par $(y,x)\mapsto x$, l'ensemble simplicial $\Nrf X$ étant considéré comme ensemble bisimplicial constant en la première variable
\[
(\Nrf X)_{m,n}=(\Nrf X)_n=\{x:\On{n}\to X\}\,.
\]
Pour tout $n\geq0$, le morphisme d'ensembles simpliciaux $(U_f)^{}_{\bullet,n}$ s'identifie à la somme
\[
\textstyle\coprod\limits_{x:\On{n}\to X}\kern -5pt\tr{\Nrf Y}{fx}\kern6pt\toto\kern-3pt\coprod\limits_{x:\On{n}\to X}\kern -5pt\ast\,\,,
\]
indexée par les $n$\nbd-simplexes $x:\On{n}\to X$ de $\Nrf X$, des
morphismes de source $\tr{\Nrf Y}{fx}$ et de but le point simplicial, qui
sont des équivalences faibles puisque, en vertu du lemme~\ref{lemme:decal},
les ensembles simpliciaux $\tr{\Nrf Y}{fx}$ sont contractiles. La stabilité
des équivalences faibles par sommes et le lemme bisimplicial (lemme~\ref{lembsmpl}) impliquent alors que le morphisme d'ensembles bisimpliciaux $U_f$ est une équivalence faible diagonale.
\smallbreak

Enfin, on remarque que, par définition, pour tout $m\geq0$, on a un isomorphisme canonique d'ensembles simpliciaux
\[
(S(f))_{m,\bullet}\simeq\textstyle\coprod\limits_{y:\On{m}\to Y}\cotr{\Nrf X}{y}\,.
\]
\end{paragr}

\emph{Dans la suite, on fixe un triangle dans $\ooCat$, commutatif à transformation oplax donnée près:}
\[\mathcal{T}=
\raise 25pt
\vbox{
    \shorthandoff{;}
    \xymatrix@C=1.5pc{
      A \ar[rr]^u \ar[dr]_(.45){v}_{}="f" & & B \ar[dl]^(0.45){w} \\
      & C
      \ar@{}"f";[ur]_(.15){}="ff"
      \ar@{}"f";[ur]_(.55){}="oo"
      \ar@<-0.5ex>@2"ff";"oo"^{\alpha}
      &.
    }
}
\]

\begin{paragr}\label{paragr:def_kappa}
On va définir un morphisme d'ensembles bisimpliciaux $S(\mathcal{T}):S(v)\to S(w)$ comme suit. Soient $m,n\geq0$ et $(c,a)\in (S(v))_{m,n}$. Par définition, avec les notations~\ref{notation:simili_joint}, on a un diagramme commutatif 
\[
\xymatrixcolsep{.4pc}
\xymatrix{
\On{n}\ar@{}[r]|-{\textstyle\simeq}\ar[d]_{\On{j_{m,n}}}
&\{0\}\otimes\On{n}\ar@{^{(}->}[rr]
&&\Dsk\otimes\On{n}\ar[rrr]^{1_{\Dsk}\otimes a}
&&&\Dsk\otimes A\ar[d]^{\alpha}
\\
\On{m+1+n}\ar[rrrrrr]_c
&&&&&&C
&,
}
\]
d'où un \oo-foncteur
\[
(c,\alpha(1_{\Dsk}\otimes a)):\On{m+1+n}\amalg_{\On{n}}(\Dsk\otimes\On{n})\toto C\,.
\]
On va définir un morphisme de complexes dirigés augmentés 
\[
\kappa_{m,n}:\cn \Deltan{m+1+n}\toto \cn \Deltan{m+1+n}\amalg_{\cn \Deltan{n}}(\cn \Deltan{1}\otimes \cn \Deltan{n})\,,
\]
d'où en vertu du théorème~\ref{thm:produit_Gray}, des propositions~\ref{prop:produit_Stf} et~\ref{prop:amalg_Stf}, et du paragraphe~\ref{paragr:def_cX},
un \oo-foncteur
\[
\nu(\kappa_{m,n}):\On{m+1+n}\toto\On{m+1+n}\amalg_{\On{n}}(\Dsk\otimes\On{n})\,,
\]
et on posera
\[
S(\mathcal{T})(c,a)=((c,\alpha(1_{\Dsk}\otimes a))\nu(\kappa_{m,n}),ua)\,.
\]
Le complexe dirigé augmenté $\cn \Deltan{m+1+n}$ admet comme base l'ensemble gradué $E$,
\[
E_p = \{(i_0, \dots, i_p) \mid 0 \le i_0 < \cdots < i_p \le m+1+n\}\,,\quad p\geq0\,,
\]
et $\cn \Deltan{n}$, identifié à un sous-complexe dirigé augmenté de $\cn \Deltan{m+1+n}$ par le morphisme $\cn (j_{m,n})$, admet comme base le sous-ensemble gradué $E'$,
\[
E'_p = \{(i_0, \dots, i_p) \mid  m+1\leq i_0<\cdots<i_p\leq m+1+n \}\,,\quad p\geq0\,,
\]
de $E$. En identifiant $\cn \Deltan{n}$ au sous-complexe dirigé augmenté $\{0\}\otimes \cn \Deltan{n}$ de $\cn \Deltan{1}\otimes \cn \Deltan{n}$, les éléments
\[
(i_0, \dots, i_p)\,,\ (i'_0, \dots, i'_p)\,,\ (\bar \imath_0, \dots, \bar \imath_p)\,,\quad m+1\leq i_0<\cdots<i_p\leq m+1+n\,,\quad p\geq0\,,
\]
où on a posé $(i'_0, \dots, i'_p)=(1)\otimes(i_0, \dots, i_p)$ et $(\bar \imath_0, \dots, \bar \imath_p)=(01)\otimes(i_0, \dots, i_p)$, forment une base de $\cn \Deltan{1}\otimes \cn \Deltan{n}$. Avec ces notations, en vertu de la proposition~\ref{prop:amalg_Stf}, les éléments
\[\begin{matrix}
&(i_0, \dots, i_p)\,,\hfill&\qquad0\leq i_0<\cdots<i_p\leq m+1+n\,,\quad p\geq0\,,\hfill\cr
\noalign{\vskip 3pt}
&(i'_0, \dots, i'_p)\,,\quad (\bar \imath_0, \dots, \bar \imath_p)\,,&\qquad m+1\leq i_0<\cdots<i_p\leq m+1+n\,,\quad p\geq0\,,
\end{matrix}\]
forment une base de $\cn \Deltan{m+1+n}\amalg_{\cn \Deltan{n}}(\cn \Deltan{1}\otimes \cn \Deltan{n})$. Par définition, le morphisme de groupes gradués sous-jacent à $\kappa_{m,n}$ associe, à un élément $(i_0, \dots, i_p)$ de la base de $\cn \Deltan{m+1+n}$, l'élément homogène de degré $p$
\[\begin{matrix}
&(1)\qquad&(i_0, \dots, i_p)\hfill &\qquad\qquad\hbox{si}\quad r\geq2\,,\cr
\noalign{\vskip 4pt}
&(2)\qquad&(i_0, \dots, i_p)+(\bar \imath_1, \dots, \bar \imath_p)&\qquad\qquad\hbox{si}\quad r=1\,,\cr
\noalign{\vskip 4pt}
&(3)\qquad&(i'_0, \dots, i'_p)\hfill&\qquad\qquad\hbox{si}\quad r=0\,
\end{matrix}\kern 60pt\]
de $\cn \Deltan{m+1+n}\amalg_{\cn \Deltan{n}}(\cn \Deltan{1}\otimes \cn \Deltan{n})$, où $r=\card\{0\leq k\leq p\mid i_k\leq m\}$, et où par convention dans $(2)$, si $p=0$, alors $(\bar \imath_1, \dots, \bar \imath_p)=0$ (de sorte que pour $p=0$, on a $\kappa_{m,n}(i_0)=(i_0)$ si $i_0\leq m$ et $\kappa_{m,n}(i_0)=(i'_0)$ si $i_0> m$). \'Etant donné des entiers $0\leq i_0\leq\cdots\leq i_p\leq n$, s'il existe $l$ tel que $0\leq l<p$ et tel que  $i_l=i_{l+1}$, les formules définissant $\kappa_{m,n}$ sont compatibles avec la convention $(i_0,\dots,i_p)=0$ (voir le paragraphe~\ref{paragr:def_cX}). En effet, cela est évident pour les cas $(1)$ et $(3)$, et dans le cas (2), cela résulte du fait que $r=1$ implique que $l\geq1$.
\smallbreak

On doit vérifier que $\kappa_{m,n}$ est un morphisme de complexes dirigés augmentés, que pour tout $(c,a)\in (S(v))_{m,n}$,  le couple $((c,\alpha(1_{\Dsk}\otimes a))\nu(\kappa_{m,n}),ua)$ appartient à $(S(w))_{m,n}$, et que la famille des applications $(S(\mathcal{T}))_{m,n}$, $m,n\geq0$, ainsi définies est un morphisme d'ensembles bisimpliciaux. 
\end{paragr}

\begin{paragr}\label{paragr:comp_diff_kappa}
Montrons pour commencer que $\kappa_{m,n}$ est un morphisme de complexes dirigés augmentés. La compatibilité à l'augmentation et aux sous-monoïdes de positivité est évidente. Vérifions la compatibilité aux différentielles. Soient $p\geq0$ et $(i_0,\dots,i_p)$ un élément de la base de $\cn \Deltan{m+1+n}$; on doit prouver que
\[
d\kappa_{m,n}(i_0,\dots,i_p)=\kappa_{m,n}d(i_0,\dots,i_p)\,.
\]
Soit $r=\card\{0\leq k\leq p\mid i_k\leq m\}$. Si $r=0$ ou $r>2$, l'égalité est évidente. Il reste à la vérifier pour $r=1$ et $r=2$. Si $r=1$, autrement dit si $i_0\leq m$ et $i_k>m$, pour $1\leq k\leq p$, on a
\[\begin{aligned}
d\kappa_{m,n}(i_0,\dots,i_p)&=d((i_0, \dots, i_p)+(\bar \imath_1, \dots, \bar \imath_p))\cr
\noalign{\vskip 3pt}
&=\textstyle\sum\limits_{k=0}^p(-1)^k(i_0,\dots,\hat i_k,\dots,i_p)\cr
\noalign{\vskip -8pt}
&\qquad\kern7pt+(i'_1, \dots, i'_p)-(i_1, \dots, i_p)-\textstyle\sum\limits_{k=1}^p(-1)^{k-1}(\bar \imath_1,\dots,\hat{ \bar {\imath}}_k,\dots,\bar \imath_p)\cr
&=(i'_1, \dots, i'_p)+\textstyle\sum\limits_{k=1}^p(-1)^k\bigl((i_0,\dots,\hat i_k,\dots,i_p)+(\bar {\imath}_1,\dots,\hat{\bar{\imath}}_k,\dots,\bar \imath_p)\bigr)\cr
&=\kappa_{m,n}\Bigl(\textstyle\sum\limits_{k=0}^p(-1)^k(i_0,\dots,\hat i_k,\dots,i_p)\Bigr)=\kappa_{m,n}d(i_0,\dots,i_p)\,,
\end{aligned}\]
et si $r=2$, autrement dit si $i_0,i_1\leq m$ et $i_k>m$, pour $2\leq k\leq p$, on a
\[\begin{aligned}
\kappa_{m,n}d(i_0,\dots,i_p)&=\kappa_{m,n}\Bigl(\textstyle\sum\limits_{k=0}^p(-1)^k(i_0,\dots,\hat i_k,\dots,i_p)\Bigr)\cr
\noalign{\vskip 3pt}
&=(i_1, \dots, i_p)+(\bar \imath_2, \dots, \bar \imath_p)-((i_0,i_2, \dots, i_p)+(\bar \imath_2, \dots, \bar \imath_p))\cr
\noalign{\vskip -3pt}
&\qquad\kern7pt+\textstyle\sum\limits_{k=2}^p(-1)^k(i_0,\dots,\hat i_k,\dots,i_p)\cr
&=\textstyle\sum\limits_{k=0}^p(-1)^k(i_0,\dots,\hat i_k,\dots,i_p)=d(i_0,\dots,i_p)=d\kappa_{m,n}(i_0,\dots,i_p)\,.
\end{aligned}\]
\end{paragr}

\begin{paragr}
Vérifions maintenant que pour tout élément $(c,a)$ de $(S(v))_{m,n}$, le couple $((c,\alpha(1_{\Dsk}\otimes a))\nu(\kappa_{m,n}),ua)$ appartient à $(S(w))_{m,n}$. On remarque qu'en vertu du cas~(3) de la définition de $\kappa_{m,n}$ dans le paragraphe~\ref{paragr:def_kappa}, et avec les notations~\ref{notation:simili_joint}, on a un diagramme commutatif 
\[
\xymatrixcolsep{.6pc}
\xymatrix{
&\On{n}\ar@{}[r]|-{\textstyle\simeq}\ar[d]_{\On{j_{m,n}}}
&\{1\}\otimes\On{n}\ar@{^{(}->}[r]^{\{1\}\otimes1_{\On{n}}}
&\Dsk\otimes\On{n}\ar@{^{(}->}[d]^{\mathrm{can}}
\\
&\On{m+1+n}\ar[rr]_-{\nu(\kappa_{m,n})}
&&\On{m+1+n}\amalg_{\On{n}}(\Dsk\otimes\On{n})
&.
}
\]
On a donc
\[\begin{aligned}
(c,\alpha(1_{\Dsk}\otimes a))\nu(\kappa_{m,n})\On{j_{m,n}}&=\alpha(1_{\Dsk}\otimes a)(\{1\}\otimes1_{\On{n}})\cr
&=\alpha(\{1\}\otimes1_A)a=wua\,,
\end{aligned}\]
ce qui prouve l'assertion. On définit donc bien une application
\[
(S(\mathcal{T}))_{m,n}:(S(v))_{m,n}\to(S(w))_{m,n}\,,\qquad(c,a)\mapsto((c,\alpha(1_{\Dsk}\otimes a))\nu(\kappa_{m,n}),ua)\,.
\]
\end{paragr}

\begin{paragr}\label{paragr:comp_smpl_kappa}
On va montrer maintenant que les applications $(S(\mathcal{T}))_{m,n}$, $m,n\geq0$, ainsi définies, forment un morphisme d'ensembles bisimpliciaux $S(\mathcal{T}):S(v)\to S(w)$, autrement dit que pour tous morphismes $\varphi:\Deltan{m'}\to\Deltan{m}$ et $\psi:\Deltan{n'}\to\Deltan{n}$ de $\cDelta$, le carré suivant est commutatif
\[
\xymatrixcolsep{4pc}
\xymatrix{
&(S(v))_{m,n}\ar[r]^{(S(\mathcal{T}))_{m,n}}\ar[d]_{(S(v))_{\varphi,\psi}}
&(S(w))_{m,n}\ar[d]^{(S(w))_{\varphi,\psi}}
\\
&(S(v))_{m',n'}\ar[r]_{(S(\mathcal{T}))_{m',n'}}
&(S(w))_{m',n'}
&\kern -30pt.\kern30pt
}
\]
Or, pour tout $(c,a)$ dans $(S(v))_{m,n}$, on a
\[\begin{aligned}
(S(\mathcal{T}))_{m',n'}(S(v))_{\varphi,\psi}(c,a)&=(S(\mathcal{T}))_{m',n'}(c\,\On{\varphi\amalg\psi},a\,\On{\psi})\cr
\noalign{\vskip 3pt}
&=((c\,\On{\varphi\amalg\psi},\alpha(1_{\Dsk}\otimes a\,\On{\psi}))\nu(\kappa_{m',n'}),ua\,\On{\psi})\cr
\noalign{\vskip 3pt}
&=((c,\alpha(1_{\Dsk}\otimes a))\cr
&\kern 20pt\circ(\On{\varphi\amalg\psi}\amalg_{\On{\psi}}(1_{\Dsk}\otimes \On{\psi}))\nu(\kappa_{m',n'}),ua\,\On{\psi})\,,\cr
\noalign{\vskip 6pt}
(S(w))_{\varphi,\psi}(S(\mathcal{T}))_{m,n}(c,a)&=(S(w))_{\varphi,\psi}((c,\alpha(1_{\Dsk}\otimes a))\nu(\kappa_{m,n}),ua)\cr
\noalign{\vskip 3pt}
&=((c,\alpha(1_{\Dsk}\otimes a))\nu(\kappa_{m,n})\On{\varphi\amalg\psi},ua\On{\psi})\,.
\end{aligned}\]
Il suffit donc de montrer que le carré suivant est commutatif
\[
\xymatrixcolsep{3pc}
\xymatrix{
\cn \Deltan{m'+1+n'}\ar[r]^-{\kappa_{m',n'}}\ar[d]_{\cn (\varphi\amalg\psi)}
&\cn \Deltan{m'+1+n'}\amalg_{\cn \Deltan{n'}}(\cn \Deltan{1}\otimes \cn \Deltan{n'})\ar[d]^{\cn (\varphi\amalg\psi)\amalg_{\cn (\psi)}(1_{\cn \Deltan{1}}\otimes \cn (\psi))}
\\
\cn \Deltan{m+1+n}\ar[r]_-{\kappa_{m,n}}
&\cn \Deltan{m+1+n}\amalg_{\cn \Deltan{n}}(\cn \Deltan{1}\otimes \cn \Deltan{n})\,,
}
\]
ce qui résulte aussitôt des formules définissant $\kappa_{m,n}$ dans le paragraphe~\ref{paragr:def_kappa} et du commentaire qui les suit.
\end{paragr}

\begin{paragr}
Ainsi, on dispose d'un carré commutatif d'ensembles bisimpliciaux
\[
\xymatrixcolsep{3pc}
\xymatrix{
S(v)\ar[r]^{S(\mathcal{T})}\ar[d]_{U_v}
&S(w)\ar[d]^{U_w}
\\
\Nrf A\ar[r]_{\Nrf u}
&\Nrf B
}
\]
dont les flèches verticales sont, en vertu du paragraphe~\ref{paragr:def_S_ooCat}, des équivalences faibles diagonales. On va montrer que l'hypothèse du théorème~\ref{thm:Th_A_oocat} implique que $S(\mathcal{T})$ est aussi une équivalence faible diagonale. On en déduira par deux sur trois qu'il en est de même pour le morphisme bisimplicial constant en la première variable $\Nrf u$, autrement dit, que le morphisme d'ensembles simpliciaux $\Nrf u$ est une équivalence faible, ce qui prouvera le théorème~\ref{thm:Th_A_oocat}.

Or, en vertu du paragraphe~\ref{paragr:def_S_ooCat}, pour tout $m\geq0$, on a des isomorphismes d'ensembles simpliciaux
\[
(S(v))_{m,\bullet}\simeq\textstyle\coprod\limits_{c:\On{m}\to C}\cotr{\Nrf A}{c}\qquad\hbox{et}\qquad(S(w))_{m,\bullet}\simeq\textstyle\coprod\limits_{c:\On{m}\to C}\cotr{\Nrf B}{c}\,.
\] 
D'autre part, dans les notations~\ref{notation:simili_joint}, pour tout $n\geq0$, on a, par définition de $\kappa_{m,n}$, un carré commutatif
\[
\xymatrixcolsep{3pc}
\xymatrix{
&\On{m}\ar[r]^-{\On{i_{m,n}}}\ar[d]_{\On{i_{m,n}}}
&\On{m+1+n}\ar[d]^{\mathrm{can}}
\\
&\On{m+1+n}\ar[r]_-{\nu(\kappa_{m,n})}
&\On{m+1+n}\amalg_{\On{n}}(\Dsk\otimes\On{n})
&\kern -30pt.\kern30pt
}
\] 
On en déduit que pour tout $(c',a)\in (S(v))_{m,n}$, on a 
\[
(c',\alpha(1_{\Dsk}\otimes a))\nu(\kappa_{m,n})\On{i_{m,n}}=c'\On{i_{m,n}}\,,
\]
et par suite, le morphisme d'ensembles simpliciaux \hbox{$(S(\mathcal{T}))_{m,\bullet}\kern -.7pt:\kern -.7pt(S(v))_{m,\bullet}\to(S(w))_{m,\bullet}$} s'identifie à la somme 
\[
\textstyle\coprod\limits_{c:\On{m}\to C}\cotr{\Nrf A}{c}\kern 3pt\toto\kern -3pt\textstyle\coprod\limits_{c:\On{m}\to C}\cotr{\Nrf B}{c}\,,
\]
indexée par les $m$\nbd-simplexes $c$ de $\Nrf C$, des morphismes, notés
$\cotr{\Nrf\mathcal{T}}{c}:\cotr{\Nrf A}{c}\to\cotr{\Nrf B}{c}$, associant à
un $n$\nbd-simplexe $(c',a)$ de $\cotr{\Nrf A}{c}$ le $n$\nbd-simplexe
$((c',\alpha(1_{\Dsk}\otimes a))\nu(\kappa_{m,n}),\,ua)$ de $\cotr{\Nrf
B}{c}$. En vertu du lemme bisimplicial (lemme~\ref{lembsmpl}) et de la stabilité par sommes des équivalences faibles, pour montrer que $S(\mathcal{T})$ est une équivalence faible diagonale, il suffit de montrer que ces morphismes sont des équivalences faibles.
\end{paragr}

\begin{paragr}
Il s'agit donc de montrer que pour tout $m$\nbd-simplexe $c:\On{m}\to C$ de $\Nrf C$, le morphisme $\cotr{\Nrf\mathcal{T}}{c}:\cotr{\Nrf A}{c}\to\cotr{\Nrf B}{c}$ est une équivalence faible simpliciale. Or, il résulte du paragraphe~\ref{paragr:comp_smpl_kappa} qu'on a un carré commutatif d'ensembles simpliciaux
\[
\xymatrixcolsep{4pc}
\xymatrix{
&(S(v))_{m,\bullet}\ar[r]^{(S(\mathcal{T}))_{m,\bullet}}\ar[d]_{(S(v))_{\varphi,\bullet}}
&(S(w))_{m,\bullet}\ar[d]^{(S(w))_{\varphi,\bullet}}
\\
&(S(v))_{0,\bullet}\ar[r]_{(S(\mathcal{T}))_{0,\bullet}}
&(S(w))_{0,\bullet}
&\kern -30pt,\kern30pt
}
\]
où $\varphi:\Deltan{0}\to\Deltan{m}$ désigne le morphisme de $\cDelta$ défini par $0\mapsto m$. En vertu du paragraphe précédent, ce carré induit un carré commutatif
\[
\raise 20pt\vbox{
\xymatrixcolsep{4pc}
\xymatrix{
&\cotr{\Nrf A}{c}\ar[r]^{\cotr{\Nrf\mathcal{T}}{c}}\ar[d]
&\cotr{\Nrf B}{c}\ar[d]
\\
&\cotr{\Nrf A}{c_m}\ar[r]_{\cotr{\Nrf\mathcal{T}}{c_m}}
&\cotr{\Nrf B}{c_m}
&\kern -40pt,\kern40pt
}
}\leqno(*)
\]
dont la flèche verticale de gauche associe à un $n$\nbd-simplexe $(c',a)$ de $\cotr{\Nrf A}{c}$, le $n$\nbd-simplexe $(c'_{m,m+1,\dots,m+1+n},a)$ de $\cotr{\Nrf A}{c_m}$, celle de droite étant définie de façon analogue. 
\smallbreak

D'autre part, en vertu de la proposition~\ref{prop:iso_comp_cotr}, on a des isomorphismes canoniques $\cotr{\Nrf A}{c_m}\simeq\Nrf(\cotr{A}{c_m})$ et $\cotr{\Nrf B}{c_m}\simeq\Nrf(\cotr{B}{c_m})$, et le carré
\[
\xymatrixcolsep{4pc}
\xymatrix{
\cotr{\Nrf A}{c_m}\ar[r]^{\cotr{\Nrf\mathcal{T}}{c_m}}\ar[d]_(.45)\wr
&\cotr{\Nrf B}{c_m}\ar[d]^(.45)\wr
\\
\Nrf(\cotr{A}{c_m})\ar[r]_{\Nrf(\cotr{\mathcal{T}}{c_m})}
&\Nrf(\cotr{B}{c_m})
}
\]
est commutatif. En effet, en vertu de la description de ces isomorphismes, donnée dans la preuve de cette proposition, avec les notations de cette preuve, et en tenant compte des définitions de $\cotr{\Nrf\mathcal{T}}{c_m}$ et de $\cotr{\mathcal{T}}{c_m}$, il suffit de montrer que, pour tout $n$\nbd-simplexe $(c',a)$ de $\cotr{\Nrf A}{c_m}$, on a 
\[
\bigl(ua,(c',\alpha(1_{\Dsk}\otimes a))\nu(\kappa_{0,n})\nu(\pi_n)\bigr)=\bigl(ua,(\alpha\comp a)(c'\nu(\pi_n))\bigr)\,,
\]
où $(\alpha\comp a)(c'\nu(\pi_n))$ désigne le composé vertical des transformations oplax $\alpha\comp a$ et $c'\nu(\pi_n)$. En vertu de la définition de cette composition (voir le paragraphe~\ref{paragr:def_comp_vert}), il s'agit de montrer que le pentagone suivant est commutatif
\[
\xymatrixrowsep{.4pc}
\xymatrixcolsep{3.5pc}
\xymatrix{
\Dsk\otimes \On{n}\ar[dd]\ar[rr]^{\nu(\pi_n)}\ar[dd]_{\delta\otimes1_{\On{n}}}
&&\On{1+n}\ar[ddd]^{\nu(\kappa_{0,n})}
\\
\\
(\Dsk\amalg_{\Deltan{0}}\Dsk)\otimes \On{n}\ar@{-}[d]^{\kern -1pt\wr}
\\
(\Dsk\otimes \On{n})\amalg_{\On{n}}(\Dsk\otimes \On{n})\ar[rr]^-{\nu(S)(1_{\Dsk\otimes \On{n}}\amalg_{\On{n}}\nu(\pi_n))}\ar[rddd]_{(\alpha(1_{\Dsk}\otimes a),c'\nu(\pi_n))\kern17pt}
&&\On{1+n}\amalg_{\On{n}}(\Dsk\otimes \On{n})\ar[lddd]^{\kern 15pt(c',\alpha(1_{\Dsk}\otimes a))}
\\
\\
\\
&C
&.
}
\]
Or, le triangle du bas, où $S$ désigne l'isomorphisme canonique
\[
(\cn \Deltan{1}\otimes \cn \Deltan{n})\amalg_{\cn \Deltan{n}}\cn \Deltan{1+n}\toto \cn \Deltan{1+n}\amalg_{\cn \Deltan{n}}(\cn \Deltan{1}\otimes \cn \Deltan{n})\,,
\]
est trivialement commutatif. Il suffit donc de montrer que le carré
\[
\xymatrixrowsep{.4pc}
\xymatrixcolsep{3pc}
\xymatrix{
\cn \Deltan{1}\otimes \cn \Deltan{n}\ar[dd]\ar[rr]^{\pi_n}\ar[dd]_{\lambda(\delta)\otimes1_{\cn \Deltan{n}}}
&&\cn \Deltan{1+n}\ar[ddd]^{\kappa_{0,n}}
\\
\\
(\cn \Deltan{1}\amalg_{\cn \Deltan{0}}\cn \Deltan{1})\otimes \cn \Deltan{n}\ar@{-}[d]^{\kern -1pt\wr}
\\
(\cn \Deltan{1}\otimes \cn \Deltan{n})\amalg_{\cn \Deltan{n}}(\cn \Deltan{1}\otimes \cn \Deltan{n})\ar[rr]^-{S(1_{\cn \Deltan{1}\otimes \cn \Deltan{n}}\amalg_{\cn \Deltan{n}}\pi_n)}
&&\cn \Deltan{1+n}\amalg_{\cn \Deltan{n}}(\cn \Deltan{1}\otimes \cn \Deltan{n})
}
\]
est commutatif, ce qui résulte d'une vérification facile laissée au lecteur.
\smallbreak

Comme, en vertu de l'hypothèse du théorème, le morphisme d'ensembles simpliciaux $\Nrf(\cotr{\mathcal{T}}{c_m})$ est une équivalence faible, on en déduit qu'il en est de même du morphisme $\cotr{\Nrf\mathcal{T}}{c_m}$. Pour conclure, il suffit donc de prouver que les deux flèches verticales du carré commutatif $(*)$ sont des équivalences faibles, ce qui résultera de la section suivante.
\end{paragr}

\section{L'homotopie simpliciale}\label{section:hmtpsmpl}

\begin{paragr}
Dans cette section, on fixe deux \oo-catégories $A$ et $B$, un morphisme d'ensembles simpliciaux $u:\Nrf A\to\Nrf B$, un entier $m\geq0$, et un $m$\nbd-simplexe \hbox{$b:\On{m}\to B$} de $\Nrf B$. On a un morphisme d'ensembles simpliciaux
\[
r:\cotr{\Nrf A}{b}\to\cotr{\Nrf A}{b_m}
\]
qui, pour $n\geq0$, associe à un $n$\nbd-simplexe $(y,x)$ de $\cotr{\Nrf A}{b}$,
\[
y:\On{m+1+n}\to B\,,\ x:\On{n}\to A\quad \hbox{tels que}\quad y_{0,\dots,m}=b\,,\ y_{m+1,\dots,m+1+n}=u(x)\,,
\]
le $n$\nbd-simplexe $(y_{m,m+1,\dots,m+1+n},x)$ de $\cotr{\Nrf A}{b_m}$. Le but de cette section est de montrer que ce morphisme est une équivalence d'homotopie simpliciale, et plus précisément, qu'il fait de $\cotr{\Nrf A}{b_m}$ un rétracte par déformation fort de $\cotr{\Nrf A}{b}$.
\end{paragr}

\begin{paragr}\label{paragr:def_section}
Pour commencer, on définit une section
\[
s:\cotr{\Nrf A}{b_m}\to\cotr{\Nrf A}{b}
\]
de ce morphisme comme suit. Il s'agit de définir des applications
\[
s_n:(\cotr{\Nrf A}{b_m})_n\to(\cotr{\Nrf A}{b})_n\,,\quad n\geq0\,,
\]
compatibles aux opérateurs simpliciaux. Soient $n\geq0$ et $(y',x)\in(\cotr{\Nrf A}{b_m})_n$,
\[
y':\On{1+n}\to B\,,\ x:\On{n}\to A\quad \hbox{tels que}\quad y'_{0}=b_m\,,\ y'_{1,\dots,1+n}=u(x)\,.
\]
On en déduit un \oo-foncteur
\[
(b,y'):\On{m}\amalg_{\On{0}}\On{1+n}\to B\,.
\]
On va définir un morphisme de complexes dirigés augmentés 
\[
f_n:\cn \Deltan{m+1+n}\to \cn \Deltan{m}\amalg_{\cn \Deltan{0}}\cn \Deltan{1+n}\,,
\] 
d'où en vertu de la proposition~\ref{prop:amalg_Stf} et des paragraphes~\ref{paragr:def_cX} et~\ref{paragr:def_nrf_Str}, un \oo-foncteur 
\[
\nu(f_n):\On{m+1+n}\to\On{m}\amalg_{\On{0}}\On{1+n}\,, 
\]
et on posera 
\[
s_n(y',x)=((b,y')\nu(f_n),x)\,.
\]

Le complexe dirigé augmenté $\cn \Deltan{m+1+n}$ admet comme base l'ensemble gradué $E$,
\[
E_p = \{(i_0, \dots, i_p) \mid 0 \le i_0 < \cdots < i_p \le m+1+n\}\,,\quad p\geq0\,,
\]
et la somme amalgamée $\cn \Deltan{m}\amalg_{\cn \Deltan{0}}\cn \Deltan{1+n}$ le sous-ensemble gradué $E'$,
\[
E'_p = \{(i_0, \dots, i_p)\in E_p \mid i_p\leq m\quad \hbox{ou}\quad m\leq i_0 \}\,,\quad p\geq0\,,
\]
de $E$, l'inclusion de $E'$ dans $E$ définissant une inclusion de complexes dirigés augmentés
\[
\cn \Deltan{m}\amalg_{\cn \Deltan{0}}\cn \Deltan{1+n}\hookrightarrow \cn \Deltan{m+1+n}\,.
\]
Par définition, le morphisme de groupes gradués sous-jacent à $f_n$ associe, à un élément $(i_0, \dots, i_p)$ de la base de $\cn \Deltan{m+1+n}$, l'élément homogène de degré $p$
\[\begin{matrix}
&(1)\qquad&(i_0, \dots, i_p)\hfill\quad &\hbox{si}\quad i_p\leq m\quad \hbox{ou}\quad m\leq i_0\,,\hfill\cr
\noalign{\vskip 4pt}
&(2)\qquad&(i_0,m)+(m,i_1)\hfill\quad&\hbox{si}\quad p=1\quad\hbox{et}\quad i_0<m<i_1\,,\hfill\cr
\noalign{\vskip 4pt}
&(3)\qquad&(m,i_1, \dots, i_p)\hfill\quad&\hbox{si}\quad p>1\quad\hbox{et}\quad i_0<m<i_1\,,\hfill\cr
\noalign{\vskip 4pt}
&(3')\qquad&(i_0, \dots, i_{p-1},m)\hfill\quad&\hbox{si}\quad p>1\quad\hbox{et}\quad i_{p-1}<m<i_p\,,\hfill\cr
\noalign{\vskip 4pt}
&(4)\qquad&\,0\hfill\quad&\hbox{sinon}\ (\hbox{\emph{i.e.} \,si}\quad i_1\leq m\leq i_{p-1}\,)\,\kern30pt\hfill
\end{matrix}\]
de $\cn \Deltan{m}\amalg_{\cn \Deltan{0}}\cn \Deltan{1+n}$. On rappelle (voir le paragraphe~\ref{paragr:def_cX}) qu'étant donné des entiers $0\leq i_0\leq\cdots\leq i_p\leq n$, s'il existe $k$ tel que $0\leq k<p$ et tel que  $i_k=i_{k+1}$, alors par convention $(i_0,\dots,i_p)=0$. On observe que les formules définissant $f_n$ sont compatibles avec cette convention, dans le sens où ces formules impliquent alors que $f_n(i_0,\dots,i_p)=0$. En effet, cela est évident dans les cas (1) et (4), dans le cas (3), cela résulte du fait que l'inégalité stricte $i_0<i_1$ implique que $k\geq1$, le cas $(3')$ est analogue, et le cas (2) n'intervient pas.
\smallbreak

On doit vérifier que $f_n$ est un morphisme de complexes dirigés augmentés, que pour tout $n$\nbd-simplexe $(y',x)$ de $\cotr{\Nrf A}{b_m}$ le couple $((b,y')\nu(f_n),x)$ est un $n$\nbd-simplexe de $\cotr{\Nrf A}{b}$, que la famille des applications $s_n$, $n\geq0$, ainsi définies est un morphisme d'en\-sem\-bles simpliciaux et que ce morphisme est une section du morphisme $r$.
\end{paragr}

\begin{paragr}\label{paragr:comp_diff}
Montrons d'abord que $f_n$ est un morphisme de complexes dirigés augmentés. La compatibilité à l'augmentation et aux sous-monoïdes de positivité est évidente. Vérifions la compatibilité aux différentielles. Soient $p\geq0$, et $(i_0,\dots,i_p)$ un élément de la base de $\cn \Deltan{m+1+n}$; on doit prouver que
\[
df_n(i_0,\dots,i_p)=f_nd(i_0,\dots,i_p)\,.
\]
On distingue plusieurs cas, en suivant la définition de $f_n$:
\begin{itemize}
\item[(1)] $i_p\leq m$ ou $m\leq i_0$: l'égalité ci-dessus est alors évidente.
\smallbreak
\item[(2)] $p=1$ et $i_0<m<i_1$:
\[\begin{aligned}
df_n(i_0,i_1)&=d((i_0,m)+(m,i_1))=(m)-(i_0)+(i_1)-(m)= (i_1)-(i_0)\cr
&=f_n(i_1)-f_n(i_0)=f_n((i_1)-(i_0))=f_nd(i_0,i_1)\,.
\end{aligned}\]
\item[(3)] $p>1$ et $i_0<m<i_1$:
\[\begin{aligned}
df_n(i_0,\dots,i_p)&=d(m,i_1,\dots,i_p)\cr
&=(i_1,\dots,i_p)+\textstyle\sum\limits_{k=1}^p(-1)^k(m,i_1,\dots,\hat i_k,\dots,i_p)\,.
\end{aligned}\]
Pour le calcul de $f_nd(i_0,\dots,i_p)$, on distingue deux sous-cas:
\begin{itemize}
\item $p=2$:
\[\begin{aligned}
f_nd(i_0,i_1,i_2)&=f_n((i_1,i_2)-(i_0,i_2)+(i_0,i_1))\cr
&=(i_1,i_2)-((i_0,m)+(m,i_2))+((i_0,m)+(m,i_1))\cr
&=(i_1,i_2)-(m,i_2)+(m,i_1).
\end{aligned}\]
\item $p>2$:
\[\begin{aligned}
f_nd(i_0,\dots,i_p)&=f_n\Bigl(\textstyle\sum\limits_{k=0}^p(-1)^k(i_0,\dots,\hat i_k,\dots,i_p)\Bigr)\cr
&=(i_1,\dots,i_p)+\textstyle\sum\limits_{k=1}^p(-1)^k(m,i_1,\dots,\hat i_k,\dots,i_p)\,.
\end{aligned}\]
\end{itemize}
\item[$(3')$] $p>1$ et $i_{p-1}<m<i_p$: le calcul est parfaitement analogue au précédent.
\smallbreak
\item[(4)] $i_1\leq m\leq i_{p-1}$ (ce qui implique $p>1$):
\[
df_n(i_0,\dots,i_p)=0\,.\qquad
\]
Pour le calcul de $f_nd(i_0,\dots,i_p)$, on distingue plusieurs sous-cas:
\begin{itemize}
\item $p=2$ (l'inégalité $i_1\leq m\leq i_{p-1}=i_1$ implique alors que $i_1=m$):
\[\begin{aligned}
f_nd(i_0,i_1,i_2)&=f_n((i_1,i_2)-(i_0,i_2)+(i_0,i_1))\cr
&=(i_1,i_2)-((i_0,m)+(m,i_2))+(i_0,i_1)\cr
&=(m,i_2)-(i_0,m)-(m,i_2)+(i_0,m)=0\,.
\end{aligned}\]
\item $p>2$ et $i_1=m$:
\[\begin{aligned}
\kern 7pt f_nd(i_0,\dots,i_p)&=f_n\Bigl(\textstyle\sum\limits_{k=0}^p(-1)^k(i_0,\dots,\hat i_k,\dots,i_p)\Bigr)\cr
&=(i_1,\dots,i_p)-(m,i_2,\dots,i_p)+\textstyle\sum\limits_{k=2}^p(-1)^k\,0=0\,.
\end{aligned}\]
\item $p>2$ et $i_{p-1}=m$: le calcul est parfaitement analogue au précédent.
\smallbreak
\item $p=3$ et $i_1<m<i_2$:
\[\begin{aligned}
\kern 40pt f_nd(i_0,i_1,i_2,i_3)&=f_n((i_1,i_2,i_3)-(i_0,i_2,i_3)+(i_0,i_1,i_3)-(i_0,i_1,i_2))\cr
&=(m,i_2,i_3)-(m,i_2,i_3)+(i_0,i_1,m)-(i_0,i_1,m)=0\,.
\end{aligned}\]
Désormais le cas $p=3$ est réglé.
\smallbreak

\item $p>3$ et $i_1<m<i_2$:
\[\begin{aligned}
\kern 12pt f_nd(i_0,\dots,i_p)&=f_n\Bigl(\textstyle\sum\limits_{k=0}^p(-1)^k(i_0,\dots,\hat i_k,\dots,i_p)\Bigr)\cr
&=(m,i_2,\dots,i_p)-(m,i_2,\dots,i_p)+\textstyle\sum\limits_{k=2}^p(-1)^k\,0=0\,.
\end{aligned}\]
\item $p>3$ et $i_{p-2}<m<i_{p-1}$: le calcul est parfaitement analogue au précédent.
\smallbreak
\item $p>3$ et $i_2\leq m\leq i_{p-2}$:
\[\begin{aligned}
f_nd(i_0,\dots,i_p)&=f_n\Bigl(\textstyle\sum\limits_{k=0}^p(-1)^k(i_0,\dots,\hat i_k,\dots,i_p)\Bigr)\cr
&=\textstyle\sum\limits_{k=0}^p(-1)^k\,0=0\,.
\end{aligned}\]
\end{itemize}
\end{itemize}
\end{paragr}

\begin{paragr}
Le fait que, pour tout $n$\nbd-simplexe $(y',x)$ de $\cotr{\Nrf A}{b_m}$, le couple $((b,y')\nu(f_n),x)$ est un $n$\nbd-simplexe de $\cotr{\Nrf A}{b}$ 
résulte aussitôt du cas (1) de la définition de $f_n$ dans le paragraphe~\ref{paragr:def_section}. On définit donc bien une application 
\[
s_n:(\cotr{\Nrf A}{b_m})_n\to(\cotr{\Nrf A}{b})_n\,,\qquad(y',x)\mapsto((b,y')\nu(f_n),x)\,.
\]
\end{paragr}

\begin{paragr}\label{paragr:comp_smpl}
On va montrer maintenant que les applications $s_n$, $n\geq0$, ainsi définies, forment un morphisme d'ensembles simpliciaux $s:\cotr{\Nrf A}{b_m}\to\cotr{\Nrf A}{b}$, autrement dit que pour tout $\psi:\Deltan{n'}\to\Deltan{n}$, le carré suivant est commutatif
\[
\xymatrix{
&(\cotr{\Nrf A}{b_m})_n\ar[r]^{s_n}\ar[d]_{(\cotr{\Nrf A}{b_m})_\psi}
&(\cotr{\Nrf A}{b})_n\ar[d]^{(\cotr{\Nrf A}{b})_\psi}
\\
&(\cotr{\Nrf A}{b_m})_{n'}\ar[r]_{s_{n'}}
&(\cotr{\Nrf A}{b})_{n'}
&\kern-20pt.\kern20pt
}
\]
Or, pour tout $n$\nbd-simplexe $(y',x)$ de $\cotr{\Nrf A}{b_m}$, on a
\[\begin{aligned}
&(\cotr{\Nrf A}{b})_\psi s_n(y',x)=(\cotr{\Nrf A}{b})_\psi((b,y')\nu(f_n),x)=((b,y')\nu(f_n)\nu \cn (\psi'),x\nu\cn(\psi))\,,\cr
\noalign{\vskip 3pt}
&s_{n'}(\cotr{\Nrf A}{b_m})_\psi(y',x)=s_{n'}(y'\nu \cn (\psi''),x\nu\cn(\psi))=((b,y'\nu \cn (\psi''))\nu(f_{n'}),x\nu\cn(\psi))\,,
\end{aligned}\]
où avec les notations~\ref{notation:simili_joint},
\[
\psi'=\Deltan{m}\amalg\psi:\Deltan{m+1+n'}\to\Deltan{m+1+n}\qquad\hbox{et}\qquad\psi''=\Deltan{0}\amalg\psi:\Deltan{1+n'}\to\Deltan{1+n}\,.
\]
En tenant compte du diagramme
\[
\xymatrixcolsep{3pc}
\xymatrixrowsep{1pc}
\xymatrix{
\On{m+1+n'}\ar[r]^-{\nu(f_{n'})}\ar[dd]_{\nu \cn (\psi')}
&\On{m}\amalg_{\On{0}}\On{1+n'}\ar[dd]_{\On{m}\amalg_{\On{0}}\nu \cn (\psi'')}\ar[rd]^(.6){\kern 10pt(b,y'\nu \cn (\psi''))}
\\
&&B
\\
\On{m+1+n}\ar[r]_-{\nu(f_n)}
&\On{m}\amalg_{\On{0}}\On{1+n}\ar[ru]_(.6){(b,y')}
}
\]
dont le triangle de droite est commutatif, il suffit de prouver que
\[
f_n\cn (\psi')=(\cn \Deltan{m}\amalg_{\cn \Deltan{0}}\cn (\psi''))f_{n'}\,.
\]
Pour commencer, on observe qu'avec les notations du paragraphe~\ref{paragr:def_section} pour la base de $\cn \Deltan{m}\amalg_{\cn \Deltan{0}}\cn \Deltan{1+n'}$ et en tenant compte de l'inclusion $\cn \Deltan{m}\amalg_{\cn \Deltan{0}}\cn \Deltan{1+n'}\hookrightarrow \cn \Deltan{m+1+n'}$, pour tout $p\geq0$, et tout élément $(i_0,\dots,i_p)$ de la base de $\cn \Deltan{m}\amalg_{\cn \Deltan{0}}\cn \Deltan{1+n'}$, on a
\[
(\cn \Deltan{m}\amalg_{\cn \Deltan{0}}\cn (\psi''))(i_0,\dots,i_p)=\cn (\psi')(i_0,\dots,i_p)=(\psi'(i_0),\dots,\psi'(i_p))\,.
\]
Soit maintenant $(i_0,\dots,i_p)$ un élément de la base de $\cn \Deltan{m+1+n'}$; montrons que
\[
f_n\cn (\psi')(i_0,\dots,i_p)=(\cn \Deltan{m}\amalg_{\cn \Deltan{0}}\cn (\psi''))f_{n'}(i_0,\dots,i_p)\,.
\]
On distingue plusieurs cas en suivant la définition de $f_n$ dans le paragraphe~\ref{paragr:def_section} et en tenant compte de l'observation qui suit cette définition:
\begin{itemize}
\item[(1)] $i_p\leq m$ ou $m\leq i_0$: alors on a aussi $\psi'(i_p)\leq m$ ou $m\leq\psi'(i_0)$, et l'égalité ci-dessus est évidente.
\smallbreak
\item[(2)] $p=1$ et $i_0<m<i_1$: alors on a aussi $\psi'(i_0)<m<\psi'(i_1)$ et
\[
f_n\cn (\psi')(i_0,i_1)=f_n(\psi'(i_0),\psi'(i_1))=(\psi'(i_0),m)+(m,\psi'(i_1))\,,\vrule depth 4pt width0pt\kern 6pt
\]
\[\begin{aligned}
\kern 6pt(\cn \Deltan{m}\amalg_{\cn \Deltan{0}}\cn (\psi''))f_{n'}(i_0,i_1)&=(\cn \Deltan{m}\amalg_{\cn \Deltan{0}}\cn (\psi''))((i_0,m)+(m,i_1))\cr
&=\cn (\psi')((i_0,m)+(m,i_1))\cr
&=(\psi'(i_0),\psi'(m))+(\psi'(m),\psi'(i_1))\cr
&=(\psi'(i_0),m)+(m,\psi'(i_1))\,.
\end{aligned}\]
\item[(3)] $p>1$ et $i_0<m<i_1$: alors on a aussi $\psi'(i_0)<m<\psi'(i_1)$ et
\[
f_n\cn (\psi')(i_0,\dots,i_p)=f_n(\psi'(i_0),\dots,\psi'(i_p))=(m,\psi'(i_1),\dots,\psi'(i_p))\,,
\]
\[\begin{aligned}
(\cn \Deltan{m}\amalg_{\cn \Deltan{0}}\cn (\psi''))f_{n'}(i_0,\dots,i_p)&=(\cn \Deltan{m}\amalg_{\cn \Deltan{0}}\cn (\psi''))(m,i_1,\dots,i_p)\cr
&=\cn (\psi')(m,i_1,\dots,i_p)\cr
&=(\psi'(m),\psi'(i_1),\dots,\psi'(i_p))\cr
&=(m,\psi'(i_1),\dots,\psi'(i_p))\,.
\end{aligned}\]
\item[$(3')$] $p>1$ et $i_{p-1}<m<i_p$: le calcul est parfaitement analogue au précédent.
\smallbreak
\item[(4)] $i_1\leq m\leq i_{p-1}$: alors on a $f_{n'}(i_0,\dots,i_p)=0$ et $\psi'(i_1)\leq m\leq \psi'(i_{p-1})$, d'où
\[
f_n\cn (\psi')(i_0,\dots,i_p)=f_n(\psi'(i_0),\dots,\psi'(i_p))=0\,.
\]
\end{itemize}
Le fait que $s$ est une section du morphisme $r$ résulte aussitôt du cas (1) de la définition de $f_n$.
\end{paragr}

\begin{paragr}\label{paragr:def_hmtp}
Dans ce qui suit, on va définir une homotopie du morphisme composé 
\[
sr:\cotr{\Nrf A}{b}\to\cotr{\Nrf A}{b}
\]
vers le morphisme identité de $\cotr{\Nrf A}{b}$. On vérifie facilement que, pour $n\geq0$, ce morphisme associe à un $n$-simplexe $(y,x)$ de $\cotr{\Nrf A}{b}$, le $n$-simplexe $(y\nu(f_n),x)$, où l'on note aussi $f_n:\cn \Deltan{m+1+n}\to \cn \Deltan{m+1+n}$ le composé de $f_n:\cn \Deltan{m+1+n}\to \cn \Deltan{m}\amalg_{\cn \Deltan{0}}\cn \Deltan{1+n}$ avec l'inclusion canonique $\cn \Deltan{m}\amalg_{\cn \Deltan{0}}\cn \Deltan{1+n}\hookrightarrow \cn \Deltan{m+1+n}$. On définit une homotopie \hbox{$h:\Deltan{1}\times\cotr{\Nrf A}{b}\to\cotr{\Nrf A}{b}$} comme suit. Soient $n\geq0$ et $(\varphi,y,x)$,
\[
\varphi:\Deltan{n}\to\Deltan{1}\,,\ y:\On{m+1+n}\to B\,,\ x:\On{n}\to A\,,\ \ y_{0,\dots,m}=b\,,\ y_{m+1,\dots,m+1+n}=u(x)\,,
\]
un $n$\nbd-simplexe de $\Deltan{1}\times\cotr{\Nrf A}{b}$. On va définir un morphisme de complexes dirigés augmentés $f_\varphi:\cn \Deltan{m+1+n}\to \cn \Deltan{m+1+n}$, et on posera
\[
h(\varphi,y,x)=(y\nu(f_\varphi),x)\,.
\]
On définit $\bar\varphi:\Deltan{m+1+n}\to\Deltan{1}$ en posant
\[
\bar\varphi(i)=\left\{
\begin{matrix}
&0\,,\hfill&\quad0\leq i\leq m\,,\hfill\cr
\noalign{\vskip 3pt}
&\varphi(i-m-1)\,,&\quad m+1\leq i\leq m+1+n\,.
\end{matrix}
\right.
\]
Pour tout élément $(i_0,\dots,i_p)$, $p\geq0$, de la base de $\cn \Deltan{m+1+n}$, il existe un unique entier $k_\varphi$ tel que $-1\leq k_\varphi\leq p$, et tel que
\[
\bar\varphi(i_k)=\left\{
\begin{matrix}
&0\,,&\quad0\leq k\leq k_\varphi\,,\hfill\cr
&1\,,&\quad k_\varphi+1\leq k\leq p\,.
\end{matrix}
\right.
\]
On remarque que si $k_\varphi<p$, alors $i_{k_{\varphi}+1}>m$. On définit $f_\varphi$ en posant
\[
f_\varphi(i_0,\dots,i_p)=f_n(i_0,\dots,i_{k_\varphi})(i_{k_\varphi+1},\dots,i_p)\,,
\]
ce qui signifie que:
\begin{itemize}
\item si $0\leq k_\varphi<p$, et si $f_n(i_0,\dots,i_{k_\varphi})$ est une somme d'éléments de la base de $\cn \Deltan{m+1+n}$, alors $f_\varphi(i_0,\dots,i_p)$ est la somme dont les termes sont obtenus en concaténant à droite $(i_{k_\varphi+1},\dots,i_p)$ à chacun des termes de la somme $f_n(i_0,\dots,i_{k_\varphi})$ (et en particulier, si $f_n(i_0,\dots,i_{k_\varphi})=0$, alors on a aussi $f_\varphi(i_0,\dots,i_p)=0$);
\item si $k_\varphi=-1$, alors $f_\varphi(i_0,\dots,i_p)=(i_0,\dots,i_p)$;
\item si $k_\varphi=p$, alors $f_\varphi(i_0,\dots,i_p)=f_n(i_0,\dots,i_p)$.
\end{itemize}
\'Etant donné des entiers $0\leq i_0\leq\cdots\leq i_p\leq n$, s'il existe $k$ tel que $0\leq k<p\,$ et tel que  $i_k=i_{k+1}$, la définition de $k_\varphi$ garde un sens, et la formule définissant $f_\varphi$ est compatible avec la convention $(i_0,\dots,i_p)=0$. En effet, puisque alors $\bar\varphi(i_k)=\bar\varphi(i_{k+1})$, on a ou bien $k+1\leq k_\varphi$ ou bien $k_\varphi+1\leq k$, et l'assertion résulte de la propriété analogue des formules définissant $f_n$ (voir le paragraphe~\ref{paragr:def_section}).

On doit vérifier que $f_\varphi$ est un morphisme de complexes dirigés augmen\-tés, que pour tout $n$\nbd-simplexe $(\varphi,y,x)$ de $\Deltan{1}\times\cotr{\Nrf A}{b}$, le couple $(y\nu(f_\varphi),x)$ est un $n$\nbd-simplexe de $\cotr{\Nrf A}{b}$, et que $h$ est bien une homotopie simpliciale de $sr$ vers $1_{\cotr{\Nrf A}{b}}$.
\end{paragr}

\begin{paragr}\label{paragr:comp_diff_hmtp}
Montrons pour commencer que $f_\varphi$ est un morphisme de complexes dirigés augmentés. La compatibilité à l'augmentation et aux sous-monoïdes de positivité est évidente. Vérifions la compatibilité aux différentielles. Soient $p\geq0$ et $(i_0,\dots,i_p)$ un élément de la base de $\cn \Deltan{m+1+n}$; on doit prouver que
\[
df_\varphi(i_0,\dots,i_p)=f_\varphi d(i_0,\dots,i_p)\,.
\]
On a
{
\allowdisplaybreaks
\begin{align*}
f_\varphi
d(i_0,\dots,i_p)&=f_\varphi\Bigl(\textstyle\sum\limits_{k=0}^p(-1)^k(i_0,\dots,\hat
i_k,\dots,i_p)\Bigr) \\
&=\textstyle\sum\limits_{k=0}^{k_\varphi}(-1)^kf_\varphi(i_0,\dots,\hat
i_k,\dots,i_{k_\varphi},i_{k_\varphi+1},\dots,i_p) \\*
\noalign{\vskip -7pt}
&\kern25pt+\textstyle\sum\limits_{k=k_\varphi+1}^p(-1)^kf_\varphi(i_0,\dots,i_{k_\varphi},i_{k_\varphi+1},\dots,\hat
i_k,\dots,i_p)\\
&=\textstyle\sum\limits_{k=0}^{k_\varphi}(-1)^kf_n(i_0,\dots,\hat
i_k,\dots,i_{k_\varphi})(i_{k_\varphi+1},\dots,i_p)\\*
\noalign{\vskip -7pt}
&\kern25pt+(-1)^{k_\varphi+1}\textstyle\sum\limits_{k=0}^{p-k_\varphi-1}(-1)^kf_n(i_0,\dots,i_{k_\varphi})(i_{k_\varphi+1},\dots,\hat
i_{k_\varphi+1+k},\dots,i_p)\\
\noalign{\vskip 6pt}
&=f_nd(i_0,\dots,i_{k_\varphi})(i_{k_\varphi+1},\dots,i_p)\\*
\noalign{\vskip 2pt}
&\kern25pt+(-1)^{k_\varphi+1}f_n(i_0,\dots,i_{k_\varphi})d(i_{k_\varphi+1},\dots,i_p)\\
\noalign{\vskip 6pt}
&=df_n(i_0,\dots,i_{k_\varphi})(i_{k_\varphi+1},\dots,i_p)\\*
\noalign{\vskip 2pt}
&\kern25pt+(-1)^{k_\varphi+1}f_n(i_0,\dots,i_{k_\varphi})d(i_{k_\varphi+1},\dots,i_p)=df_\varphi(i_0,\dots,i_p)\,,
\end{align*}
}%
l'avant-dernière égalité résultant de la compatibilité de $f_n$ aux différentielles, prouvée dans le paragraphe~\ref{paragr:comp_diff}.
\end{paragr}

\begin{paragr}
Le fait que pour tout $n$\nbd-simplexe $(\varphi,y,x)$ de $\Deltan{1}\times\cotr{\Nrf A}{b}$, le couple $(y\nu(f_\varphi),x)$ est un $n$\nbd-simplexe de $\cotr{\Nrf A}{b}$ résulte de l'observation qu'en vertu du cas (1) de la définition de $f_n$ dans le paragraphe~\ref{paragr:def_section}, pour tout $p\geq0$, et tout $(i_0,\dots,i_p)$ dans la base de $\cn \Deltan{m+1+n}$, si $i_p\leq m$ ou si $m+1\leq i_0$, on a
\[
f_\varphi(i_0,\dots,i_p)=f_n(i_0,\dots,i_{k_\varphi})(i_{k_\varphi+1},\dots,i_p)=(i_0,\dots,i_p)\,.
\]
On définit donc bien une application 
\[
h_n:(\Deltan{1}\times\cotr{\Nrf A}{b})_n\to(\cotr{\Nrf A}{b})_n\,,\qquad(\varphi,y,x)\mapsto(y\nu(f_\varphi),x)\,.
\]
\end{paragr}

\begin{paragr}
On va montrer maintenant que les applications $h_n$, $n\geq0$, ainsi définies, forment un morphisme d'ensembles simpliciaux $h:\Deltan{1}\times\cotr{\Nrf A}{b}\to\cotr{\Nrf A}{b}$, autrement dit que pour tout $\psi:\Deltan{n'}\to\Deltan{n}$, le carré suivant est commutatif
\[
\xymatrix{
&(\Deltan{1}\times\cotr{\Nrf A}{b})_n\ar[r]^-{h_n}\ar[d]_{(\Deltan{1}\times\cotr{\Nrf A}{b})_\psi}
&(\cotr{\Nrf A}{b})_n\ar[d]^{(\cotr{\Nrf A}{b})_\psi}
\\
&(\Deltan{1}\times\cotr{\Nrf A}{b})_{n'}\ar[r]_-{h_{n'}}
&(\cotr{\Nrf A}{b})_{n'}
&\kern-20pt.\kern20pt
}
\]
Or, pour tout $n$\nbd-simplexe $(\varphi,y,x)$ de $\Deltan{1}\times\cotr{\Nrf A}{b}$, on a
\[\begin{aligned}
&\ h_{n'}(\Deltan{1}\times\cotr{\Nrf A}{b})_\psi(\varphi,y,x)=h_{n'}(\varphi\psi,y\nu \cn (\psi'),x\nu\cn(\psi))=(y\nu \cn (\psi')\nu(f_{\varphi\psi}),x\nu\cn(\psi))\,,\cr
\noalign{\vskip 3pt}
&\ (\cotr{\Nrf A}{b})_\psi h_n(\varphi,y,x)=(\cotr{\Nrf A}{b})_\psi(y\nu(f_\varphi),x)=(y\nu(f_\varphi)\nu \cn (\psi'),x\nu\cn(\psi))\,,
\end{aligned}\]
où, dans les notations de~\ref{notation:simili_joint}, $\psi'=\Deltan{m}\amalg\psi:\Deltan{m+1+n'}\to\Deltan{m+1+n}$.
Il suffit donc de prouver que pour tout $p\geq0$, et tout élément $(i_0,\dots,i_p)$ de la base de $\Deltan{m+1+n'}$, on a
\[
\cn (\psi')f_{\varphi\psi}(i_0,\dots,i_p)=f_\varphi\,\cn (\psi')(i_0,\dots,i_p)\,.
\]
Or, d'une part, on a
\[\begin{aligned}
\cn (\psi')f_{\varphi\psi}(i_0,\dots,i_p)&=\cn (\psi')(f_{n'}(i_0,\dots,i_{k_{\varphi\psi}})(i_{k_{\varphi\psi}+1},\dots,i_p))\cr
\noalign{\vskip 3pt}
&=\cn (\psi')f_{n'}(i_0,\dots,i_{k_{\varphi\psi}})\,\cn (\psi')(i_{k_{\varphi\psi}+1},\dots,i_p)\cr
\noalign{\vskip 3pt}
&=f_{n}\cn (\psi')(i_0,\dots,i_{k_{\varphi\psi}})\,\cn (\psi')(i_{k_{\varphi\psi}+1},\dots,i_p)\cr
\noalign{\vskip 3pt}
&=f_{n}(\psi'(i_0),\dots,\psi'(i_{k_{\varphi\psi}}))(\psi'(i_{k_{\varphi\psi}+1}),\dots,\psi'(i_p))\,,
\end{aligned}\]
(l'avant-dernière égalité résultant du paragraphe~\ref{paragr:comp_smpl}), et d'autre part, on vérifie aussitôt que $\overline{\varphi\psi}=\bar\varphi\psi'$, d'où par définition de $k_{\varphi\psi}$, on a
\[
\bar\varphi\psi'(i_0)=\cdots=\bar\varphi\psi'(i_{k_{\varphi\psi}})=0\qquad\hbox{et}\qquad\bar\varphi\psi'(i_{k_{\varphi\psi}+1})=\cdots=\bar\varphi\psi'(i_p)=1\,,
\]
et par suite,
\[\begin{aligned}
f_\varphi\,\cn (\psi')(i_0,\dots,i_p)&=f_\varphi(\psi'(i_0),\dots,\psi'(i_p))\cr
\noalign{\vskip 3pt}
&=f_{n}(\psi'(i_0),\dots,\psi'(i_{k_{\varphi\psi}}))(\psi'(i_{k_{\varphi\psi}+1}),\dots,\psi'(i_p))\,.
\end{aligned}\]
Ces égalités étant valables, en vertu des commentaires qui suivent la définition de $f_\varphi$ dans le paragraphe~\ref{paragr:def_hmtp}, même s'il existe $k$ tel que $0\leq k<p\,$ et tel que \hbox{$\psi'(i_k)=\psi'(i_{k+1})$}, ceci achève la preuve de l'assertion.
\end{paragr}

\begin{paragr}
Pour tous $p,n\geq0$, et tout élément $(i_0,\dots,i_p)$ de la base de $\Deltan{m+1+n}$, si \hbox{$\varphi:\Deltan{n}\to\Deltan{1}$} est constant de valeur $0$, alors $k_\varphi=p$ et $f_\varphi(i_0,\dots,i_p)=f_n(i_0,\dots,i_p)$, et si $\varphi$ est constant de valeur $1$, alors $k_\varphi=-1$ et $f_\varphi(i_0,\dots,i_p)=(i_0,\dots,i_p)$, ce qui implique que le morphisme $h:\Deltan{1}\times\cotr{\Nrf A}{b}\to\cotr{\Nrf A}{b}$ est une homotopie simpliciale de $sr$ vers $1_{\cotr{\Nrf A}{b}}$.
\end{paragr}

\begin{paragr}
Il reste à montrer que $\cotr{\Nrf A}{b_m}$ est un rétracte par déformation fort de $\cotr{\Nrf A}{b}$, autrement dit que le carré
\[
\xymatrix{
&\Deltan{1}\times\cotr{\Nrf A}{b_m}\ar[r]^-{pr_2}\ar[d]_{1_{\Deltan{1}}\times s}
&\cotr{\Nrf A}{b_m}\ar[d]^s
\\
&\Deltan{1}\times\cotr{\Nrf A}{b}\ar[r]_-h
&\cotr{\Nrf A}{b}
&\kern -20pt,\kern20pt
}
\]
où $pr_2$ désigne la deuxième projection, est commutatif. Or, pour tout $n\geq0$, et tout $n$\nbd-simplexe $(\varphi,y',x)$ de $\Deltan{1}\times\cotr{\Nrf A}{b_m}$, on a
\[
\begin{aligned}
&s\,pr_2(\varphi,y',x)=s(y',x)=((b,y')\nu(f_n),x)\,,\cr
\noalign{\vskip 3pt}
&h(1_{\Deltan{1}}\times s)(\varphi,y',x)=h(\varphi,(b,y')\nu(f_n),x)=((b,y')\nu(f_n)\nu(f_\varphi),x)\,.
\end{aligned}
\]
Il suffit donc de montrer que pour tout $p\geq0$, et tout $(i_0,\dots,i_p)$ dans la base de $\cn\Deltan{m+1+n}$, on a $f_nf_\varphi(i_0,\dots,i_p)=f_n(i_0,\dots,i_p)$, autrement dit en vertu de la définition de $f_\varphi$, que
\[
f_n\bigl(f_n(i_0,\dots,i_{k_\varphi})(i_{k_\varphi+1},\dots,i_p)\bigr)=f_n(i_0,\dots,i_p)\,.
\]
On remarque que si $k_\varphi=-1$, cette égalité est évidente, et que si $k_\varphi=p$, elle résulte de la relation $f_nf_n=f_n$, qui est conséquence immédiate de la définition de $f_n$ dans le paragraphe~\ref{paragr:def_section}. On peut donc supposer que $0\leq k_\varphi<p$.
On distingue plusieurs cas, suivant la définition de $f_n$:
\begin{itemize}
\item[(1)] $i_{k_\varphi}\leq m$ ou $i_0\geq m$: alors l'égalité est évidente. On peut donc supposer dans ce qui suit que $k_\varphi>0$ et $i_{k_\varphi}>m$.
\item[(2)] $k_\varphi=1$ et $i_0<m<i_1=i_{k_\varphi}$: alors, vu que $p>k_\varphi=1$, on a
\[\begin{aligned}
f_n\bigl(f_n(i_0,i_1)(i_2,\dots,i_p)\bigr)&=f_n(i_0,m,i_2,\dots,i_p)+f_n(m,i_1,i_2,\dots,i_p)\cr
&=0+(m,i_1,i_2,\dots,i_p)=f_n(i_0,\dots,i_p)\,.
\end{aligned}\]
\item[(3)] $k_\varphi>1$ et $i_0<m<i_1$: alors on a
\[\begin{aligned}
f_n\bigl(f_n(i_0,\dots,i_{k_\varphi})(i_{k_\varphi+1},\dots,i_p)\bigr)&=f_n(m,i_1,\dots,i_{k_\varphi},i_{k_\varphi+1},\dots,i_p)\cr
&=(m,i_1,\dots,i_p)=f_n(i_0,\dots,i_p)\,.
\end{aligned}\]
\item[$(3')$] $k_\varphi>1$ et $i_{k_\varphi-1}<m<i_{k_\varphi}$: alors, vu que $p>k_\varphi$, on a
\[\begin{aligned}
f_n\bigl(f_n(i_0,\dots,i_{k_\varphi})(i_{k_\varphi+1},\dots,i_p)\bigr)&=f_n(i_0,\dots,i_{k_\varphi-1},m,i_{k_\varphi+1},\dots,i_p)\cr
&=0=f_n(i_0,\dots,i_p)\,.
\end{aligned}\]
\item[(4)] $i_1\leq m\leq i_{k_\varphi-1}$: alors on a
\[
f_n\bigl(f_n(i_0,\dots,i_{k_\varphi})(i_{k_\varphi+1},\dots,i_p)\bigr)=0=f_n(i_0,\dots,i_p)\,.
\]
\end{itemize}
\smallbreak

On a donc établi le théorème suivant.
\end{paragr}

\begin{thm}\label{thm:hmtpsmpl}
Soient $A$ et $B$ deux \oo-catégories, $u:\Nrf A\to \Nrf B$ un morphisme d'ensembles simpliciaux, $m\geq0$ un entier, et \hbox{$b:\On{m}\to B$} un $m$\nbd-simplexe de $\Nrf B$. Alors le morphisme d'ensembles simpliciaux 
\[
r:\cotr{\Nrf A}{b}\to\cotr{\Nrf A}{b_m}\,,
\]
associant à un $n$\nbd-simplexe $(y,x)$ de $\cotr{\Nrf A}{b}$, $n\geq0$,
\[
x:\On{n}\to A\,,\ y:\On{m+1+n}\to B\quad \hbox{tels que}\quad y_{0,\dots,m}=b\,,\ y_{m+1,\dots,m+1+n}=u(x)\,,
\]
le $n$\nbd-simplexe $(y_{m,m+1,\dots,m+1+n},x)$ de $\cotr{\Nrf A}{b_m}$, admet une section $s$ faisant de $\cotr{\Nrf A}{b_m}$ un rétracte par déformation fort de $\cotr{\Nrf A}{b}$. En particulier, $r$ est une équivalence faible simpliciale.
\end{thm}

\vspace{-2pt}

Le cas particulier de ce théorème, pour $u$ le nerf de Street d'un \oo-foncteur \hbox{$A\to B$}, achève la démonstration du théorème~A \oo-catégorique de la section précédente.

\backmatter

{
\vspace{-10pt}
\bibliography{Catmod}
\bibliographystyle{mysmfplain}
}

\end{document}